\documentclass{amsart}
\newtheorem{theorem}{Theorem}[section]
\newtheorem{lemma}[theorem]{Lemma}
\newtheorem{corollary}[theorem]{Corollary}

\theoremstyle{definition}
\newtheorem{definition}[theorem]{Definition}

\theoremstyle{remark}
\newtheorem{remark}[theorem]{Remark}

\numberwithin{equation}{section}
\usepackage{enumerate}

\usepackage{amsmath,graphicx,adjustbox}
\usepackage{amssymb}
\usepackage{euscript}
\usepackage{pstricks}
\usepackage{hyperref}
\usepackage{xcolor,cleveref}
\usepackage[linesnumbered,ruled,vlined]{algorithm2e}
\usepackage{soul}
\usepackage{mathabx}
\usepackage{enumitem}
\usepackage{pgfplots}
\usepackage{algpseudocode}
\usepackage{color}

\hypersetup
{
	colorlinks,
	citecolor=black,
	filecolor=black,
	linkcolor=black,
	urlcolor=magenta
}
\hypersetup{linktocpage} %
\numberwithin{equation}{section}

\begin{document}

\title[Gevrey regularity and QMC quadrature for parametric PDEs]{Local Gevrey regularity and Quasi-Monte Carlo quadrature for  PDEs parameterized  on non-compact  domains}
\author[D. D\~ung]{Dinh D\~ung}
\address{Information Technology Institute, Vietnam National University, Hanoi}
\curraddr{144 Xuan Thuy, Cau Giay, Hanoi, Vietnam}
\email{dinhzung@gmail.com}
\thanks{The work of Dinh D\~ung is funded by the Vietnam National Foundation for Science and Technology Development (NAFOSTED) under the Vietnamese-Swiss Joint Research Project, Grant No.
	IZVSZ2$_{ - }$229568. A part of this work was done when Dinh D\~ung was working at the Vietnam Institute for Advanced Study in Mathematics (VIASM). He would like to thank the VIASM for providing a fruitful research environment and working condition.}

\author[T. L\^e]{T\`ung L\^e}
\address{Institut für Mathematik, Carl von Ossietzky Universität Oldenburg}
\email{tung.le@uni-oldenburg.de}
 \thanks{
	The second author thanks the Hausdorff Research Institute for Mathematics (Bonn) for the hospitality and support during the Trimester Program \textit{Computational multifidelity, multilevel, and multiscale methods}, funded by the Deutsche Forschungsgemeinschaft (DFG, German Research Foundation) under Germany‘s Excellence Strategy – EXC-2047/2 – 390685813.}

 	\author[A. Chernov]{Alexey Chernov}
 	\address{Institut für Mathematik, Carl von Ossietzky Universität Oldenburg}
 	\email{alexey.chernov@uni-oldenburg.de}
 	\thanks{ The simulations were conducted on the HPC cluster ROSA, located at the University of Oldenburg (Germany). ROSA was funded by the German Research Foundation (DFG) through its Major Research Instrumentation Programme (INST 184/225-1 FUGG) and the Ministry of Science and Culture (MWK) of Lower Saxony.}



\date{\today}

\dedicatory{}

\begin{abstract}
	We establish local Gevrey regularity for the weak solution to parametric divergence-form diffusion elliptic PDEs, assuming the diffusion coefficient itself possesses local Gevrey{ parametric regularity over a non-compact domain}. Here ``local Gevrey regularity'' means that the regularity is determined in a neighborhood of each parametric point and depends on that point. Explicit bounds on parametric derivatives of the weak solution are proved in the $H^1$-norm.
	Building on this local Gevrey framework, we develop a novel theoretical treatment of the dimension truncation error for infinite-dimensional integration. We prove convergence rates of quasi-Monte Carlo quadrature for the finite-dimensional Gaussian-weighted integration of integrands having local Gevrey regularity. Together with the finite element discretization error, the errors arising from dimension truncation and quasi-Monte Carlo quadrature yield a complete error analysis {and convergence rates} for the fully discrete approximation of a bounded linear functional of the weak solution.
	Numerical experiments confirming the theoretical convergence rates are presented.
\end{abstract}

\maketitle

\section{Introduction}

{%

In recent years, computational uncertainty quantification (UQ) has witnessed substantial advances in the efficient approximation of infinite-dimensional parametric partial differential equations (PDEs) driven by random inputs. 
For a detailed overview together with an extensive bibliography, we refer the reader to {\cite{ABDM2023,CoDe15a,DNSZ2023}.}
The high computational cost inherent to such problems has motivated the design of more efficient numerical techniques, including approaches based on sparse grids 
{\cite{BCDC17,BCDM17,BCM17,CCS15,Dung19,Dung21,DHNS2026,DNSZ2023,ZDS19,ZS20}}  as well as methods employing least-squares approximation \cite{BD2024,CCMNT2015,CM2018,DD2026a,DD2026,DHV2026} or quasi-Monte Carlo (QMC) strategies \cite{DHNS2026,Gilbert2019,Graham15,KuNu16,KuoSchwabSloan2012}. Approaches based on sparse grids and least-squares approximation typically rely on a generalized polynomial chaos  expansion of the parametric solution map and, in particular, on sparsity of the resulting coefficients. In contrast, QMC methods require substantially less regularity: namely, suitable weighted bounds on parametric mixed partial derivatives of a prescribed order. This weaker set of assumptions permits extending the QMC approach to a broader class of parametric PDEs, notably those whose solution maps exhibit (so-called) Gevrey regularity, studied in \cite{CL2024,CL2024b,DHNS2026,GuthKaarnioja2025,HSS2024}.

  Consider an important model parametric PDE, the divergence-form diffusion elliptic equation 
\begin{equation} \label{ellip}
	\begin{cases}
		-\operatorname{div} \left( a({\boldsymbol{y}})({\boldsymbol{x}}) \nabla u({\boldsymbol{y}})({\boldsymbol{x}}) \right) = {f({\boldsymbol{x}})} & \text{in } D, \\
		\hspace{2.8cm} u({\boldsymbol{y}})({\boldsymbol{x}}) = 0 & \text{on } \partial D,
	\end{cases}
\end{equation}
where $D \subset \mathbb{R}^d$ is a bounded Lipschitz domain, 
${f }\in H^{-1}(D)$ and 
 $a: \, D \times \Omega \to {\mathbb R}$ is a parametric diffusion coefficient. The parametric domain $\Omega$ is assumed to be a non-empty subset of ${\mathbb R}^{\mathbb N}$. Throughout this paper, we assume these conditions and the below condition \eqref{PDE-ellipticity} on uniform ellipticity hold without further comment.

Assume that for every ${\boldsymbol{y}} \in \Omega$, the function $Z({\boldsymbol{y}})$ is such that 
$a({\boldsymbol{y}})$ is strictly positive and bounded in $D$:
\begin{align}\label{PDE-ellipticity}
	0 < a_{\min}({\boldsymbol{y}}) \leq a({\boldsymbol{y}})(\boldsymbol{x}) \leq a_{\max}({\boldsymbol{y}}) < \infty, \ \ \boldsymbol{x} \in D.
\end{align}
Then by the Lax--Milgram lemma,  there exists a unique weak solution $u({\boldsymbol{y}}) \in V := H^1_0(D)$ satisfying the variational formulation:
\begin{equation} \label{weak-form}
	\int_{D} a({\boldsymbol{y}})({\boldsymbol{x}}) \nabla u({\boldsymbol{y}})({\boldsymbol{x}}) \cdot \nabla v({\boldsymbol{x}}) \, {\rm d} {\boldsymbol{x}} = \langle f, v \rangle, \quad \forall v \in V.
\end{equation}
Here $\langle \cdot, \cdot\rangle $ is the duality pairing of $V$ and $V' = H^{-1}(D)$.

In the present paper, we focus our attention on the case when the parameter domain $\Omega={\mathbb R}^{\mathbb N}$ is non-compact. It is convenient in consideration to present 
$a({\boldsymbol{y}})({\boldsymbol{x}})$ in the form
\begin{equation} \label{def:diffus-coeff}
	a({\boldsymbol{y}})({\boldsymbol{x}}) := \exp\big( Z({\boldsymbol{y}})({\boldsymbol{x}})\big), \quad {\boldsymbol{x}} \in D,
\end{equation}
where $Z: \, D \times \Omega \to {\mathbb R}$. 

{One of the most important instances of \eqref{def:diffus-coeff} is the log-normal representation of the parametric diffusion coefficient 
$a({\boldsymbol{y}})$, which is based on the Karhunen--Lo\`eve decomposition of $Z({\boldsymbol{y}})$. In this setting, the function  $Z({\boldsymbol{y}})$ takes the form
\begin{equation} \label{KL-decomp}
	Z({\boldsymbol{y}}) = \sum_{j \in {\mathbb N}} y_j \psi_j,
\end{equation}
where $(\psi_j)_{j \in {\mathbb N}}$ is a sequence of  functions in $L_\infty(D)$, and $(y_j)_{j \in {\mathbb N}}$ is
a sequence of i.i.d. normal random variables. For further details, see, e.g., \cite{BCDM17}. }

Throughout this paper, we explicitly distinguish between spatial and parametric derivatives. We employ the notation $\mathsf{D}^{\boldsymbol{\alpha}}$ to denote the partial derivative with respect to the spatial variables ${\boldsymbol{x}} \in D$, {where ${\boldsymbol{\alpha}} \in {\mathbb N}^d_0$ is a spatial multi-index. Conversely, the symbol $\partial^{\boldsymbol{\nu}}$  is strictly reserved for  partial  derivatives with respect to the parameter variables ${\boldsymbol{y}}$, where  ${\boldsymbol{\nu}} \in {\mathcal F}$ is a parametric multi-index.}

A fundamental challenge in the numerical analysis of  the parametric PDE \eqref{ellip} is that the diffusion coefficient is not uniformly elliptic with respect to the parameters ${\boldsymbol{y}}$. In the particular setting where $Z({\boldsymbol{y}})$ is restricted to an affine-parametric combination (log-normal case), this issue has been extensively studied and successfully resolved in the literature.  The resolution relies on the real analyticity — or complex-analytic extension in the parametric variables — of the weak solution, which is established through parametric derivative bounds on the diffusion coefficient
\begin{equation*} 
	\left\|{\partial^{\boldsymbol{\nu}} a({\boldsymbol{y}})}\right\|_{L^\infty(D)}
	\le C\boldsymbol{{\boldsymbol{\beta}}}^{\boldsymbol{\nu}}\,|\boldsymbol{\nu}|! \exp\big(\left\|{Z({\boldsymbol{y}})}\right\|_{L^\infty(D)}\big), \quad \forall {\boldsymbol{y}} \in {\mathbb R}^{\mathbb N}.
\end{equation*} 
See, e.g., \cite{BCDM17,DHNS2026,DHV2026,DNSZ2023,Graham15,HOANG2014}. It is worth  noting that this bound  depends point-wise on parameter ${\boldsymbol{y}}$.

Moving beyond this affine structure, the works \cite{CL2024,CL2024b,HSS2024} have introduced a novel framework of Gevrey parametric regularity for parametrizations of general form and the compact parameter domain $\Omega=[-1,1]^{\mathbb N}$. Guth and Kaarnioja have recently extended this setting in \cite{GuthKaarnioja2025} to the case $\Omega={\mathbb R}^{\mathbb N}$.

The analysis in \cite{GuthKaarnioja2025}  relies  on the assumption of global Gevrey regularity
\begin{equation} \label{norm-G-regularity}
	\left\|{\frac{\partial^{\boldsymbol{\nu}} a({\boldsymbol{y}})}{a({\boldsymbol{y}})}}\right\|_{L^\infty(D)}
	\le C\boldsymbol{{\boldsymbol{\beta}}}^{\boldsymbol{\nu}}\,(|\boldsymbol{\nu}|!)^\delta, \quad \forall {\boldsymbol{y}} \in {\mathbb R}^{\mathbb N},
\end{equation}
(\cite[Eq. (1.4)]{GuthKaarnioja2025}), which holds only in some particular cases such as the affine-parametric case of $Z({\boldsymbol{y}})$. Here, $\delta\ge 1$ and $C$ and ${\boldsymbol{\beta}}:=(\beta_j)_{j \in {\mathbb N}}$ are a constant and a sequence, respectively,  independent of ${\boldsymbol{y}}$.
While such conditions allow for generalizations beyond the log-normal case, they often necessitate structural constraints (i.e. the constant ${\boldsymbol{\beta}}$ being independent of~${\boldsymbol{y}}$) that may not easily accommodate the complex, non-linear dependencies of the diffusion coefficient on the parameters ${\boldsymbol{y}}$ that frequently arise in practice. For instance, a prototypical example of diffusion coefficient $a$ that fails to satisfy the global condition \eqref{norm-G-regularity} is given by
	\begin{equation*}
		a ({\boldsymbol{y}})({\boldsymbol{x}}) = 
		\exp\left(\sum_{j \in \mathbb{N}} \psi_j({\boldsymbol{x}})
			\cosh({\mathsf{b}}_j y_j)\right),
	\end{equation*}
	This example highlights the limitations of \eqref{norm-G-regularity} on unbounded domains: the rapid growth of the $\cosh$ function as $|y_j| \to \infty$ strictly precludes the existence of a parameter-independent sequence ${\boldsymbol{\beta}}$.

The aim of this work is  to provide a significant generalization of the results with the QMC applications of \cite{GuthKaarnioja2025} by relaxing these structural constraints. Our approach is built upon the assumption that the diffusion coefficient $a({\boldsymbol{y}})$ satisfies the \emph{local} Gevrey $\delta$-regularity: $\forall {\boldsymbol{y}}\in\mathbb{R}^{\mathbb{N}}$ there exists a neighborhood $U_{{\boldsymbol{y}}}$ of ${\boldsymbol{y}}$ such that 
\begin{equation} \label{dev-a-bounds}
	\left\|{\frac{\partial^{\boldsymbol{\nu}} a(\boldsymbol{\xi})}{a(\boldsymbol{\xi})}}\right\|_{L^\infty(D)}
	\le \kappa({\boldsymbol{y}}) \boldsymbol{\beta}({\boldsymbol{y}})^{\boldsymbol{\nu}} (|\boldsymbol{\nu}|!)^{\delta}, \quad \forall \boldsymbol{\xi} \in U_{{\boldsymbol{y}}}, \quad \forall \boldsymbol{\nu}\in \mathcal{F}\setminus\{\boldsymbol{0}\},
\end{equation}	
or  the weaker \emph{pointwise} Gevrey $\delta$-regularity
\begin{equation} \label{dev-a-bounds-pointwise}
	\left\|{\frac{\partial^{\boldsymbol{\nu}} a({\boldsymbol{y}})}{a({\boldsymbol{y}})}}\right\|_{L^\infty(D)}
	\le \kappa({\boldsymbol{y}}) \boldsymbol{\beta}({\boldsymbol{y}})^{\boldsymbol{\nu}} (|\boldsymbol{\nu}|!)^{\delta}, \quad \forall {\boldsymbol{y}} \in {\mathbb R}^{\mathbb N}, \quad \forall \boldsymbol{\nu}\in \mathcal{F}\setminus\{\boldsymbol{0}\}.
\end{equation}	
{Here, $\delta\ge 1$, and   $\kappa$ and ${\boldsymbol{\beta}}:=(\beta_j)_{j \in {\mathbb N}}$ are a positive function and a sequence of positive functions on ${\mathbb R}^{\mathbb N}$.}

{As established in our key result,}  Theorems \ref{thm:G-regularity-u(by)}, we prove that the weak solution $u({\boldsymbol{y}})$ inherit the Gevrey regularity of  diffusion coefficient $a({\boldsymbol{y}})$. 
This property  is derived by using a novel \emph{alternative-factorial technique}, initiated in \cite{CL2024,CL2024b, CL2025}. Unlike the conditions in \cite{GuthKaarnioja2025}, our Gevrey regularity results allow the quantity {$\kappa$ and the sequence  ${\boldsymbol{\beta}}$} in the derivative bounds in  \eqref{dev-a-bounds} to depend on ${\boldsymbol{y}}$. This on one hand, allows to extend to a significantly wider class of Gevrey regularity, and on the other hand, provides a more realistic and flexible foundation for the subsequent analysis of dimension truncation and QMC integration errors.

The main contributions of this work are summarized as follows:
\begin{itemize} 
	\item We establish the local and pointwise Gevrey $\delta$-regularities for the weak solution $u({\boldsymbol{y}})$ to the parametric elliptic PDE \eqref{ellip} based on the local and pointwise Gevrey $\delta$-regularity of diffusion coefficient in terms of  \eqref{dev-a-bounds}, respectively. The bounds of parametric derivatives of $u({\boldsymbol{y}})$ are given in the norm of the space $V$.
	
	\item 
	Building on this local or pointwise Gevrey $\delta$-regularity framework, we develop a novel theoretical treatment of the dimension truncation error for infinite-dimensional weighted integration. Notably, in the setting where the derivative bounds are restricted to parameter-independent constants, our estimates seamlessly recover the recent results established in \cite{GuthKaarnioja2024}.
	
	\item We prove convergence rates of quasi-Monte Carlo quadrature for  the finite-dimensional Gaussian integration for integrands having a local Gevrey $\delta$-regularity and show that the integration error remains robust under the local or pointwise Gevrey $\delta$-regularity properties. 
	  
	  \item
	  Together with the finite element discretization error, the errors arising from dimension truncation and quasi-Monte Carlo quadrature yield a total error analysis {and convergence rates} for the fully discrete approximation of a bounded linear functional of the weak solution. 
\end{itemize}
 
The remainder of this paper is structured as follows. Section~\ref{sec:Definitions and preliminaries} recalls essential facts about the falling factorial technique and introduces the notions of local and pointwise Gevrey $\delta$-regularities. Building on the local Gevrey $\delta$-regularity of the diffusion coefficient, Section~\ref{sec:Gevrey regularity of the weak solution} establishes the local Gevrey $\delta$-regularity in the energy space for the weak solution. Section~\ref{sec: truncation} provides an analysis of the dimension truncation error, while Section~\ref{sec:Error Analysis for Quasi-Monte Carlo methods} proves the convergence rate of quasi-Monte Carlo quadrature for finite-dimensional integrals of functions exhibiting Gevrey $\delta$-regularity. After recalling relevant results on finite element approximation, Section~\ref{sec: total error} derives a bound on the total error for the fully discrete approximation of a bounded linear functional of the weak solution, combining the finite element discretization error with the errors arising from dimension truncation and quasi-Monte Carlo quadrature. Finally, Section~\ref{sec:Numerical experiment} presents numerical experiments that validate our theoretical findings.  %
}

\subsection*{{Notations}}
The set of finitely supported sequences of non-negative integers $\mathcal F$ is defined as
\begin{equation}\label{cF-def}
	\mathcal F := 
	\left\{
	\boldsymbol{\nu} ={(\nu_j)_{j \in {\mathbb N}}} : \nu_j\in {\mathbb{N}}_0 \text{ with } \nu_j \neq 0 \text{ for  finitely many } j
	\right\}.
\end{equation}
For multi-indices ${\boldsymbol{\alpha}}, \boldsymbol{\beta} \in \mathcal F$, the algebraic operation ${\boldsymbol{\alpha}} + \boldsymbol{\beta}$ and the partial orderings ${\boldsymbol{\alpha}} \leq \boldsymbol{\beta}$ are understood componentwise. Specially, ${\boldsymbol{\alpha}} < \boldsymbol{\beta}$ means ${\boldsymbol{\alpha}} \leq \boldsymbol{\beta}$ and ${\boldsymbol{\alpha}} \neq \boldsymbol{\beta}$. For a given sequence of positive real numbers ${\boldsymbol{R}}=(R_j)_{j{\in {\mathbb N}}}$ and a multi-index $\boldsymbol{\nu} \in \mathcal F$, it is convenient to define
\begin{align*}
	{\boldsymbol{R}}^{\boldsymbol{\nu}} := \prod_{j{\in {\mathbb N}}} R_j^{\nu_j},
	\qquad \qquad
	{\left|\boldsymbol{\nu}\right|} := \sum_{j{\in {\mathbb N}}} \nu_j, 
	\qquad \qquad
	\boldsymbol{\nu}! := \prod_{j{\in {\mathbb N}}} \nu_j!. 
\end{align*}

For a $\boldsymbol{\nu}\in \mathcal F$ with  $\operatorname{supp} (\boldsymbol{\nu}) \subset\left\{1,2,\dots, n\right\}$ and a function $v$ on ${\mathbb R}^{\mathbb N}$, the partial derivative $\partial^{\boldsymbol{\nu}} v$ with respect to the variables $\boldsymbol{y}$ are defined as
\begin{align*}
	\partial^{\boldsymbol{\nu}} v
	=
	\frac{\partial^{{\left|\boldsymbol{\nu}\right|}}v}
	{\partial y_1^{\nu_1} \partial y_2^{\nu_2} \dots \partial y_n^{\nu_n}}
\end{align*}
with the convention that $\partial y_j^{\nu_j}$ is omitted if $\nu_j =0$. 
For two multi-indices $\boldsymbol{\nu}, {\boldsymbol{\eta}} \in \mathcal F$, the binomial coefficient is given by
\begin{equation*}
	\left(\boldsymbol{\nu} \atop {\boldsymbol{\eta}}\right)  = 
	\prod_{j{\in {\mathbb N}}}  \left(\nu_j \atop \eta_j\right)
	=\frac{\boldsymbol{\nu}!}{{\boldsymbol{\eta}}!\,(\boldsymbol{\nu}-{\boldsymbol{\eta}})!
	}.
\end{equation*}

\section{{Gevrey regularities}} 
\label{sec:Definitions and preliminaries}

{This section introduces the concepts of local and pointwise Gevrey $\delta$-regularities, along with the necessary preliminaries on the falling factorial — an efficient tool for establishing these regularities for weak solutions to the parametric elliptic PDEs given by \eqref{ellip}.}

\subsection{{Gevrey regularity and real analyticity classes}}
	\label{Real analytic and Gevrey regularities}
	
	The following definition of  Gevrey  $\delta$-regularities for functions with countably many variables will serve as the foundation for our subsequent analysis. We note that this definition is equivalent to the definition of Gevrey regularity for a compact set as in \cite[Def. 1.4.1]{Rodino1993}.
	
	\begin{definition}[Local Gevrey {regularity}]\label{def:Gevrey-Class-Local}
Let $\Omega \subseteq \mathbb{R}^{\mathbb{N}}$ {be a domain} and let $X$ be a Banach space. A function 
$v: \Omega \to X$ belongs to the \emph{Gevrey class} {$G^{\delta}_{\operatorname{loc}}(\Omega,X)$} or equivalently, $v$ has the local Gevrey $\delta$-regularity if  $v$ is infinite differentiable on $\Omega$ and if for  every $\boldsymbol{y} \in \Omega$, there exist an open neighborhood $U_{{\boldsymbol{y}}} \subseteq \Omega$ of $\boldsymbol{y}$, a positive function $\kappa(\boldsymbol{y}) \ge 1$, and a sequence ${\boldsymbol{\beta}}(\boldsymbol{y}) = \left(\beta_j(\boldsymbol{y})\right)_{j \in \mathbb{N}}$ of positive functions in variable ${\boldsymbol{y}}$, such that 
	\begin{equation} \label{eq:G-regularity-local}
		\left\| \partial^{\boldsymbol{\nu}} v(\boldsymbol{\xi}) \right\|_X \leq \kappa(\boldsymbol{y})\, {\boldsymbol{\beta}}(\boldsymbol{y})^{{\boldsymbol{\nu}}} \left( |{\boldsymbol{\nu}}|! \right)^\delta, \quad \forall {\boldsymbol{\nu}} \in \mathcal{F}, \; \forall \boldsymbol{\xi} \in U_{{\boldsymbol{y}}}.
	\end{equation}
\end{definition}

\begin{remark}\label{rem:Gevrey-differences}
		When $\Omega$ is a compact domain, a standard finite subcover argument
		ensures that for any $v \in {G^{\delta}_{\operatorname{loc}}(\Omega,X)}$, the bounding functions $\kappa(\boldsymbol{y})$ and ${\boldsymbol{\beta}}(\boldsymbol{y})$ can be replaced by universal constants independent of $\boldsymbol{y}$. Thus, the bound \eqref{eq:G-regularity-local} holds uniformly over the entire domain $\Omega$. 
\end{remark}
\begin{remark}\label{rem:analyticity}
	In the special case where $\delta = 1$, the Gevrey class {$G^{\delta}_{\operatorname{loc}}(\Omega,X)$} coincides with the space of $X$-valued real analytic functions on $\Omega$. Consequently, any $v \in {G^{\delta}_{\operatorname{loc}}(\Omega,X)}$ admits a locally and absolutely convergent multivariate Taylor series. Specifically, for every $\boldsymbol{y} \in \Omega$, there exists a neighborhood $U_{{\boldsymbol{y}}}$ of $\boldsymbol{y}$ such that for all $\boldsymbol{\xi} \in U_{{\boldsymbol{y}}}$, $v$ can be represented as:
	\begin{equation} \label{eq:Taylor-analytic}
		v(\boldsymbol{\xi}) = \sum_{{\boldsymbol{\nu}} \in \mathcal{F}} \frac{\partial^{{\boldsymbol{\nu}}} v(\boldsymbol{y})}{{\boldsymbol{\nu}}!} (\boldsymbol{\xi}-\boldsymbol{y})^{{\boldsymbol{\nu}}},
	\end{equation}
	where the convergence holds in the norm of $X$.
\end{remark}

While Definition \ref{def:Gevrey-Class-Local} imposes a uniform bound over a local neighborhood $U_{{\boldsymbol{y}}}$ with fixed control parameters for {this} neighborhood, it is often sufficient and more natural in certain parametric analyses to enforce this growth condition strictly at the evaluation point. %
{This is understood in the sense of the following definition.}

\begin{definition}[Pointwise Gevrey {regularity}]\label{def:Gevrey-Class-Pointwise}
Let $\Omega \subseteq \mathbb{R}^{\mathbb{N}}$ {be a domain}  and let $X$ be a Banach space. A function 
$v: \Omega \to X$ belongs to the \emph{Gevrey class} {$G^{\delta}_{\operatorname{pw}}(\Omega,X)$} or equivalently, $v$ has the pointwise Gevrey $\delta$-regularity if and only if $v$ is infinite differentiable on $\Omega$ and if for every $\boldsymbol{y} \in \Omega$, there exist
a positive function $\kappa(\boldsymbol{y}) \ge 1$, and a sequence ${\boldsymbol{\beta}}(\boldsymbol{y}) = \left(\beta_j(\boldsymbol{y})\right)_{j \in \mathbb{N}}$ of positive functions in variable ${\boldsymbol{y}}$, such that
	\begin{equation} \label{eq:G-regularity-pointwise}
		\left\| \partial^{\boldsymbol{\nu}} {v}(\boldsymbol{y}) \right\|_X \leq \kappa(\boldsymbol{y})\, {\boldsymbol{\beta}}(\boldsymbol{y})^{{\boldsymbol{\nu}}} \left( |{\boldsymbol{\nu}}|! \right)^\delta, \quad \forall {\boldsymbol{\nu}} \in \mathcal{F}.
	\end{equation}
\end{definition}

	{ The parameter $\delta^{-1}$ in Definitions~ \ref{def:Gevrey-Class-Local} and \ref{def:Gevrey-Class-Pointwise}
serves as a measure of Gevrey regularity, placing the corresponding function classes on the spectrum between real analytic and infinitely differentiable regularities.}

It is important to emphasize the core distinction between the class of local $\delta$-regularity {$G^{\delta}_{\operatorname{loc}}(\Omega,X)$} and the class of pointwise $\delta$-regularity {$G^{\delta}_{\operatorname{pw}}(\Omega,X)$}. The local class requires the bound \eqref{eq:G-regularity-local} to hold uniformly over the entire open neighborhood $U_{{\boldsymbol{y}}}$ using fixed local bounding functions $\kappa(\boldsymbol{y})$ and ${\boldsymbol{\beta}}(\boldsymbol{y})$ anchored at $\boldsymbol{y}$. In contrast, the pointwise class only imposes the growth condition strictly at the evaluation point $\boldsymbol{y}$ itself. 

The local Gevrey $\delta$-regularity introduced in Definition \ref{def:Gevrey-Class-Local} provides {a broader} framework than the settings recently analyzed in \cite{CL2024,CL2024b,CL2025}, naturally yielding uniform bounds when $\Omega$ is compact. However, our subsequent analysis of dimension truncation and QMC error bounds does not require such local uniformity. Instead, the weaker pointwise Gevrey $\delta$-regularity introduced in Definition \ref{def:Gevrey-Class-Pointwise} is entirely sufficient for our purposes.

To rigorously formalize this structural hierarchy, we state the following lemma {to  clarify} the inclusion relationships between the {local and pointwise Gevrey $\delta$-regularities}.

\begin{lemma}\label{lem:Gevrey-inclusion}
	For any $\delta \geq 1$, the local Gevrey class is a proper subset of the pointwise Gevrey class, i.e., {$G^{\delta}_{\operatorname{loc}}(\Omega,X) 
	\subsetneq 
	G^{\delta}_{\operatorname{pw}}(\Omega,X)$}

\end{lemma}

\begin{proof}
	The inclusion 
	{$G^{\delta}_{\operatorname{loc}}(\Omega,X) 
		\subseteq 
		G^{\delta}_{\operatorname{pw}}(\Omega,X)$}
	is obvious {by the definitions.} To establish the strict inclusion 
	{$G^{\delta}_{\operatorname{loc}}(\Omega,X) 
		\subsetneq 
		G^{\delta}_{\operatorname{pw}}(\Omega,X)$}, we must demonstrate the existence of a function 
		{$v \in G^{\delta}_{\operatorname{pw}}(\Omega,X)$} 		
	such that 	{$v \not\in G^{\delta}_{\operatorname{loc}}(\Omega,X)$}.
	To this end, let us fix an arbitrary parameter $\sigma > \delta$, and consider a generalization of Cauchy's classic {flat function}, defined on $X = {\mathbb{R}}$ as:
	\begin{align*}
		f_{\sigma}(x) =
		\begin{cases}
			\exp\left(-x^{-\frac{1}{\sigma-1}}\right) & \text{if } x > 0, \\
			0 & \text{if } x \leq 0.
		\end{cases}
	\end{align*}
	
	On the one hand, for any $x \neq 0$, $f_{\sigma}$ is real analytic in a neighborhood of $x$, meaning it trivially satisfies the pointwise Gevrey bound \eqref{eq:G-regularity-pointwise} for any $\delta \geq 1$. At $x=0$, all derivatives vanish, i.e., $f_{\sigma}^{(n)}(0) = 0$ for all $n \in {\mathbb{N}}$, which obviously satisfies the pointwise bound for any chosen constants $\kappa, \beta > 0$. Therefore, 
	$f_{\sigma} \in {G^\delta_{\operatorname{pw}}({\mathbb{R}},{\mathbb{R}})}$.	
	On the other hand, it is a well-known result (cf. \cite[Lemma 1]{Chung2004}, \cite{Chernov2012}) that $f_{\sigma}$ belongs to the local Gevrey class {$G^\sigma_{\operatorname{loc}}({\mathbb{R}},{\mathbb{R}})$}, but fails to be in {$G^\delta_{\operatorname{loc}}({\mathbb{R}},{\mathbb{R}})$} for any 
	{$\delta \in [1,\sigma)$}. Specifically, the uniform local bound required by Definition \ref{def:Gevrey-Class-Local} is violated around $x=0$: for any neighborhood $(-\varepsilon, \varepsilon)$ and any constants $\kappa, \beta > 0$, there exist an integer $n \in {\mathbb{N}}$ and a point $\xi \in (-\varepsilon, \varepsilon)$ such that
	\begin{equation*}
		{\left|f_{\sigma}^{(n)}(\xi)\right|} > \kappa \, \beta^n (n!)^\delta.
	\end{equation*}
	{Consequently, $f_{\sigma}$ is not locally Gevrey of order $\delta$ at the origin.}
	\hfill
\end{proof}
\begin{remark}\label{rem:nowhere_Gevrey}
	For any $\sigma>\delta \geq 1$, it is possible to construct a pathological function that globally belongs to $f_{\sigma} \in {G^\sigma_{\operatorname{pw}}({\mathbb{R}},{\mathbb{R}})}$, yet is nowhere in {$G^\delta_{\operatorname{loc}}({\mathbb{R}},{\mathbb{R}})$}. Following Chung et al. \cite[Theorem 2]{Chung2004}, such a function can be explicitly defined as:
	\begin{equation}
		G(x) = \sum_{n  \in \mathbb{N}} \exp(n!)\, F\left(\{2^n x\}\right),
	\end{equation}
	where $F(\xi) = f_{\sigma}(\xi)f_{\sigma}(1-\xi)$ is a smooth bump function, and $\{\cdot\}$ indicates the fractional part.
\end{remark}

\begin{lemma}\label{lemma:equiv defs}
	Suppose $v: \Omega \to X$ is an infinitely differentiable function. Then the following characterizations hold:
	\begin{enumerate}[label={\rm (\roman*)}]
		\item \label{item:equiv-local} 
		{$v \in G^{\delta}_{\operatorname{loc}}(\Omega,X)$} if and only if for every $\boldsymbol{y} \in \Omega$, there exist an open neighborhood $U_{{\boldsymbol{y}}} \subseteq \Omega$ of $\boldsymbol{y}$ and a sequence of positive functions ${\boldsymbol{\beta}}(\boldsymbol{y}) = \left(\beta_j(\boldsymbol{y})\right)_{j \in \mathbb{N}}$ such that 
		\begin{equation} \label{eq:G-regularity-local-equiv}
			\left\| \partial^{\boldsymbol{\nu}} v(\boldsymbol{\xi}) \right\|_X 
			\leq  
			{\boldsymbol{\beta}}(\boldsymbol{y})^{{\boldsymbol{\nu}}} \left( |{\boldsymbol{\nu}}|! \right)^\delta, \quad 
			\forall {\boldsymbol{\nu}} \in \mathcal{F} \setminus \left\{\boldsymbol{0}\right\}, \; 
			\forall \boldsymbol{\xi} \in U_{{\boldsymbol{y}}}.
		\end{equation}
		
		\item \label{item:equiv-pointwise}  
		{$v \in G^{\delta}_{\operatorname{pw}}(\Omega,X)$} if and only if for every $\boldsymbol{y} \in \Omega$, there exists a sequence of positive functions ${\boldsymbol{\beta}}(\boldsymbol{y}) = \left(\beta_j(\boldsymbol{y})\right)_{j \in \mathbb{N}}$ such that 
		\begin{equation} \label{eq:G-regularity-equiv}
			\left\| \partial^{\boldsymbol{\nu}} v(\boldsymbol{y}) \right\|_X  
			\leq 
			{\boldsymbol{\beta}}(\boldsymbol{y})^{{\boldsymbol{\nu}}} \left( |{\boldsymbol{\nu}}|! \right)^\delta, \quad 
			\forall {\boldsymbol{\nu}} \in \mathcal{F} \setminus \left\{\boldsymbol{0}\right\}.
		\end{equation}
	\end{enumerate}
\end{lemma}

\begin{proof} 
	We establish part \ref{item:equiv-local}, as assertion \ref{item:equiv-pointwise} follows by identical reasoning evaluated strictly at $\boldsymbol{\xi} = \boldsymbol{y}$.
	
	Suppose {$v \in G^{\delta}_{\operatorname{loc}}(\Omega,X)$}. By definition, there exists a bounding function $\kappa(\boldsymbol{y}) \ge 1$ and a sequence ${\boldsymbol{\beta}}(\boldsymbol{y})$ such that the Gevrey bound holds for all ${\boldsymbol{\nu}} \in \mathcal{F}$. We define a modified sequence $\bar{{\boldsymbol{\beta}}}(\boldsymbol{y}) := \kappa(\boldsymbol{y}){\boldsymbol{\beta}}(\boldsymbol{y})$. For any multi-index ${\boldsymbol{\nu}} \ne \boldsymbol{0}$, we have $|{\boldsymbol{\nu}}| \ge 1$, which implies $\kappa(\boldsymbol{y}) \le \kappa(\boldsymbol{y})^{|{\boldsymbol{\nu}}|}$ since $\kappa(\boldsymbol{y}) \ge 1$. Consequently,
	$$\kappa(\boldsymbol{y}){\boldsymbol{\beta}}(\boldsymbol{y})^{{\boldsymbol{\nu}}} \le \kappa(\boldsymbol{y})^{|{\boldsymbol{\nu}}|}{\boldsymbol{\beta}}(\boldsymbol{y})^{{\boldsymbol{\nu}}} = \bar{{\boldsymbol{\beta}}}(\boldsymbol{y})^{{\boldsymbol{\nu}}}.$$
	Replacing the original sequence with $\bar{{\boldsymbol{\beta}}}(\boldsymbol{y})$ immediately yields the desired bound \eqref{eq:G-regularity-local-equiv}.
	
	Conversely, suppose \eqref{eq:G-regularity-local-equiv} holds. For ${\boldsymbol{\nu}} \ne \boldsymbol{0}$, this matches the standard Gevrey condition with $\kappa(\boldsymbol{y}) \equiv 1$. It remains to verify the case ${\boldsymbol{\nu}} = \boldsymbol{0}$. Since $v$ is infinitely differentiable, it is continuous, implying that $\left\|v(\boldsymbol{\xi})\right\|_X$ is bounded uniformly on the compact closure of an adequately shrunken neighborhood $U_{{\boldsymbol{y}}}$. By choosing $\kappa(\boldsymbol{y}) := \max\big(1, \sup_{\boldsymbol{\xi} \in U_{{\boldsymbol{y}}}} \left\|v(\boldsymbol{\xi})\right\|_X\big)$, the bound is satisfied for all ${\boldsymbol{\nu}} \in \mathcal{F}$, confirming that 
	{$v \in G^{\delta}_{\operatorname{loc}}(\Omega,X)$}.
	\hfill
\end{proof}

\subsection{Falling factorial} \label{sec: multiindex}
The falling factorial $\left(\tfrac{1}{2}\right)_m$ is defined by
\begin{equation*}
	\left(\tfrac{1}{2}\right)_m
	:=
	\begin{cases}
		1 \quad & \text{if } m=0, \\
		\displaystyle \prod_{{k=0}}^{m-1} \left(\tfrac{1}{2} - k\right)& \text{if } m > 0. 
	\end{cases}	
\end{equation*}	
For $m \in {\mathbb N}_0$, where ${\mathbb N}_0$ is the set of nonnegative integers, we introduce the quantity
\begin{equation*}
	\left[\tfrac{1}{2} \right]_{m}  := {\left|\left(\tfrac{1}{2}\right)_m\right|}.
\end{equation*}	
An immediate consequence of this definition is that
\begin{equation} \label{ineq: ffac_m}
	\left[\tfrac{1}{2} \right]_{m} 
	\ \le \ m!
	\ \le \ 
	2^{m+1} \left[\tfrac{1}{2} \right]_{m} .
\end{equation}	
Furthermore, we have the following convolution-type bounds:
\begin{subequations}\label{ineq: sum-ffac a}
	\begin{align} 
		\sum_{j=0}^{m-1}\left[\tfrac{1}{2} \right]_{m-j} \left[\tfrac{1}{2} \right]_{j}  \ &\le \  3 \left[\tfrac{1}{2} \right]_{m} , \label{ineq: sum-ffac}\\
		\sum_{j=0}^{m}\left[\tfrac{1}{2} \right]_{m+1-j} \left[\tfrac{1}{2} \right]_{j}  \ &\le \  \left[\tfrac{1}{2} \right]_{m+1} . \label{ineq: sum-ffac 2}
	\end{align}	
\end{subequations}
For the proofs of these inequalities, we refer to \cite[Lemma 2.3]{CL2024} and \cite[Lemma 2.1]{CL2024b}.
To streamline the presentation, we introduce the shorthand notation
\begin{equation*}
	\phi_{\delta}(\boldsymbol{\nu}) :=  \left[\tfrac{1}{2} \right]_{{\left|{\boldsymbol{\nu}}\right|}}   ({\left|{\boldsymbol{\nu}}\right|}!)^{\delta-1}, {\qquad \boldsymbol{\nu} \in \mathcal F.}
\end{equation*}

In our subsequent analysis, we will frequently invoke the following combinatorial inequalities:
\begin{equation} \label{frac{1}{4}...}	
	\sum_{{\boldsymbol{0} \le {\boldsymbol{m}} < {\boldsymbol{\nu}}}} \binom{{\boldsymbol{\nu}}}{{\boldsymbol{m}}}\phi_{\delta}({{\boldsymbol{\nu}}-{\boldsymbol{m}}})\phi_{\delta}({{\boldsymbol{m}}})
	\  \le \
	3\phi_{\delta}({{\boldsymbol{\nu}}})
	\ \le \ 	
	3(|{\boldsymbol{\nu}}|!)^\delta,		
\end{equation}	
and, for any sequence ${\boldsymbol{e}} \in \mathcal F$ of length ${\left|{\boldsymbol{e}}\right|}=1$,
\begin{equation} \label{fall_fac_ineq}	
	\sum_{{\boldsymbol{0} \leq}{\boldsymbol{m}} \le {\boldsymbol{\nu}}} \binom{{\boldsymbol{\nu}}}{{\boldsymbol{m}}}\phi_{\delta}({{\boldsymbol{\nu}}+{\boldsymbol{e}}-{\boldsymbol{m}}})\phi_{\delta}({{\boldsymbol{m}}})
	\  \le \
	\phi_{\delta}({{\boldsymbol{\nu}}+{\boldsymbol{e}}})
	\ \le \ 	
	(|{\boldsymbol{\nu}}+{\boldsymbol{e}}|!)^\delta.		
\end{equation}	
These bounds follow as direct consequences of \eqref{ineq: sum-ffac a}, the elementary inequality $n!m!\leq (n+m)!$ for all $n,m\in {\mathbb{N}}_0$, and the multivariate Vandermonde identity (cf.\ \cite[Lemma 7.1]{CL2024}).

Furthermore, we recall the {multinomial theorem} (see e.g. \cite[p.~28]{Stanley2011}), which states that for any integer $k \ge 0$ and any number sequence ${\boldsymbol{z}} = (z_j)_{j \in {\mathbb N}}$, the $k$-th power of its sum can be expanded over multi-indices as
\begin{equation} \label{eq:multinomial}
	\left( \sum_{j \in {\mathbb N}} z_j \right)^k 
	= \sum_{\boldsymbol{\nu} \in \mathcal F \atop {\left|\boldsymbol{\nu}\right|} = k} \frac{k!}{\boldsymbol{\nu}!} {\boldsymbol{z}}^{\boldsymbol{\nu}},
\end{equation}
where the sum is taken over all finitely supported multi-indices $\boldsymbol{\nu} \in \mathcal F$ of length ${\left|\boldsymbol{\nu}\right|} = k$. By applying a formal summation over $k \ge 0$, \eqref{eq:multinomial} naturally extends to the expansion of geometric series or exponential functions involving multi-indices.

A prominent application of this extension, which will strongly facilitate our subsequent estimates in Section \ref{sec: truncation}, is a multi-index identity arising from the second-order partial derivatives of a multivariate geometric series. Namely, for any number sequence $\boldsymbol{\mathsf{b}} = ({\mathsf{b}}_j)_{j \in {\mathbb N}}$ satisfying $\left\|{\boldsymbol{\mathsf{b}}}\right\|_{\ell^1} < 1$ and each index $i \in {\mathbb N}$, we have
\begin{align}\label{eq:geometric series}
	\sum_{\boldsymbol{\nu}\in\mathcal F}
	\frac{{\left|\boldsymbol{\nu}\right|}!}{\boldsymbol{\nu}!}
	\nu_i(\nu_i-1)
	\boldsymbol{\mathsf{b}}^{\boldsymbol{\nu}}
	=
	\frac{2\,b_i^2}{(1-\sum_{j \in {\mathbb N}}{\mathsf{b}}_j)^3}.
\end{align} 
To justify this identity, we note that the corresponding generating function expands into a geometric series:
\begin{align*}
	\frac{1}{1-\sum_{j \in {\mathbb N}}{\mathsf{b}}_j z_j} 
	= \sum_{k \in {\mathbb N}_0} \left(\sum_{j \in {\mathbb N}}{\mathsf{b}}_j z_j\right)^k 
	= \sum_{\boldsymbol{\nu}\in\mathcal F} \frac{{\left|\boldsymbol{\nu}\right|}!}{\boldsymbol{\nu}!} \boldsymbol{\mathsf{b}}^{\boldsymbol{\nu}} {\boldsymbol{z}}^{\boldsymbol{\nu}},
\end{align*}
for any number sequence $\boldsymbol{z}=(z_j)_{j \in \mathbb N}$ with $|z_j| \le 1$, where the last equality follows from the multinomial theorem. Taking the second-order partial derivative with respect to $z_i$ on both sides yields
\begin{align*}
	\frac{\partial^2}{\partial z_i^2} \left( \frac{1}{1-\sum_{j \in {\mathbb N}}{\mathsf{b}}_j z_j} \right) 
	= \sum_{\boldsymbol{\nu}\in\mathcal F} \frac{{\left|\boldsymbol{\nu}\right|}!}{\boldsymbol{\nu}!}\, \nu_i(\nu_i-1) z_i^{-2}\, \boldsymbol{\mathsf{b}}^{\boldsymbol{\nu}} {\boldsymbol{z}}^{\boldsymbol{\nu}}.
\end{align*}
Alternatively, differentiating the analytic expression of the generating function directly gives
\begin{align*}
	\frac{\partial^2}{\partial z_i^2} \left( \frac{1}{1-\sum_{j \in {\mathbb N}}{\mathsf{b}}_j { z}_j}\right)
	= \frac{2\,b_i^2}{(1-\sum_{j \in {\mathbb N}}{\mathsf{b}}_j z_j)^3}.
\end{align*}
Evaluating both derivative results at $z_j = 1$ for all $j \in {\mathbb N}$ establishes the desired equality in \eqref{eq:geometric series}.

\section{Gevrey regularities of the weak solution}
\label{sec:Gevrey regularity of the weak solution}
In this section, we investigate {the local and pointwise Gevrey $\delta$-regularities}  of the weak solution map ${\boldsymbol{y}} \mapsto u({\boldsymbol{y}})$ to the parametric equation \eqref{ellip}. To facilitate the analysis of partial parametric derivatives, we first introduce necessary functional framework and weighted norm structures.

For $1 \le p \le \infty$ and a positive function ${\lambda}$ on $D$, we  introduce the weighted {norms}
$$
\left\|{v}\right\|_{p,\lambda}:= \left(\int_{{D}} v({\boldsymbol{x}})^p\, \lambda({\boldsymbol{x}}) \, {\rm d}{\boldsymbol{x}}\right)^{\frac{1}{p}},\ \ 1\le p < \infty, 
$$	
and
{$$\left\|{v}\right\|_{\infty,\lambda}:= \textrm{ess sup}_{{\boldsymbol{x}}\in D} {{\left|v({\boldsymbol{x}})\right|}}\,{\lambda({\boldsymbol{x}})}$$}for a function $v$ defined on $D$ ($\lambda$ is dropped in the case $\lambda =1$). For $r \in {\mathbb N}_0$, we define the weighted Sobolev norm 
{\begin{equation*}
		\left\|{v}\right\|_{{W^r_{p,\lambda}}}
		:= \
	\max_{|{\boldsymbol{\alpha}}| \le r} \left\|{\mathsf{D}^{\boldsymbol{\alpha}} v}\right\|_{p,\lambda}.
\end{equation*}	}
We also use the notation $H^r_{\lambda}:= W^r_{2,\lambda}$. Recall that $V = H^1_0(D)$ endowed with the norm $\|v\|_V := \|\nabla v\|_2$. 

{{}The following theorem establishes the parametric Gevrey $\delta$-regularity for the solution to the sub-problem on any linear subspace $X \subseteq V$, based on certain Gevrey $\delta$-regularity assumptions on the diffusion coefficient $a({\boldsymbol{y}})$. Note that this includes the exact weak solution to \eqref{ellip} when $X = V$.}

	\begin{theorem}\label{thm:G-regularity-u(by)}
For a given $\delta\ge 1$,
assume that there exist a  positive function $\kappa({\boldsymbol{y}})\ge 1$ and  sequence ${\boldsymbol{\beta}}({\boldsymbol{y}})$ of positive functions in variable ${\boldsymbol{y}}$ such that 
	\begin{equation} \label{cond:G-regularity-a(by)-diff-coeff}
	\forall  {\boldsymbol{y}}  \in {{\mathbb R}^{\mathbb{N}}} \quad \exists U_{{\boldsymbol{y}}}: \quad \forall \boldsymbol{\xi} \in U_{{\boldsymbol{y}}}  \quad
		\left\|{\partial^{\boldsymbol{\nu}} a(\boldsymbol{\xi})}\right\|_{\infty,a(\boldsymbol{\xi})^{-1}}	
		\ \le \
		\kappa({\boldsymbol{y}}){\boldsymbol{\beta}}({\boldsymbol{y}})^{{\boldsymbol{\nu}}}\,(|{\boldsymbol{\nu}}|!)^\delta
	\end{equation}	
	for some positive function $\kappa({\boldsymbol{y}})\ge 1$ and  sequence ${\boldsymbol{\beta}}({\boldsymbol{y}})$ of positive functions in variable ${\boldsymbol{y}}$.
	
	{{}	Let $X \subseteq V$ be a linear subspace and consider the associated sub-problem of \eqref{weak-form}: Find $u \in X$ such that 
	\begin{align}\label{ellip sub-prob}
		\int_{D} a({\boldsymbol{y}})({\boldsymbol{x}}) \nabla u({\boldsymbol{y}})({\boldsymbol{x}}) \cdot \nabla v({\boldsymbol{x}}) \, {\rm d} {\boldsymbol{x}} = \langle f, v \rangle \quad \forall v \in X.
	\end{align}
	Then, for all ${\boldsymbol{\nu}} \in \mathcal F$, the parametric solution $u$ to \eqref{ellip sub-prob} satisfies}
	\begin{equation} \label{ineq:G-regularity-u(by)}
		\begin{aligned}
			\forall  {\boldsymbol{y}}  \in {{\mathbb R}^{\mathbb{N}}} \quad \exists U_{{\boldsymbol{y}}}: \quad \forall \boldsymbol{\xi} \in U_{{\boldsymbol{y}}}  \quad
			\left\|{\partial^{\boldsymbol{\nu}} u(\boldsymbol{\xi})}\right\|_{V}
			&\le \
			\left(12 \kappa({\boldsymbol{y}}) {\boldsymbol{\beta}}({\boldsymbol{y}})\right)^{{\boldsymbol{\nu}}}\,\phi_\delta({\boldsymbol{\nu}}) \left\|{{f}}\right\|_{V'} 
			\left\|{a(\boldsymbol{\xi})^{-1}}\right\|_{\infty}
			\\&
			\le \
			\left(12 \kappa({\boldsymbol{y}}) {\boldsymbol{\beta}}({\boldsymbol{y}})\right)^{{\boldsymbol{\nu}}}\,(|{\boldsymbol{\nu}}|!)^\delta \left\|{{f}}\right\|_{V'} 
			\left\|{a(\boldsymbol{\xi})^{-1}}\right\|_{\infty}.
		\end{aligned}
	\end{equation}	
\end{theorem}

\begin{proof}
	We adopt the abbreviations 
		$\kappa := \kappa(\boldsymbol{y})$ and 
	${\boldsymbol{\beta}} := {\boldsymbol{\beta}}(\boldsymbol{y})$.
	We have by \eqref{cond:G-regularity-a(by)-diff-coeff}  and \eqref{ineq: ffac_m},
	\begin{equation} \label{ineq:G-regularity-a(by)-diff-coeff}
		\left\|{\partial^{\boldsymbol{\nu}} a(\boldsymbol{\xi})}\right\|_{\infty,a(\boldsymbol{\xi})^{-1}}
		\le 2\kappa\, {(2{\boldsymbol{\beta}})^{\boldsymbol{\nu}}}\,\phi_{\delta}({{\boldsymbol{\nu}}}), \quad \forall \boldsymbol{\xi} \in U_{{\boldsymbol{y}}}, \; \forall {\boldsymbol{\nu}} \in \mathcal F.
	\end{equation}	
	{{}By the linearity of the subspace $X \subseteq V$, the parametric derivatives $\partial^{\boldsymbol{\nu}} u$ of the solution $u \in X$ also belong to $X$. We shall prove by induction on ${\boldsymbol{\nu}}$ that for all $\boldsymbol{y} \in \mathbb{R}^{\mathbb{N}}$, the following holds}
	\begin{equation} \label{norm{partial^bs u}{H^1_{a^{-1}}}}
		\left\|{\partial^{\boldsymbol{\nu}} u(\boldsymbol{\xi})}\right\|_{H^1_{a(\boldsymbol{\xi})}}	
		\le \left(12\kappa{\boldsymbol{\beta}}\right)^{{\boldsymbol{\nu}}} \, \phi_{\delta}({{\boldsymbol{\nu}}}) \, \left\|{f}\right\|_{V'} \,
		\left\|{a(\boldsymbol{\xi})^{-1}}\right\|_{\infty}^{\frac{1}{2}}
		\quad \forall \boldsymbol{\xi} \in U_{{\boldsymbol{y}}}, \; \forall {\boldsymbol{\nu}} \in \mathcal F.
	\end{equation}	
	For ${\boldsymbol{\nu}} = \boldsymbol{0}$, the weak formulation \eqref{ellip sub-prob} implies
	\begin{equation} \label{weak-form1}
		\int_{D} a(\boldsymbol{\xi})|\nabla u(\boldsymbol{\xi})|^2\, {\rm d} {\boldsymbol{x}}
		\le \left\|{f}\right\|_{V'}
		\left\|{a(\boldsymbol{\xi})^{-1}}\right\|_{\infty}^{\frac{1}{2}}
		\left(\int_{D} a(\boldsymbol{\xi})|\nabla u(\boldsymbol{\xi})|^2\, {\rm d} {\boldsymbol{x}}\right)^{1/2},
	\end{equation}
	which simplifies to
	\begin{equation} \label{weak-form2}
		\left(\int_{D} a(\boldsymbol{\xi})|\nabla u(\boldsymbol{\xi})|^2\, {\rm d} {\boldsymbol{x}}\right)^{1/2}
		\le \left\|{f}\right\|_{V'} 
		\left\|{a(\boldsymbol{\xi})^{-1}}\right\|_{\infty}^{\frac{1}{2}}.
	\end{equation}
	Thus, the base case holds.
	
	Assume the claim \eqref{norm{partial^bs u}{H^1_{a^{-1}}}} holds for all ${{\boldsymbol{m}}} < {\boldsymbol{\nu}}$.
	Differentiating the weak formulation \eqref{ellip sub-prob} with respect to the multi-index ${\boldsymbol{\nu}}$ and applying the Leibniz rule, we have for any {{}$v \in X$}:
	\begin{equation*}	
		\sum_{\boldsymbol{0} \le {\boldsymbol{m}} \le {\boldsymbol{\nu}}} \int_{D} \binom{{\boldsymbol{\nu}}}{{\boldsymbol{m}}}\partial^{{\boldsymbol{\nu}}- {\boldsymbol{m}}} a(\boldsymbol{\xi}) \, \left(\nabla \partial^{{\boldsymbol{m}}} u(\boldsymbol{\xi}) \cdot \nabla v\right) \, {\rm d} {\boldsymbol{x}}	
		= 0.
	\end{equation*}	
	By choosing $v = \partial^{{\boldsymbol{\nu}}} u$ and isolating the term ${{\boldsymbol{m}}} = {\boldsymbol{\nu}}$, it follows that
	\begin{equation*}
		\int_{D} a(\boldsymbol{\xi}) {\left|\nabla \partial^{{\boldsymbol{\nu}}} u(\boldsymbol{\xi})\right|}^2 \, {\rm d} {\boldsymbol{x}}	
		= -\sum_{{\boldsymbol{0} \le {\boldsymbol{m}} < {\boldsymbol{\nu}}}}\,\int_{D} \binom{{\boldsymbol{\nu}}}{{\boldsymbol{m}}}\partial^{{\boldsymbol{\nu}}  - {\boldsymbol{m}}} a(\boldsymbol{\xi}) \, \left(\nabla \partial^{{\boldsymbol{m}}} u(\boldsymbol{\xi}) \cdot \nabla \partial^{{\boldsymbol{\nu}}}u(\boldsymbol{\xi})\right) \, {\rm d} {\boldsymbol{x}}.	
	\end{equation*}	
	Applying the Cauchy-Schwartz inequality and dividing both sides by $\left\|{\partial^{{\boldsymbol{\nu}}} u(\boldsymbol{\xi})}\right\|_{H^1_{a(\boldsymbol{\xi})}}$, we deduce:	
	\begin{equation} \label{norm{partial...}}
		\left\|{\partial^{{\boldsymbol{\nu}}} u(\boldsymbol{\xi})}\right\|_{H^1_a(\boldsymbol{\xi})}
		\le \sum_{{\boldsymbol{0} \le {\boldsymbol{m}} < {\boldsymbol{\nu}}}}
		\left\|{\partial^{{\boldsymbol{\nu}}- {\boldsymbol{m}}} a(\boldsymbol{\xi})}\right\|_{\infty,{a(\boldsymbol{\xi})}^{-1}} \, \left\|{\partial^{{{\boldsymbol{m}}}} u(\boldsymbol{\xi})}\right\|_{H^1_{a(\boldsymbol{\xi})}}.
	\end{equation}	
	Substituting \eqref{ineq:G-regularity-a(by)-diff-coeff} and the inductive hypothesis {\eqref{norm{partial^bs u}{H^1_{a^{-1}}}}}, we get:\\
\resizebox{\textwidth}{!}{
	\begin{minipage}{\textwidth}
		\begin{center}
	\begin{align*} 
		\left\|{\partial^{{\boldsymbol{\nu}}} u(\boldsymbol{\xi})}\right\|_{H^1_{a(\boldsymbol{\xi})}}	
		&\le \sum_{{\boldsymbol{0} \le {\boldsymbol{m}} < {\boldsymbol{\nu}}}} \binom{{\boldsymbol{\nu}}}{{\boldsymbol{m}}}
		2\kappa \left(2{\boldsymbol{\beta}}\right)^{{\boldsymbol{\nu}}-{\boldsymbol{m}}} \,\phi_{\delta}({{\boldsymbol{\nu}}-{\boldsymbol{m}}})
		\left(12 \kappa{\boldsymbol{\beta}}\right)^{{\boldsymbol{m}}}\,\phi_{\delta}({{\boldsymbol{m}}})
		\left\|{f}\right\|_{V'}
		\left\|{a(\boldsymbol{\xi})^{-1}}\right\|_{\infty}^{\frac{1}{2}}
		\notag \\
		&\le2\kappa\, (6\kappa)^{{\left|\boldsymbol{\nu}\right|}-1} \left(2{\boldsymbol{\beta}}\right)^{{\boldsymbol{\nu}}} \left\|{f}\right\|_{V'}
		\left\|{a(\boldsymbol{\xi})^{-1}}\right\|_{\infty}^{\frac{1}{2}}
		\sum_{{\boldsymbol{0} \le {\boldsymbol{m}} < {\boldsymbol{\nu}}}}\binom{{\boldsymbol{\nu}}}{{\boldsymbol{m}}}
		\phi_{\delta}({{\boldsymbol{\nu}}-{\boldsymbol{m}}}) \, \phi_{\delta}({{\boldsymbol{m}}}).%
	\end{align*}
	\end{center}
\end{minipage}
}
	Applying property \eqref{frac{1}{4}...}, {the sum in the right-hand side} is bounded by $3\phi_{\delta}({\boldsymbol{\nu}})$.
	This completes the induction for $\left\|{\partial^{{\boldsymbol{\nu}}} u}\right\|_{H^1_{a(\boldsymbol{\xi})}}$.
	
	Finally, using the relationship $\left\|{\partial^{\boldsymbol{\nu}} u(\boldsymbol{\xi})}\right\|_{V} \le 
	\left\|{a(\boldsymbol{\xi})^{-1}}\right\|_{\infty}^{\frac{1}{2}}
	\left\|{\partial^{\boldsymbol{\nu}} u}\right\|_{H^1_{a(\boldsymbol{\xi})}}$ and the fact that $\phi_{\delta}(k) \le (k!)^\delta$, we arrive at the desired estimate \eqref{ineq:G-regularity-u(by)}. 
	\hfill
\end{proof}

	The following preliminary result provides a Gevrey $\delta$-regularity for the diffusion coefficient $a$ based on a Gevrey $\delta$-regularity of the function $Z$ defined as in \eqref{def:diffus-coeff}, which serves as the cornerstone for our subsequent regularity analysis.

\begin{lemma}\label{lemma:G-regularity-a(by)}
	Let  $a$ be a function on ${\mathbb R}^{\mathbb{N}}$ 
	of the form 
	$$a({\boldsymbol{y}}):= \exp(Z({\boldsymbol{y}})). 
	$$ 
	Assume that for all $\boldsymbol{y} \in {\mathbb R}^{{\mathbb{N}}}$ there exist an open neighborhood $U_{{\boldsymbol{y}}}$ of $\boldsymbol{y}$, a function $\rho(\boldsymbol{y})\geq 1$ and a sequence ${{\boldsymbol{\alpha}}(\boldsymbol{y})}$ of positive functions such that 
	\begin{equation} \label{ineq:G-regularity-b(by) 2}
		\left\|{\partial^{\boldsymbol{\nu}} Z(\boldsymbol{\xi})}\right\|_{\infty}	
		\le \rho(\boldsymbol{y}) {\boldsymbol{\alpha}}(\boldsymbol{y})^{{\boldsymbol{\nu}}} \left({\left|{\boldsymbol{\nu}}\right|}!\right)^{\delta}, \quad \forall \boldsymbol{\xi} \in U_{{\boldsymbol{y}}},
		\ \ \forall {\boldsymbol{\nu}} \in \mathcal F\setminus \{\boldsymbol{0}\}.
	\end{equation}	
	Then the following bounds hold true
	\begin{equation} \label{ineq:G-regularity-a(by)}
		\begin{split}
			\left\|{\partial^{\boldsymbol{\nu}} a(\boldsymbol{\xi})}\right\|_{\infty,a(\boldsymbol{\xi})^{-1}}
			&\le \
			\left(4\rho(\boldsymbol{y}) {\boldsymbol{\alpha}}(\boldsymbol{y})\right)^{{\boldsymbol{\nu}}}\,\phi_{\delta}({{\boldsymbol{\nu}}}),  \\
			&\le \
			\left(4\rho(\boldsymbol{y}) {\boldsymbol{\alpha}}(\boldsymbol{y})\right)^{{\boldsymbol{\nu}}}\,(|{\boldsymbol{\nu}}|!)^\delta  \ \ \forall  \boldsymbol{\xi}\in U_{{\boldsymbol{y}}}, \    \forall  {\boldsymbol{y}}  \in \mathbb{R}^{{\mathbb{N}}}, \ \forall {\boldsymbol{\nu}} \in \mathcal F.
		\end{split}
	\end{equation}	
\end{lemma}
\begin{remark}\label{rem:unbounded_exponent}
	We emphasize that the regularity assumption \eqref{ineq:G-regularity-b(by) 2} on $Z$ is imposed only on the strict derivatives (${\boldsymbol{\nu}} \neq \boldsymbol{0}$). By not requiring a uniform bound on $Z$ itself, we allow the underlying field to be potentially unbounded. We point out that the relevant derivative bound \eqref{ineq:G-regularity-a(by)} for the diffusion coefficient $a = \exp(Z)$ still holds for all ${\boldsymbol{\nu}} \in \mathcal F$, including $\boldsymbol{0}$. This relaxation is practically significant, as it readily accommodates log-normal models where $Z$ is not  $L_\infty$-bounded. In Section \ref{sec: total error} below, we consider a more general situation where the exponent $Z({\boldsymbol{y}})$ is not necessarily restricted to an affine combination. Consequently, the classical log-normal model is naturally recovered as a special case.
\end{remark}
\begin{proof}
	For brevity, we omit the explicit dependence on $\boldsymbol{y}$ and write $\rho$ and ${\boldsymbol{\alpha}}$ instead of $\rho(\boldsymbol{y})$ and ${\boldsymbol{\alpha}}(\boldsymbol{y})$.	
	First, note that the assumption \eqref{ineq:G-regularity-b(by) 2} and the %
	{right inequality in}	\eqref{ineq: ffac_m} imply
	\begin{equation} \label{ineq:G-regularity-b(by) 1}
		\left\|{\partial^{\boldsymbol{\nu}} Z(\boldsymbol{\xi})}\right\|_{\infty}	
		\le 2\rho \, (2{\boldsymbol{\alpha}})^{{\boldsymbol{\nu}}} \, \phi_{\delta}({{\boldsymbol{\nu}}}), \quad \forall \boldsymbol{\xi} \in U, \; \forall {\boldsymbol{\nu}} \in \mathcal F \setminus \left\{\boldsymbol{0}\right\}.
	\end{equation}	
	We proceed to prove \eqref{ineq:G-regularity-a(by)} by induction on the multi-index $\boldsymbol{\nu}$. For $\boldsymbol{\nu} = \mathbf{0}$, the estimate holds trivially since $\left\|{\partial^{\boldsymbol{0}} a(\boldsymbol{\xi})}\right\|_{\infty, a(\boldsymbol{\xi})^{-1}} = 1$.
	
	Suppose that for any ${\boldsymbol{\nu}} \in \mathcal F$ and for any $\boldsymbol{\xi}$ in the neighborhood {$U_{\boldsymbol{y}}$} of $\boldsymbol{y}$, the inequality \eqref{ineq:G-regularity-a(by)} holds for all ${{\boldsymbol{m}}} \le {\boldsymbol{\nu}}$. Let ${{\boldsymbol{e}}} \in {\mathcal F}$ be a unit vector with $|{{\boldsymbol{e}}}|=1$. Applying the Leibniz rule to $\partial^{{\boldsymbol{\nu}}+{{\boldsymbol{e}}}} a = \partial^{{\boldsymbol{\nu}}} (a \partial^{{{\boldsymbol{e}}}} Z)$, we obtain
	\begin{equation} \label{partial^bs a=}
		\partial^{{\boldsymbol{\nu}}+{\boldsymbol{e}}} a(\boldsymbol{\xi})
		= \partial^{{\boldsymbol{\nu}} } \left(\partial^{\boldsymbol{e}} a(\boldsymbol{\xi})\right)
		= \partial^{{\boldsymbol{\nu}}} \left(a(\boldsymbol{\xi})\,\partial^{\boldsymbol{e}} Z(\boldsymbol{\xi})\right)
		= \sum_{\boldsymbol{0} \leq {\boldsymbol{m}} \le {\boldsymbol{\nu}}} \binom{{\boldsymbol{\nu}}}{{\boldsymbol{m}}}\partial^{{\boldsymbol{\nu}} +{\boldsymbol{e}} - {\boldsymbol{m}}} Z(\boldsymbol{\xi}) \, \partial^{{\boldsymbol{m}}} a(\boldsymbol{\xi}). 
	\end{equation}	
	Taking the norm and using the inductive hypothesis together with \eqref{ineq:G-regularity-b(by) 1}, it follows that
	\begin{equation} \label{norm{partial^bs a}{infty, a}}
		\begin{aligned}
			\left\|{\partial^{{\boldsymbol{\nu}}+{\boldsymbol{e}}} a(\boldsymbol{\xi})}\right\|_{\infty, a(\boldsymbol{\xi})^{-1}}
			&\le \sum_{\boldsymbol{0}\le{\boldsymbol{m}} \le {\boldsymbol{\nu}}} \binom{{\boldsymbol{\nu}}}{{\boldsymbol{m}}} 
			\left\|{\partial^{{\boldsymbol{\nu}} +{\boldsymbol{e}} - {\boldsymbol{m}}}Z(\boldsymbol{\xi})}\right\|_{\infty}\,\left\|{\partial^{\boldsymbol{m}} a(\boldsymbol{\xi})}\right\|_{\infty, a(\boldsymbol{\xi})^{-1}}\\[1ex]
			&\le \sum_{\boldsymbol{0}\le{\boldsymbol{m}} \le {\boldsymbol{\nu}}} \binom{{\boldsymbol{\nu}}}{{\boldsymbol{m}}}
			{2 \rho} \, (2{\boldsymbol{\alpha}})^{{\boldsymbol{\nu}} +{\boldsymbol{e}}-{\boldsymbol{m}}} \, \phi_{\delta}({{\boldsymbol{\nu}}+{\boldsymbol{e}}-{\boldsymbol{m}}}) \, 
			(4\rho{\boldsymbol{\alpha}})^{\boldsymbol{m}} \, \phi_{\delta}({{\boldsymbol{m}}})\\[1ex]
			&\le 2\rho\cdot (2\rho)^{{\left|\boldsymbol{\nu}\right|}} \, (2{\boldsymbol{\alpha}})^{{\boldsymbol{\nu}}+{\boldsymbol{e}}} 
			\sum_{\boldsymbol{0}\le{\boldsymbol{m}} \le {\boldsymbol{\nu}}} \binom{{\boldsymbol{\nu}}}{{\boldsymbol{m}}}
			\phi_{\delta}({{\boldsymbol{\nu}}+{\boldsymbol{e}}-{\boldsymbol{m}}}) \phi_{\delta}({{\boldsymbol{m}}}).
		\end{aligned}
	\end{equation}	
	Using the property of the Gevrey-type growth factor in \eqref{fall_fac_ineq}, the sum reduces to
	\begin{equation*} 
		\left\|{\partial^{{\boldsymbol{\nu}}+{\boldsymbol{e}}} a(\boldsymbol{\xi})}\right\|_{\infty, a(\boldsymbol{\xi})^{-1}}
		\le (4\rho{\boldsymbol{\alpha}})^{{\boldsymbol{\nu}}+{\boldsymbol{e}}} \phi_{\delta}({{\boldsymbol{\nu}}+{\boldsymbol{e}}}), \quad \forall \boldsymbol{\xi} \in {U_{\boldsymbol{y}}},
	\end{equation*}
	which confirms \eqref{ineq:G-regularity-a(by)}. Finally, applying estimate \eqref{ineq: ffac_m} to the right-hand side of %
	{the above inequality} completes the proof.
	\hfill
\end{proof}

As an immediate consequence of Lemma \ref{lemma:G-regularity-a(by)} and  Theorem \ref{thm:G-regularity-u(by)}, we obtain a parametric Gevrey regularity for the weak solution assuming a  parametric Gevrey regularity of $Z$.

\begin{corollary}\label{cor:G-regularity-lognormal}
	Let $D$ be a Lipschitz bounded domain in ${\mathbb R}^d$ {{}and $X\subseteq V$ a linear subspace.} Let {{}$u\in X$} be the parametric solution to {\eqref{ellip sub-prob}} with the diffusion coefficient given by $a({\boldsymbol{y}}):= \exp(Z({\boldsymbol{y}}))$. 
	Assume that for all $\boldsymbol{y} \in \mathbb{R}^{{\mathbb{N}}}$ there exist an open neighborhood $U_{{\boldsymbol{y}}}$ of $\boldsymbol{y}$, a function $\rho(\boldsymbol{y})\ge 1$, and a sequence ${\boldsymbol{\beta}(\boldsymbol{y})}$ of positive functions such that 
	\begin{equation} \label{cond:G-regularity-Z-lognormal}
		\left\|{\partial^{\boldsymbol{\nu}} Z(\boldsymbol{\xi})}\right\|_{\infty}	
		\ \le \ \rho(\boldsymbol{y}) {\boldsymbol{\alpha}}(\boldsymbol{y})^{{\boldsymbol{\nu}}} \left({\left|{\boldsymbol{\nu}}\right|}!\right)^{\delta}, \quad \forall \boldsymbol{\xi} \in U_{{\boldsymbol{y}}},
		\ \  \forall {\boldsymbol{\nu}} \in \mathcal F\setminus \{\boldsymbol{0}\}.
	\end{equation}	
Then for the parametric solution $u$ to \eqref{ellip sub-prob}, it holds  for all ${\boldsymbol{\nu}} \in \mathcal F$ that
	\begin{equation} \label{ineq:G-regularity-u-lognormal}
		\begin{aligned}
		{\boldsymbol{y}}  \in \mathbb{R}^{{\mathbb{N}}} ~~ \exists U_{{\boldsymbol{y}}}: ~~ \forall \boldsymbol{\xi} \in {U_{{\boldsymbol{y}}}}~~
			\left\|{\partial^{\boldsymbol{\nu}} u(\boldsymbol{\xi})}\right\|_{V}
			\le \
			\left(12 \rho(\boldsymbol{y}) {\boldsymbol{\alpha}}(\boldsymbol{y})\right)^{{\boldsymbol{\nu}}}\, (|{\boldsymbol{\nu}}|!)^\delta \left\|{{f}}\right\|_{V'} 
			\exp\left(\left\|{Z(\boldsymbol{\xi})}\right\|_{\infty}\right).
		\end{aligned}
	\end{equation}	
\end{corollary}

\begin{proof}
	The proof follows immediately by combining Lemma \ref{lemma:G-regularity-a(by)} and Theorem \ref{thm:G-regularity-u(by)}. By setting $\kappa(\boldsymbol{y}):=1$ and ${\boldsymbol{\beta}}({\boldsymbol{y}}) := 4\rho(\boldsymbol{y}) {\boldsymbol{\alpha}}(\boldsymbol{y})$ in \eqref{cond:G-regularity-a(by)-diff-coeff}, the condition on the diffusion coefficient $a(\boldsymbol{\xi})$ is satisfied 
	\begin{align*}
		\left\|{\partial^{\boldsymbol{\nu}} a(\boldsymbol{\xi})}\right\|_{\infty,a(\boldsymbol{\xi})^{-1}}
		\le  {{\boldsymbol{\beta}}({\boldsymbol{y}})^{\boldsymbol{\nu}}}\,\phi_{\delta}({{\boldsymbol{\nu}}}), \quad \forall \boldsymbol{\xi} \in U_{{\boldsymbol{y}}}, \; \forall {\boldsymbol{\nu}} \in \mathcal F.
	\end{align*}
	Substituting this bounds into \eqref{norm{partial...}} yields the constant $12$, and observing that $\left\|{a(\boldsymbol{\xi})^{-1}}\right\|_{\infty} \le \exp(\left\|{Z(\boldsymbol{\xi})}\right\|_{\infty})$ completes the proof.
	\hfill
\end{proof}

\begin{remark}\label{rem:affine_lognormal}
	We emphasize that Corollary \ref{cor:G-regularity-lognormal} naturally encompasses the standard log-normal diffusion model widely studied in the literature. Specifically, when the exponent $Z$ is defined as an affine combination of the parameters, i.e., 
	$$
	Z({\boldsymbol{x}}, {\boldsymbol{y}}) = \sum_{j \in {\mathbb N}} y_j \psi_j({\boldsymbol{x}}).
	$$
	 In this context, the regularity condition \eqref{cond:G-regularity-Z-lognormal} is trivially satisfied, and our estimate \eqref{ineq:G-regularity-u-lognormal} seamlessly recovers the established regularity bounds for the affine log-normal case (see, e.g., \cite{Graham15}).
\end{remark}

\begin{remark}
	It is worth noting the flexibility of our proving technique: the Gevrey regularity estimates for $a$ and $u$ do not strictly require local bounds across the entire parameter space ${\mathbb R}^{{\mathbb{N}}}$. Indeed, if $Z$ satisfies the pointwise Gevrey condition 
	\begin{equation*} 
		\left\|{\partial^{\boldsymbol{\nu}} Z(\boldsymbol{y})}\right\|_{\infty} \le \rho(\boldsymbol{y}) \boldsymbol{\beta}(\boldsymbol{y})^{{\boldsymbol{\nu}}} (|\boldsymbol{\nu}|!)^{\delta}
	\end{equation*}
	for all $\boldsymbol{\nu} \neq \boldsymbol{0}$ and $\boldsymbol{y} \in \mathbb{R}^{{\mathbb{N}}}$, our analysis can be applied to yield pointwise regularity results.
	\begin{equation*} 
		\begin{aligned}
			\left\|{\partial^{\boldsymbol{\nu}} u(\boldsymbol{y})}\right\|_{V}
			\le 
			\left(12 \rho({\boldsymbol{y}}) {\boldsymbol{\alpha}}({\boldsymbol{y}})\right)^{{\boldsymbol{\nu}}}\,(|{\boldsymbol{\nu}}|!)^\delta \left\|{f}\right\|_{V'} \exp\left(\left\|{Z(\boldsymbol{y})}\right\|_{\infty}\right) \ \forall  {\boldsymbol{y}}  \in {{\mathbb R}^{\mathbb{N}}}, \ \forall {\boldsymbol{\nu}} \in \mathcal F.
		\end{aligned}
	\end{equation*}	
\end{remark}

\section{Dimension truncation error}\label{sec: truncation}

For a domain $Y\subseteq {\mathbb{R}}$, let $Y^{\mathbb{N}}$  denote the {countable} infinite Cartesian  product 
$Y^{\mathbb{N}}:= \prod_{j \in {\mathbb N}} Y_j$ where  $Y_j := Y$ for each $j\in{\mathbb{N}}$. Recalling the construction of the infinite tensor product of probability measures (see, e.g., \cite[pp. 429--435]{HewittStromberg75}), we define the probability measure ${{\boldsymbol{\mu}}}$ on $Y^{\mathbb{N}}$ as the infinite tensor product:
	\begin{equation} \label{bmu}
		{\boldsymbol{\mu}}({\boldsymbol{y}}) 
		:= \ 
		\bigotimes_{j \in {\mathbb N}} \mu_j(y_j) , \quad {\boldsymbol{y}} = (y_j)_{j \in {\mathbb N}} \in Y^{\mathbb{N}}.
	\end{equation}
Here, $\{\mu_j\}_{j \in {\mathbb{N}}}$  is a sequence of probability measures $Y$. The $\sigma$-algebra associated with ${{\boldsymbol{\mu}}}$ is the product $\sigma$-algebra generated by the cylinder sets of the form $A := \prod_{j \in {\mathbb{N}}} A_j$, where each $A_j \subseteq Y$ is $\mu_j$-measurable and $A_j = Y$ for all but finitely many indices $j$. For such cylinder sets, the measure is given by the product ${{\boldsymbol{\mu}}}(A) = \prod_{j \in {\mathbb{N}}} \mu_j(A_j)$.
Notably, this framework encompasses the construction of the infinite-dimensional standard Gaussian measure.

Let $F: Y^{\mathbb{N}} \to {\mathbb{R}}$ be {a ${\boldsymbol{\mu}}$-integrable function,} and consider the infinite-dimensional integral
	\begin{align}\label{approx integral}
		I(F) =	
		\int_{{Y}^{\mathbb{N}}} F({\boldsymbol{y}}) \, {\boldsymbol{\mu}}({\rm d} {\boldsymbol{y}})
	\end{align}
	{We construct} a fully discrete approximation using QMC rules, specifically for integrands $F$ exhibiting  a pointwise Gevrey-$\delta$ regularity. Since {QMC rules are not applicable to infinite-dimensional functions, in this section, we preliminarily} approximate $I(F)$ by the dimension-truncated integral:
	\begin{align*}
		I_s(F) {:=
			\int_{{Y}^{\mathbb N}}  F({\boldsymbol{y}}_s){\boldsymbol{\mu}}({\rm d} {\boldsymbol{y}})}
		=	
		\int_{{Y}^s}  F({\boldsymbol{y}}_s){\boldsymbol{\mu}}_s({\rm d} {\boldsymbol{y}}_s),
	\end{align*}
where $s$ is chosen sufficiently large such that $I(F) = \lim_{s \to \infty} I_s(F)$.

In this context, let {$\{1:{s}\}=\left\{1,\dots,s\right\}$} and let ${{\boldsymbol{y}}}_s := (y_1, \dots, y_s, 0, \dots)$ denote the $s$-dimensional truncation of ${{\boldsymbol{y}}} \in {{\mathbb R}^{\mathbb{N}}}$. This truncated vector is identified as an element of ${\mathbb{R}}^s$, and we denote by ${{\boldsymbol{\mu}}}_s({{\boldsymbol{y}}}_s) := \bigotimes_{j=1}^s \mu_j(y_j)$ the corresponding product measure. 

The main  result on dimension truncation error for functions having a pointwise Gevrey $\delta$-regularity is the following.

\begin{theorem}[Truncation {dimension} error bound] 
	\label{thm:general trunc theorem_revised}
		Assume that  for $s \in {\mathbb N}$,
		 the first-order moment of the measures $\mu_j$, $j \ge s+1$ has vanished, i.e.,
		\begin{align}\label{mean free}
			\int_{Y} { y }_j \, \mu_j({\rm d} { y }_j)
			=
			0 \quad \text{for all } j \geq s+1.
		\end{align}
	{Let $F: Y^{\mathbb{N}} \to {\mathbb{R}}$ be a strongly ${\boldsymbol{\mu}}$-integrable function.}	Assume that there exist an integer $k\geq 2$, constants $\delta\geq 1$, {$p \in (0,1)$} and
	\begin{itemize}
		\item a  function $C_F({ {\boldsymbol{y}} }) = \prod_{j {\in {\mathbb N}}} C_j({ y }_j)$, where $C_j({ y }_j) > 0$ for all ${ y }_j \in Y$,
		\item {a sequence of positive functions} ${\boldsymbol{\beta}}({ {\boldsymbol{y}} }) = (\beta_j({ y }_j))_{j\in {\mathbb{N}}}$,
		\item {a sequence of positive numbers} $\boldsymbol{\mathsf{b}} = ({\mathsf{b}}_j)_{j \in {\mathbb{N}}}$,
	\end{itemize}
	such that  there hold the following conditions on the derivatives of $F$  for all ${ {\boldsymbol{y}} } \in Y^{\mathbb{N}}$ and all and multi-indices $\boldsymbol{\nu} \in \mathcal F$:
	\begin{enumerate}[label={\rm (\roman*)}]
		\item {$s$-Tailed pointwise Gevrey $\delta$-regularity}: {For every $\boldsymbol{\nu} \in \mathcal F_{s}^k$,
		\begin{align}\label{QMC lemma assumption_deriv}
			{\left|\partial^{\boldsymbol{\nu}} F ({ {\boldsymbol{y}} }) \right|}
			\leq
			C_F({ {\boldsymbol{y}} }) \,\boldsymbol{\beta}({ {\boldsymbol{y}} })^{\boldsymbol{\nu}}  \, \boldsymbol{\mathsf{b}}^{\boldsymbol{\nu}} ({\left|\boldsymbol{\nu}\right|}!)^{\delta},
		\end{align}
		where $\mathcal F_{s}^k := \left\{\boldsymbol{\nu}\in\mathcal F: \sum_{j\geq s+1}\nu_j\leq k+1\right\}$.}
		\item {Integrability for $C_F$:} 
		\begin{align}\label{Integrability C_F}
		\widetilde C_F:= \prod_{j{\in {\mathbb N}}}
		\max \left\{1,\int_Y \sup_{t\in[0,1]} C_j(t { y }_j) \mu_j({{\rm d}}{ y }_j)\right\} < \infty.
	\end{align}
		\item Finite weighted mixed moments: 
		\begin{align}\label{weighted moments}
		M_k:=\sup_{j\in{\mathbb{N}}}
		\max_{1\leq m\leq k+1}  \sup_{t\in[0,1]} \int_Y C_j(t { y }_j) \beta_j(t { y }_j)^m {\left|{ y }_j\right|}^m\mu_j({{\rm d}}{ y }_j)<\infty.
	\end{align}
		\item {Sparsity of  $\boldsymbol{\mathsf{b}}$}: The $s$-tail sequence of $\boldsymbol{\mathsf{b}}$, denoted by $\overline{\boldsymbol{\mathsf{b}}}_s:=\left\{\overline{{\mathsf{b}}}_j\right\}_{j\in{\mathbb{N}}}=\left\{{{\mathsf{b}}}_{s+j}\right\}_{j\in{\mathbb{N}}}$, {is $p$-summable,} i.e.,  {$\left\|{\overline{\boldsymbol{\mathsf{b}}}_s}\right\|_{\ell^p} <\infty$}, and $\left\|{\overline{\boldsymbol{\mathsf{b}}}_s}\right\|_{\ell^1}{\leq \tfrac{1}{2}\, 
			 M_k^{-1}}$.
	\end{enumerate}
	Then, the truncation {dimension error of the approximation of  the  integral $I(F)$ by {$I_s(F)$}} is bounded by
	\begin{align}\label{equ: truncation error_bound}
		{\left|I(F) - I_s(F)\right|}
		=
		{\left|\int_{Y^{\mathbb{N}}} (F({ {\boldsymbol{y}} })-F({ {\boldsymbol{y}} }_s){\boldsymbol{\mu}}({{\rm d}}{\boldsymbol{y}})\right|}
		\leq
		C_{\delta}\, s^{{-\tau}},
	\end{align}
	where the constant $C_{\delta}$ and $\tau$ are defined as
	{
	\begin{equation} 	\label{gamma def}
		\tau :=\min\left\{\frac{2}{p}-1,(k+1)\left(\frac{1}{p}-1\right)\right\},
	\end{equation}
	and
	\begin{equation} \label{C_delta def}
		C_{\delta}:=
	\widetilde C_F 
	((k+1)!)^{\delta-1}
	\max\left\{	8{M_k^2 	\left\|{\boldsymbol{\mathsf{b}}}\right\|_{\ell^p}^2},M_k^{k+1} \left\|{\boldsymbol{\mathsf{b}}}\right\|_{\ell^p}^{k+1}\right\}.
	\end{equation}
}
\end{theorem}

\begin{remark}\label{rem:nonzero_mean}
	We point out that the zero-mean assumption \eqref{mean free} is made primarily for notational convenience and can be readily relaxed. If the expected values $c_j := \int_Y y_j \mu_j({{\rm d}}{ y }_j)$ are non-zero for some $j\in {\mathbb{N}}$, 
	{Theorem~\ref{thm:general trunc theorem_revised}} remains fully applicable via a standard coordinate shift. Specifically, by defining the mean sequence $\boldsymbol{c} := (c_j)_{j\in{\mathbb{N}}}$ and the centered variables $\widetilde{{\boldsymbol{y}}} := {\boldsymbol{y}} - \boldsymbol{c}$, one simply applies the theorem to the translated function $\widetilde{F}(\widetilde{{\boldsymbol{y}}}) := F(\widetilde{{\boldsymbol{y}}} + \boldsymbol{c})$.
\end{remark}

\begin{remark}\label{rem:sequence_factorization}
	We note that in the assumptions of this theorem, the {$s$-tailed Gevrey $\delta$-regularity} \eqref{QMC lemma assumption_deriv} is formulated by factorizing the scaling sequence into two separated components: a parameter-dependent sequence ${\boldsymbol{\beta}}(\boldsymbol{y})$ and a parameter-independent magnitude sequence $\boldsymbol{\mathsf{b}}$. While it is theoretically possible to absorb both components into a {single sequence 
	as in Definition~\ref{def:Gevrey-Class-Pointwise},} such a {separation} would necessitate imposing an explicit decay rate directly on the constants $M_k$. Therefore, this decoupling is deliberately retained to streamline the subsequent analysis.
\end{remark}

\begin{remark}\label{rem:recover result}
		It is worth emphasizing that our result seamlessly covers Theorem 4.1 of \cite{GuthKaarnioja2024} in the Gevrey regularity setting (i.e., when $C_F$ and ${\boldsymbol{\beta}}$ in \eqref{QMC lemma assumption_deriv} are independent of ${\boldsymbol{y}}$). While our proof  extends their arguments, the key difference occurs in the last step: by utilizing the geometric series \eqref{eq:geometric series}, we obtain a sharper constant $C_\delta$ compared to \cite{GuthKaarnioja2024}.
\end{remark}

{
\begin{proof} 
	{Let us fix an arbitrary ${ {\boldsymbol{y}} } \in Y^{\mathbb{N}}$.}
	Applying Taylor's formula {at} the truncated point ${ {\boldsymbol{y}} }_s$, we obtain
	\begin{align*}
		F({ {\boldsymbol{y}} })
		&=
		F({ {\boldsymbol{y}} }_s)
		+
		\sum_{i\geq s+1}
		\partial^{{\boldsymbol{e}}_i} F({ {\boldsymbol{y}} }_s)\, { y }_i
		+
		\sum_{j=2}^{k}
		\sum_{{\left|\boldsymbol{\nu}\right|}=j \atop \boldsymbol{\nu}_s=\boldsymbol{0}}
		\frac{1}{\boldsymbol{\nu}!}
		\partial^{\boldsymbol{\nu}} F({ {\boldsymbol{y}} }_s)\, { {\boldsymbol{y}} }^{\boldsymbol{\nu}}\\
		&+
		\sum_{{\left|\boldsymbol{\nu}\right|}=k+1 \atop \boldsymbol{\nu}_s=\boldsymbol{0}}
		\frac{k+1}{\boldsymbol{\nu}!}
		{ {\boldsymbol{y}} }^{\boldsymbol{\nu}}
		\int_0^1(1-{t})^k  \partial^{\boldsymbol{\nu}}  F ({t}{ {\boldsymbol{y}} }+(1-{t}) { {\boldsymbol{y}} }_s)\,{{\rm d}}{t} 
		.
	\end{align*}
 Integrating both sides with respect to ${ {\boldsymbol{y}} }$ over $Y^{\mathbb{N}}$, the difference $I(F) - I_s(F)$ becomes
	\begin{align}
		I(F)-I_s(F)
		&=
		\sum_{j\geq s+1}
		\int_{Y^{\mathbb{N}}} 
		\partial^{{\boldsymbol{e}}_j} F({ {\boldsymbol{y}} }_s)\, { y }_j \, {\boldsymbol{\mu}}({{\rm d}} { {\boldsymbol{y}} })
		+
	\sum_{j=2}^{k}
	\sum_{{\left|\boldsymbol{\nu}\right|}=j \atop \boldsymbol{\nu}_s=\boldsymbol{0}}
	\frac{1}{\boldsymbol{\nu}!}
		\int_{Y^{\mathbb{N}}} 
	\partial^{\boldsymbol{\nu}} F({ {\boldsymbol{y}} }_s)\, { {\boldsymbol{y}} }^{\boldsymbol{\nu}}  {\boldsymbol{\mu}}({{\rm d}} { {\boldsymbol{y}} })\notag\\
	&+
	\sum_{{\left|\boldsymbol{\nu}\right|}=k+1 \atop \boldsymbol{\nu}_s=\boldsymbol{0}}
	\frac{k+1}{\boldsymbol{\nu}!}
		\int_{Y^{\mathbb{N}}} 
		\int_0^1(1-{t})^k  \partial^{\boldsymbol{\nu}}  F ({t}{ {\boldsymbol{y}} }+(1-{t}) { {\boldsymbol{y}} }_s)\,{{\rm d}}{t} \,
				{ {\boldsymbol{y}} }^{\boldsymbol{\nu}}
		{\boldsymbol{\mu}}({{\rm d}} { {\boldsymbol{y}} })
	.\label{trunc form}
	\end{align}
	By the condition \eqref{mean free} of the theorem, 
	 ${\int_{Y} { y }_i \mu_i( {\rm d} { y }_i)}=0$  for all $i \ge s+1$. Let us analyze the first sum on the right-hand side of \eqref{trunc form}. Using Fubini's theorem to separate the 
	 {$s$-dimensional  integral  with respect to the variable ${\boldsymbol{y}}_s$ and the one-dimensional integrals with respect to the variables  ${ y }_i$ for $j\ge s+1$, we have 	
	 $$
	 \int_{Y^{\mathbb{N}}} 
	\partial^{{\boldsymbol{e}}_j} F({ {\boldsymbol{y}} }_s)\, { y }_j \, {\boldsymbol{\mu}}({{\rm d}} { {\boldsymbol{y}} })
	=
	\int_{Y^s}\partial^{{\boldsymbol{e}}_j} F({ {\boldsymbol{y}} }_s){\boldsymbol{\mu}}_s({{\rm d}} { {\boldsymbol{y}}_s })
	\prod_{i \ge s +1 \atop i \not= j}\int_Y   \mu_i( {\rm d} { y }_i)
	 \int_Y  y_j \, \mu_j( {\rm d} { y }_j)
	=
	0
	$$
}
	 for any $j \ge s+1$. Consequently, the {first sum in \eqref{trunc form}  is equal to zero.}
	
	Next, we examine the second sum in \eqref{trunc form}, which involves integrals of the form 
	\begin{align*}
		\int_{Y^{\mathbb{N}}} \partial^{\boldsymbol{\nu}} F({ {\boldsymbol{y}} }_s)\, { {\boldsymbol{y}} }^{\boldsymbol{\nu}}\, {\boldsymbol{\mu}}( {{\rm d}} { {\boldsymbol{y}} })
		&=
		\int_{Y^s} \partial^{\boldsymbol{\nu}} F({ {\boldsymbol{y}} }_s)\, {{\boldsymbol{\mu}}_s( {{\rm d}} { {\boldsymbol{y}}_s })
		\int_{Y^{{\mathbb{N}}\setminus{\{1:{s}\}}}} { {\boldsymbol{y}} }^{\boldsymbol{\nu}}
		{\boldsymbol{\mu}}(  {{\rm d}} { {\boldsymbol{y}} })}
	\end{align*}
	 for {all multi-indices ${{\boldsymbol{\nu}}}$ such that} $2\leq |{{\boldsymbol{\nu}}}| \leq k$ and $\nu_j=0$ for all $j\leq s$.
	If the multi-index $\boldsymbol{\nu}$ contains any component $\nu_j = 1$ for some $j \ge s+1$, the corresponding factor in the product becomes 	${\int_{Y} ({ y }_j)^{1} \, \mu_j({{\rm d}} { y }_j)}=0$ {due to the condition \eqref{mean free},} causing the entire term to vanish. Therefore, we now provide a rigorous bound for each term in the second sum {associated with a nonzero multi-index ${{\boldsymbol{\nu}}}$ satisfying the following conditions:}
	\begin{itemize}
		\item ${\left|{\boldsymbol{\nu}}\right|}\leq k$.
		\item $\nu_j=0$ for all $j\leq s$.
		\item $\nu_j\neq 1$ or $\nu_j\in\left\{0,2,3,\dots\right\}$ for all $j\geq s+1$.
	\end{itemize}
The product structure of ${{\boldsymbol{\mu}}}$ and {the assumption on $s$-tailed Gevrey regularity \eqref{QMC lemma assumption_deriv}} at ${\boldsymbol{y}}_s$ yield:\\
\resizebox{\textwidth}{!}{
	\begin{minipage}{\textwidth}
		\begin{center}
	\begin{align*}
	{\left|\int_{Y^{\mathbb{N}}} 
		\partial^{\boldsymbol{\nu}} F({ {\boldsymbol{y}} }_s)\, { {\boldsymbol{y}} }^{\boldsymbol{\nu}} \,  
		{{\boldsymbol{\mu}}( {\rm d} { {\boldsymbol{y}} })}\right|}
	&\leq
		\int_{Y^s} C_F({ {\boldsymbol{y}} }_s) {\boldsymbol{\beta}}({ {\boldsymbol{y}} }_s)^{\boldsymbol{\nu}} \boldsymbol{\mathsf{b}}^{\boldsymbol{\nu}}
({\left|\boldsymbol{\nu}\right|}!)^{\delta}\, {{\boldsymbol{\mu}}_s( {{\rm d}} { {\boldsymbol{y}}_s })}
\int_{Y^{{\mathbb{N}}\setminus{\{1:{s}\}}}} {\left|{ {\boldsymbol{y}} }^{\boldsymbol{\nu}}\right|}
{\boldsymbol{\mu}}(  {{\rm d}} { {\boldsymbol{y}} })\\
&=
\boldsymbol{\mathsf{b}}^{\boldsymbol{\nu}}
({\left|\boldsymbol{\nu}\right|}!)^{\delta}
\prod_{ j \notin \operatorname{supp}(\boldsymbol{\nu})}
\int_{Y}
C_j(\widetilde{ y }_{j})
\mu_j({{\rm d}}{ y }_j)
\prod_{i \in \operatorname{supp}(\boldsymbol{\nu})}
\int_{Y}
C_i(0)
\beta_i(0)^{\nu_i}
{{\left|{ y }_i\right|}^{\nu_i}
} \mu_i({{\rm d}}{ y }_i),
\end{align*}
	\end{center}
\end{minipage}
}
where $\widetilde { y }_j$ denotes the component of the sequence ${ {\boldsymbol{y}} }_s$, i.e. 
$$\widetilde { y }_j= \left\{\begin{matrix}
	{ y }_j &\text{ if }	1\leq j\leq s\\
	0 &\text{ otherwise}
\end{matrix}\right. .$$
Noting that $C_j(\widetilde { y }_j) \leq  \max\left\{ C_j(0), C_j({ y }_j)\right\}\leq \max_{t \in [0,1]} C_j(t \cdot { y }_j)$, {by using the assumptions \eqref{Integrability C_F} and \eqref{weighted moments} we can continue the above estimations  as follows:}
	\begin{align*}
	{\left|\int_{Y^{\mathbb{N}}} 
		\partial^{\boldsymbol{\nu}} F({ {\boldsymbol{y}} }_s)\, { {\boldsymbol{y}} }^{\boldsymbol{\nu}} \,  
		{{\boldsymbol{\mu}}( {\rm d} { {\boldsymbol{y}} })}\right|}
	&\leq
	\boldsymbol{\mathsf{b}}^{\boldsymbol{\nu}}
	({\left|\boldsymbol{\nu}\right|}!)^{\delta}
	\widetilde C_F
	\prod_{i \in \operatorname{supp}(\boldsymbol{\nu})}
	M_k
	\leq
		\widetilde C_F
		M_k^{{\left|\boldsymbol{\nu}\right|}}
			\boldsymbol{\mathsf{b}}^{\boldsymbol{\nu}}
		({\left|\boldsymbol{\nu}\right|}!)^{\delta},
\end{align*}
where we used $M_k \ge 1$ and $|\text{supp}(\boldsymbol{\nu})| \le |\boldsymbol{\nu}|$.

For the third sum on the right hand side of \eqref{trunc form}, the Taylor remainder where $|{\boldsymbol{\nu}}| = k+1$ and $\text{supp}({\boldsymbol{\nu}}) \subseteq \{s+1, s+2, \dots\}$, {by applying  the triangle inequality for integrals and the assumption (1) on $s$-tailed Gevrey regularity, we derive}
\begin{align*}
		&{\left|\int_{Y^{\mathbb{N}}} 
	\int_0^1(1- {t})^k  \partial^{\boldsymbol{\nu}}  F ( {t}{ {\boldsymbol{y}} }+(1- {t}) { {\boldsymbol{y}} }_s)\, {{\rm d}} {t} 
	\,
	{ {\boldsymbol{y}} }^{\boldsymbol{\nu}}
	\,
	{\boldsymbol{\mu}}(d { {\boldsymbol{y}} })\right|}\\
&\leq
\int_{Y^{\mathbb{N}}} 
\int_0^1 \left(1- {t}\right)^k  {\left|\partial^{\boldsymbol{\nu}}  F ( {t}{ {\boldsymbol{y}} }+(1- {t}) { {\boldsymbol{y}} }_s)\right|} \, {{\rm d}} {t} 
\,
{\left|{ {\boldsymbol{y}} }\right|}^{\boldsymbol{\nu}}
\,
{\boldsymbol{\mu}}( {{\rm d}} { {\boldsymbol{y}} })	\\
&
\leq
 \boldsymbol{\mathsf{b}}^{\boldsymbol{\nu}}
({\left|\boldsymbol{\nu}\right|}!)^{\delta}
\int_{Y^{\mathbb{N}}} 
\int_0^1
\left(1- {t}\right)^k 
C_F ( {t}{ {\boldsymbol{y}} }+(1- {t}) { {\boldsymbol{y}} }_s) {\boldsymbol{\beta}} ( {t}{ {\boldsymbol{y}} }+(1- {t}){ {\boldsymbol{y}} }_s)^{\boldsymbol{\nu}}  {{\rm d} t}\,
{\left|{ {\boldsymbol{y}} }\right|}^{\boldsymbol{\nu}}
\,
{\boldsymbol{\mu}}(d { {\boldsymbol{y}} }).
\end{align*}
Using the product structure of $C_F$ and $\boldsymbol{\mu}$, we bound {the  integral in the right-hand side} by:\\
\resizebox{\textwidth}{!}{
	\begin{minipage}{\textwidth}
\begin{align*}
		&
\int_0^1 	\left(1- {t}\right)^k  {{\rm d}} {t}
\prod_{i= 1}^s \int_Y
C_i({ y }_i) \mu_i({{\rm d}}{ y }_i)
\prod_{j\geq s+1 \atop j \notin \operatorname{supp}(\boldsymbol{\nu})} \int_Y
\max_{t\in [0,1]} C_j(t { y }_j) \mu_j({{\rm d}}{ y }_j)
\prod_{i \in \operatorname{supp}(\boldsymbol{\nu})} \int_Y
\max_{t\in [0,1]} C_i(t { y }_i) \beta_i(t { y }_i)^{\nu_i} {\left|{ y }_i\right|}^{\nu_i} \mu_i({{\rm d}}{ y }_i)
\\
&\qquad\qquad\qquad\leq
\int_0^1 	\left(1- {t}\right)^k  {{\rm d}} {t}
\prod_{j \notin \operatorname{supp}(\boldsymbol{\nu})} \int_Y
\max_{t\in [0,1]} C_j(t { y }_j)\mu_j({{\rm d}}{ y }_j)
\prod_{j \in \operatorname{supp}(\boldsymbol{\nu})}  M_k
\end{align*}
\end{minipage}
}
{Noting the definitions of $\widetilde {C}_F, M_k$ and  the identity $\int_0^1 (1-t)^k \,\mathrm{d}t = \frac{1}{k+1}$, we derive the bounds}
\begin{equation*}
	{\left|\int_{Y^{\mathbb{N}}} 
		\int_0^1(1- {t})^k  \partial^{\boldsymbol{\nu}}  F ( {t}{ {\boldsymbol{y}} }+(1- {t}) { {\boldsymbol{y}} }_s)\, {{\rm d}} {t} 
		\,
		{ {\boldsymbol{y}} }^{\boldsymbol{\nu}}
		\,
		{\boldsymbol{\mu}}(d { {\boldsymbol{y}} })\right|}
		\le \frac{1}{k+1} \widetilde {C}_F M_k^{|{\boldsymbol{\nu}}|} {\boldsymbol{\mathsf{b}}}^{{\boldsymbol{\nu}}} (|{\boldsymbol{\nu}}|!)^{\delta}.
\end{equation*}

By employing the identity  $\frac{(|{\boldsymbol{\nu}}|!)^\delta } {\boldsymbol{\nu}!} = (|{\boldsymbol{\nu}}|!)^{\delta-1} \frac{|{\boldsymbol{\nu}}|!}{{\boldsymbol{\nu}}!}$, inserting these estimates  into \eqref{trunc form}, we obtain\\
\resizebox{\textwidth}{!}{
	\begin{minipage}{\textwidth}
		\begin{align}
		{\left|I(F)-I_s(F)\right|}
		&\leq
		\widetilde C_F 
			((k+1)!)^{\delta-1}
		\left(
		\sum_{j=2}^{k}
		\sum_{\substack{{{\left|\boldsymbol{\nu}\right|}=j}\\{ \nu_i= 0, \forall i \leq s}\\{\nu_i\neq 1, \forall i \geq s+1}}}
			\frac{{\left|\boldsymbol{\nu}\right|}!}{\boldsymbol{\nu}!}
		M_k^{{\left|\boldsymbol{\nu}\right|}}\boldsymbol{\mathsf{b}}^{\boldsymbol{\nu}}
		+
		\sum_{\substack{{{\left|\boldsymbol{\nu}\right|}=k+1}\\{ \nu_i= 0, \forall i \leq s}}}
			\frac{{\left|\boldsymbol{\nu}\right|}!}{\boldsymbol{\nu}!}
		M_k^{{\left|\boldsymbol{\nu}\right|}}\boldsymbol{\mathsf{b}}^{\boldsymbol{\nu}}\right)\notag
		\\&
	{	=: \widetilde C_F 
		((k+1)!)^{\delta-1} (\mathfrak{A}_1 + \mathfrak{A}_2),}
		\label{trunc form 2}
	\end{align}
	\end{minipage}
}
	where $\mathfrak{A}_1$ and $\mathfrak{A}_2$ {are} the sums over $2 \le |{{\boldsymbol{\nu}}}| \le k$ and $|{{\boldsymbol{\nu}}}| = k+1$, respectively.
	
We bound the term $\mathfrak{A}_1$ by applying the identity \eqref{eq:geometric series} established in Section \ref{sec: multiindex}. Taking into account that $\nu_i(\nu_i-1) = 0$ for $\nu_i \in \{0,1\}$ and $\nu_i(\nu_i-1) \geq 2$ for $\nu_i \geq 2$, it follows that
\begin{align*}
\mathfrak{A}_1
&\leq
\sum_{j=2}^{k}
\sum_{\substack{{{\left|\boldsymbol{\nu}\right|}=j}}}
\frac{{\left|\boldsymbol{\nu}\right|}!}{\boldsymbol{\nu}!}
\left(\frac{1}{2}\sum_{i{\in {\mathbb N}}} \nu_i(\nu_i-1)\right)
M_k^{{\left|\boldsymbol{\nu}\right|}}\overline{\boldsymbol{\mathsf{b}}}_s^{\boldsymbol{\nu}}
\\
&\leq
\frac{1}{2}
\sum_{i{\in {\mathbb N}}}
\sum_{j\geq 0}
\sum_{\substack{{{\left|\boldsymbol{\nu}\right|}=j}}}
\frac{{\left|\boldsymbol{\nu}\right|}!}{\boldsymbol{\nu}!}
\nu_i(\nu_i-1)
M_k^{{\left|\boldsymbol{\nu}\right|}}\overline{\boldsymbol{\mathsf{b}}}_s^{\boldsymbol{\nu}}
\\
&=
\frac{1}{2}
\sum_{i{\in {\mathbb N}}}
\frac{2 M_k^2 \overline{{\mathsf{b}}}_{i}^2}{(1-M_k\sum_{j{\in {\mathbb N}}}\overline{{\mathsf{b}}}_j)^3}
=
\frac{ M_k^2 }{(1-M_k{\left\|{\overline{\boldsymbol{\mathsf{b}}}_s}\right\|_{\ell^1}})^3}
\sum_{i{\in {\mathbb N}}}
\overline{{\mathsf{b}}}_{i}^2
\end{align*}
{By applying Stechkin's lemma (see, e.g., \cite[Lemma 7.4.1]{DungTemlyakovUllrich2018}) we have  
$$
\left\|{\overline{\boldsymbol{\mathsf{b}}}_s}\right\|_{\ell^2}^2 \leq\left\|{{\boldsymbol{\mathsf{b}}}_s}\right\|_{\ell^p}^2 s^{-2/p+1}.
$$
Hence}, together with  $\left\|{\overline{\boldsymbol{\mathsf{b}}}_s}\right\|_{\ell^1}< \tfrac{1}{2} M_k^{-1}$ we obtain the bound for $\mathfrak{A}_1$
\begin{align*}
	\mathfrak{A}_1
	\leq
	\frac{M_k^2\left\|{\boldsymbol{\mathsf{b}}}\right\|_{\ell^p}^2}{(1-M_k\left\|{\overline{\boldsymbol{\mathsf{b}}}_s}\right\|_{\ell^1})^3}\,
	\,s^{-\frac{2}{p}+1}
	\leq
		8{M_k^2\left\|{\boldsymbol{\mathsf{b}}}\right\|_{\ell^p}^2}\,
	\,s^{-\frac{2}{p}+1}.
\end{align*}
Turning to the estimation of $\mathfrak{A}_2$, a direct application of the multinomial theorem \eqref{eq:multinomial} yields
\begin{align*}
	{\mathfrak{A}_2
	=
	\sum_{\substack{{{\left|\boldsymbol{\nu}\right|}=k+1}}}
		\frac{{\left|\boldsymbol{\nu}\right|}!}{\boldsymbol{\nu}!}
		M_k^{{\left|\boldsymbol{\nu}\right|}}\overline{\boldsymbol{\mathsf{b}}}_s^{\boldsymbol{\nu}}}
		=
		M_k^{k+1}\left(\sum_{i{\in {\mathbb N}}}\overline{{\mathsf{b}}}_i\right)^{k+1}
		\leq
		M_k^{k+1}
		\left\|{\boldsymbol{\mathsf{b}}}\right\|_{\ell^p}^{k+1} \, s^{-\frac{k+1}{p}+(k+1)}.
\end{align*}
{In the last step, we use Stechkin's lemma, which yields the inequality  $\left\|{\overline{\boldsymbol{\mathsf{b}}}_s}\right\|_{\ell^1}
 \leq\left\|{{\boldsymbol{\mathsf{b}}}_s}\right\|_{\ell^p} s^{-1/p+1}$.}  
 
 Finally, inserting the bounds for $\mathfrak{A}_1$ and $\mathfrak{A}_2$ into \eqref{trunc form 2} and taking the maximum of the two decay rates as {$\tau$ in \eqref{gamma def}} and defining $C_\delta$ as in \eqref{C_delta def} {complete} the proof.
 
 \hfill
\end{proof}
}
 
\section{Error analysis for quasi-Monte Carlo {quadrature}} 
\label{sec:Error Analysis for Quasi-Monte Carlo methods} 
{Consider the sequence of {standard} Gaussian probability measures $\{\mu_j\}_{j \in {\mathbb{N}}}$ on ${\mathbb R}$, each associated with the   density function 
$w(y_j):= \exp\left(-y_j^2/2\right)/(\sqrt{2\pi}) $. The normal infinite tensor product probability measure ${\boldsymbol{\mu}}$ on ${\mathbb R}^{\mathbb{N}}$ is defined by the formula \eqref{bmu}.
Let $F: {\mathbb R}^{\mathbb{N}} \to {\mathbb{R}}$ be a strongly ${\boldsymbol{\mu}}$-integrable function. In the previous section,  under certain conditions we gave an analysis of the error of approximation of the integral $I(F)$ by the dimension truncated integral
\begin{align}\label{trunc integral}
	I_s(F) =	
	\int_{{{\mathbb{R}}}^s} F(\boldsymbol{y}_s) \, {\prod_{j=1}^s w(y_j)} \, {\rm d} \boldsymbol{y}_s.
\end{align}
In this section, we consider the approximation of $I_s(F)$ by QMC methods.

  We denote by {${\mathfrak{w}}$} the cumulative distribution function, i.e., {${{\mathfrak{w}}}'(z)=w(z)$} for all $z\in {\mathbb{R}}$. 
Moreover, we denote by ${\mathfrak{w}}^{-1}$ the inverse distribution function of ${\mathfrak{w}}$ such that ${\mathfrak{w}}({\mathfrak{w}}^{-1}(y))=y$ for all $y\in (0,1)$ and $\Phi^{-1}({\boldsymbol{z}})=\left({\mathfrak{w}}^{-1}(z_j)\right)_{j \in {\mathbb N}}$ for all ${\boldsymbol{z}}\in {\mathbb{R}}^s$.} {To apply QMC rules}, we transform the integral over ${\mathbb{R}}^s$ into an integral over the bounded domain ${{\mathbb I}^s}:=(0,1)^s$ as
\begin{align*}
 I_s(F) =	
	\int_{{\mathbb I}^s}  F(\Phi^{-1}({\boldsymbol{z}}_s) )\, d \boldsymbol{z}_s.
\end{align*}

Let $Q^{\Delta}_{s,n}(F)$ denote the QMC quadrature rule constructed via a randomly shifted lattice rule. This method is characterized by a generating vector $\boldsymbol{\mathsf{z}}_{s} \in {\mathbb{N}}^s$ and a random shift $\Delta$ uniformly distributed over the unit cube $(0,1)^s$. For a total of $n$ quadrature points, the estimator is defined as
\begin{align}\label{QMC quad def}
	Q^{\Delta}_{s,n}(F) := \frac{1}{n} \sum_{j=1}^n F\left(\Phi^{-1} \left(\left\{ \frac{j\, \boldsymbol{\mathsf{z}}_{s}}{n} + \Delta \right\}\right)\right),
\end{align}
where $\{ \cdot \}$ denotes the fractional part applied component-wise. Notice that $Q^{\Delta}_{s,n}(F)$ depends on the random shift and therefore is a random variable. A popular measure of accuracy is the root mean square error 
\begin{align}\label{RMSE}
{\rm RMSE} = \sqrt{\mathbb{E}|I(F)-Q^{\Delta}_{s,n}(F)|^2},
\end{align}
where $\mathbb E$ is the expectation with respect to the random shifts $\Delta$. 
Let $\boldsymbol{\gamma} = (\gamma_{\mathfrak{u}})_{\mathfrak{u}\subseteq { \{1:{s}\}} }$ be a {family} of positive weights. {For each {$j\in {\mathbb N}$}, let  $\psi_j : {\mathbb{R}} \mapsto {\mathbb{R}}^+$ be a positive and continuous weight function ({not necessarily a probability density function}). } We define the weighted Sobolev space of mixed first order derivatives 
$\mathcal{W}_{\boldsymbol{\gamma}}({{\mathbb{R}}}^s)$ as the collection of all functions ${G}: {{\mathbb{R}}}^s \mapsto {\mathbb{R}}$ such that\\
\resizebox{\textwidth}{!}{
	\begin{minipage}{\textwidth}
\begin{align}\label{weighted norm def}
	\left\|{{G}}\right\|_{\mathcal{W}_{\boldsymbol{\gamma}}({{\mathbb{R}}}^s)}^2
	:=
	\sum_{\mathfrak{u} \subseteq {\{1:{s}\}} }
	\frac{1}{\gamma_{\mathfrak{u}}} 
	\int_{{{\mathbb{R}}}^{{\left|\mathfrak{u}\right|}}}
	\left(	\int_{{{\mathbb{R}}}^{{\left|\bar{\mathfrak{u}}\right|}}}
		\frac{\partial^{{\left|\mathfrak{u}\right|}} {G}}{\partial \boldsymbol{\xi}_{\mathfrak{u}}}(\boldsymbol{\xi})
		\prod_{j\in\operatorname{supp}\left\{\bar{\mathfrak{u}}\right\}}w(\xi_j)\,
		d\boldsymbol{\xi}_{\bar{\mathfrak{u}}}
	\right)^2
	\prod_{j\in\operatorname{supp}\left\{{\mathfrak{u}}\right\}} \psi_j^2(\xi_j)
	d\boldsymbol{\xi}_{\mathfrak{u}} < \infty.
\end{align}
\end{minipage}
}
Here $\bar{\mathfrak{u}}:= { \{1:{s}\}}\setminus \mathfrak{u} $ and $\frac{\partial^{{\left|\mathfrak{u}\right|}} {G}}{\partial \boldsymbol{\xi}_{\mathfrak{u}}}$ denotes the mixed first derivatives of ${G}$ with respect to the variable $\boldsymbol{\xi}_{\mathfrak{u}}=(\xi_j)_{j\in \mathfrak{u}}$. For the analysis from \cite{KuoSloan2010,Nichols2014} to hold, we need $\psi_j^2$ to decay slower than the standard normal density $w$. For example, we can choose 
\begin{align}\label{psi def}
	{\psi_j(\xi)=\exp(-\tau_j {\left|\xi\right|})} 
\end{align}
with the value of $\tau_j>0$ to be specified later.

The weight {family} $(\gamma_{\mathfrak{u}})_{\mathfrak{u} \subseteq { \{1:{s}\}}}$ is associated with each subset of the variables to moderate its relative importance with respect to the other subsets. With an appropriate choice of the weight we can derive an error bound, which is independent of the dimension~$s$. Moreover, we need some structure
of the weight for the component-by-component (CBC) construction cost to be feasible, see e.g. \cite{KuNu16,KuoSchwabSloan2012,KuoSchwabSloan2013}. Different types of weights have been considered depending on the problem and the estimation of $\frac{\partial^{{\left|\mathfrak{u}\right|}} F}{\partial \boldsymbol{\xi}_{\mathfrak{u}}}$.

To proceed with the QMC quadrature error analysis, we require an error bound for QMC quadratures of the form \eqref{QMC quad def}.{ A result of this type was established in \cite[Theorem 15]{Graham15} (see also \cite[Theorem 8]{Nichols2014} for a more general treatment) and is stated as follows.}

\begin{lemma} \label{lemma:QMC error} 
	Let $F \in \mathcal{W}_{\boldsymbol{\gamma}}({\mathbb{R}}^s)$, where the weight functions $\psi_j$ are defined as in \eqref{psi def}. For any given $s, n \in {\mathbb N}$, a randomly shifted lattice rule with $n$ points in $s$ dimensions can be constructed via the component-by-component (CBC) algorithm such that, for all $\vartheta \in (1/2,1]$,
	\begin{equation}\label{QMC error}
		\sqrt{\mathbb E\left({\left|I_s(F)-Q^{\Delta}_{s,n}(F)\right|}^2\right)}
		\leq
		\left(\sum_{\mathfrak{u}\subseteq { \{1:{s}\}}}
			\gamma_{\mathfrak{u}}^{\vartheta} \prod_{j\in \mathfrak{u}}\varrho_j(\vartheta)
		\right)^\frac{1}{2\vartheta}
		\varphi(n)^{-\frac{1}{2\vartheta}}
		\left\|{F}\right\|_{\mathcal{W}_{\boldsymbol{\gamma}}({{\mathbb{R}}}^s)},
	\end{equation}
	where the constants $\varrho_j(\vartheta)$ are given by
	\begin{equation}\label{varrrho_j}
		\varrho_{j}(\vartheta)
		:=
		2\left(\frac{\sqrt{2\pi}\exp(\tau_{j}^{2}/\varsigma )}{\pi^{2-2\varsigma }(1-\varsigma )\varsigma }\right)^{\vartheta}\zeta\left(\vartheta+\frac{1}{2}\right)
		\quad \text{with} \quad
		\varsigma :=\frac{2\vartheta-1}{4\vartheta}.
	\end{equation}
	Here, $\mathbb E$ denotes the expectation with respect to the random shift $\Delta$, while $\varphi$ represents the Euler totient function and $\zeta$ denotes the Riemann zeta function. 
\end{lemma}

The following {theorem establishes the convergence rate} of the QMC quadrature \eqref{QMC quad def} for numerical integration of Gevrey $\delta$-regularity functions $F$ with $\delta \geq 1$.

\begin{theorem}\label{QMC theorem}
Let $F: {\mathbb{R}}^{\mathbb{N}} \to {\mathbb{R}}$.
{Assume that there exist constants $\delta \geq 1$,  {a number} $p \in (0,\delta^{-1}]$,} a positive sequence $\boldsymbol{\mathsf{b}} \in \ell^p$, a {sequence of positive functions} $\boldsymbol{\beta}(\boldsymbol{\xi})= (\beta_j(\xi_j))_{j \in \mathbb{N}}$, and {a function 
	$C_F(\boldsymbol{\xi})=\prod_{j\ge 1} C_j(\xi_j)< \infty$ with components $C_j(\xi_j)\geq 1$ such that the following conditions on $F$  are satisfied:}

\begin{enumerate}[label={\rm (\roman*)}]
	\item For all $\boldsymbol{\xi} \in {\mathbb{R}}^{\mathbb{N}}$ and $\boldsymbol{\nu} \in \mathcal F$:
	\begin{align}\label{QMC lemma assumption}
		{\left|\partial^{\boldsymbol{\nu}} F (\boldsymbol{\xi}) \right|}
		\leq
		C_F(\boldsymbol{\xi}) \, \boldsymbol{\beta}(\boldsymbol{\xi})^{\boldsymbol{\nu}} \boldsymbol{\mathsf{b}}^{\boldsymbol{\nu}}({\left|\boldsymbol{\nu}\right|}!)^{\delta}. 
	\end{align}
	\item The function $C_F$ is integrable with respect to {the $s$-dimensional normal measure ${\boldsymbol{\mu}}_s$ and}
	\begin{align}\label{QMC lemma assumption2a}	
	    K:=
		\prod_{j=1}^s {\int_{{\mathbb{R}}} C_j(\xi_j) w(\xi_j) d\xi_j}  < \ \infty.
	\end{align}
	\item The {functions} $C_j$ and $\beta_j$ satisfy 
	\begin{align*}
		\rho:= \max_{j\in {\mathbb{N}}} \left(
			\int_{{\mathbb{R}}}C_j(\xi_j)^2\beta_j (\xi_j)^2 \psi_j^2(\xi_j) {{d}\xi_j}
		\right)^{\frac{1}{2}}
	\  {<  \infty}.
	\end{align*}
\end{enumerate}

Consider the approximation of {the  integral $I_s(F)$ by the randomly shifted lattice quadrature rule $Q^{\Delta}_{s,n}(F)$ with $n=2^m$ points}, $m \in \mathbb N$.

Define the {number} $\vartheta$ as:
\begin{align} \label{vartheta}
	\vartheta :=
	\left\{\begin{matrix}
		\omega & \text{for any } \omega \in (1/2,1), & \text{when } p\in (0,2/(3\delta)],
		\\
		\frac{\delta p}{2-\delta p}, & & \text{when } p\in (2/(3\delta),1/\delta],
	\end{matrix}\right.
\end{align}
and the weights $\gamma_{\mathfrak{u}}$ by:
\begin{align}\label{QMC weight}
	\gamma_{\mathfrak{u}} = \left(({\left|\mathfrak{u}\right|}!)^{\delta} \prod_{j\in \mathfrak{u} } {\frac{\rho {\mathsf{b}}_j}{\sqrt{\varrho_j(\vartheta)}}}\right)^{\frac{2}{1+\vartheta}},
\end{align}
where
{$\varrho_j(\vartheta)$ is defined in \eqref{varrrho_j}}
with a suitable choice of sequence $\left\{\tau_j\right\}_{j\in {\mathbb{N}}}$ such that $\varrho_j(\vartheta)\leq \varrho_{\max}(\vartheta)<\infty$ for all $j\geq 1$ and $\vartheta >0$.

Then {$F \in \mathcal{W}_{\boldsymbol{\gamma}}({\mathbb{R}}^s)$ and} there exists a constant $C_{\gamma,\vartheta}$, independent of {$n$ and $s$ such that} 
\begin{align}\label{equ: QMC error}
	\sqrt{\mathbb E_{\Delta}\left({\left|I_s(F)-Q^{\Delta}_{s,n}(F)\right|}^2\right)}
	\leq
	C_{{\gamma,\vartheta}}^{\frac{1}{2}}\,{n^{-\frac{1}{2\vartheta}}}.
\end{align}
\end{theorem}
\begin{proof}
	The proof follows the lines of \cite[Theorem 4.2] {Gilbert2019} and \cite[Theorem 6.4]{KuoSchwabSloan2012}.
	We start {with estimating the weighted norm \eqref{weighted norm def} of  the function $F$: }\\
	\resizebox{\textwidth}{!}{
		\begin{minipage}{\textwidth}
	\begin{align*}
		\left\|{F}\right\|_{\mathcal{W}_{\boldsymbol{\gamma}}({{\mathbb{R}}}^s)}^2 
		&\leq
		\sum_{\mathfrak{u} \subseteq { \{1:{s}\}} }
	\frac{1}{\gamma_{\mathfrak{u}}} 
	\int_{{{\mathbb{R}}}^{{\left|\mathfrak{u}\right|}}}
	\left(	\int_{{{\mathbb{R}}}^{{\left|\bar{\mathfrak{u}}\right|}}}
	C_F(\boldsymbol{\xi})  \, \boldsymbol{\mathsf{b}}^{\mathfrak{u}} \boldsymbol{\beta}(\boldsymbol{\xi})^{\mathfrak{u}} ({\left|\mathfrak{u}\right|}!)^{\delta}
			\prod_{i\in\operatorname{supp}\left\{\bar{\mathfrak{u}}\right\}}w(\xi_i)\,
		d\boldsymbol{\xi}_{\bar{\mathfrak{u}}}
	\right)^2
		\prod_{j\in\operatorname{supp}\left\{{\mathfrak{u}}\right\}} \psi_j^2(\xi_j)
	d\boldsymbol{\xi}_{\mathfrak{u}}\\
	&=
		\sum_{\mathfrak{u} \subseteq { \{1:{s}\}} }
	\frac{	 ({\left|\mathfrak{u}\right|}!)^{2 \delta} {\boldsymbol{\mathsf{b}}}^{2\mathfrak{u}} }{\gamma_{\mathfrak{u}}} 
	\int_{{{\mathbb{R}}}^{{\left|\mathfrak{u}\right|}}}
	\prod_{j\in \mathfrak{u}}  C_j (\xi_j)^2 
	(\beta_j(\xi_j)^2)
		\left(	\int_{{{\mathbb{R}}}^{{\left|\bar{\mathfrak{u}}\right|}}}
		\prod_{i\in \bar{\mathfrak{u}}} C_i (\boldsymbol{\xi}_{\bar{\mathfrak{u}}}) w(\xi_j)
		d\boldsymbol{\xi}_{\bar{\mathfrak{u}}}
	\right)^2
		\prod_{j\in\operatorname{supp}\left\{{\mathfrak{u}}\right\}} \psi_j^2(\xi_j)
	d\boldsymbol{\xi}_{\mathfrak{u}}\\
	&=
		\sum_{\mathfrak{u} \subseteq { \{1:{s}\}} }
			\frac{	 ({\left|\mathfrak{u}\right|}!)^{2 \delta}  \rho^{2{\left|\mathfrak{u}\right|}} {\boldsymbol{\mathsf{b}}}^{2\mathfrak{u}}}{\gamma_{\mathfrak{u}}} 
					\left(	\int_{{{\mathbb{R}}}^{{\left|\bar{\mathfrak{u}}\right|}}}
				\prod_{i\in \bar{\mathfrak{u}}} C_i (\boldsymbol{\xi}_{\bar{\mathfrak{u}}})
				\prod_{j\in\operatorname{supp}\left\{\bar{\mathfrak{u}}\right\}}w(\xi_j)\,
				d\boldsymbol{\xi}_{\bar{\mathfrak{u}}}
			\right)^2
	\leq
			{K}^2
				\sum_{\mathfrak{u} \subseteq { \{1:{s}\}} }
					\frac{	\Lambda_{\mathfrak{u}}^2	}{\gamma_{\mathfrak{u}}}, 
			\end{align*}
		\end{minipage}
		}\\	
	where
			\begin{align*}
			{ K =
			\prod_{j=1}^s
			{
				\int_{{{\mathbb{R}}}}
		C_j (\xi_j) w_j(\xi_j)
				d\xi_{j}}  < \infty,}
				\qquad
		\Lambda_{\mathfrak{u}}:= ({\left|\mathfrak{u}\right|}!)^\delta \prod_{j\in \mathfrak{u}} {\rho{{\mathsf{b}}}_j}.
	\end{align*}
	{This yields $F \in \mathcal{W}_{\boldsymbol{\gamma}}({\mathbb{R}}^s)$. We apply Lemma~\ref{lemma:QMC error} to establish \eqref{equ: QMC error}.}
	Notice that if $n=2^m$, {for the Euler quotient function,} it holds that $\varphi(n) = \frac{n}{2}$. Therefore, \eqref{QMC error} implies the estimate for the mean-square error
	\begin{align*}
		\mathbb E\left({\left|I_s(F)-Q^{\Delta}_{s,n}(F)\right|}^2\right)
		\leq
		C_{s,\gamma,\vartheta} \, \left(\frac{n}{2}\right)^{-\frac{1}{\vartheta}},
	\end{align*}
	where
	\begin{align*}
		C_{s,\gamma,\vartheta}
		=
			\left(\sum_{\mathfrak{u}\subseteq { \{1:{s}\}}}
			\gamma_{\mathfrak{u}}^{\vartheta} \prod_{j\in \mathfrak{u}}\varrho_j(\vartheta)	
		\right)^\frac{1}{\vartheta}
		\left({{K}}^2 \sum_{\mathfrak{u}\subseteq  { \{1:{s}\}}} \frac{\Lambda_{\mathfrak{u}}^2}{\gamma_{\mathfrak{u}}} \right).		 
	\end{align*}
	{We show} that $C_{s,\gamma,\vartheta}$ can be bounded independently of $s$.
	Following \cite[Lemma 6.2]{KuoSchwabSloan2012} the optimal selection of the weights is given by \eqref{QMC weight}. For this specific weights we have 
	\begin{align}\label{S def}
		C_{s,\gamma,\vartheta}
		= 
		K^2  \, S_{s,\vartheta}^{(1+\vartheta)/\vartheta}, \quad \text{where} \quad
		S_{s,\vartheta}
		:=
		\sum_{\mathfrak{u} \in { \{1:{s}\}}}
		\left(\Lambda_{\mathfrak{u}}^{2\vartheta}\,  \prod_{j\in \mathfrak{u}}\varrho_j(\vartheta)	 \right)^{\frac{1}{1+\vartheta}}.
	\end{align}
	{Thus, it is sufficient to prove} that $S_{s,\vartheta}$ can be bounded independently of $s$, {and, hence, $C_{s, \gamma, \vartheta} \leq C_{\gamma, \vartheta}$ for all $s$.}

	{Since $0<p<\delta^{-1}$ and $\delta \geq 1$ we obtain that}\\
	\resizebox{\textwidth}{!}{
		\begin{minipage}{\textwidth}
	\begin{align*}
		&S_{s,\vartheta}
		=
		\sum_{\mathfrak{u} \subseteq { \{1:{s}\}}}
		\left(\Lambda_{\mathfrak{u}}^{2\vartheta}\, \prod_{j\in \mathfrak{u}}\varrho_j(\vartheta) \right)^{\frac{1}{1+\vartheta}}
		\leq
		\left(
			\sum_{\mathfrak{u} \subseteq { \{1:{s}\}}}
			\left(\Lambda_{\mathfrak{u}}^{2\vartheta}\, \prod_{j\in \mathfrak{u}}\varrho_j(\vartheta) \right)^{\frac{1}{(1+\vartheta)\delta}}
		\right)^{\delta}\\
		&=
		\left(
			\sum_{\mathfrak{u} \subseteq { \{1:{s}\}}}
			({\left|\mathfrak{u}\right|}!)^{\frac{2 \vartheta }{1+\vartheta}}  \prod_{j \in \mathfrak{u}} 
			\left(\rho{{\mathsf{b}}}_j
			\right)^{\frac{2\vartheta}{(1+\vartheta)\delta}} {\varrho_j(\vartheta)}^{\frac{1}{(1+\vartheta)\delta}}
		\right)^{\delta}
		=
		\left(
			\sum_{\mathfrak{u} \subseteq { \{1:{s}\}}}  
			({\left|\mathfrak{u}\right|}!)^{\rho}  \prod_{j \in \mathfrak{u}}
			\left(\rho{{\mathsf{b}}}_j\right)^{\frac{\rho}{\delta}} 
			{\varrho_j(\vartheta)}^{\frac{1}{(1+\vartheta)\delta}}\right)^\delta,
	\end{align*}
	\end{minipage}
}
	where $\rho:= \frac{2\vartheta }{1+\vartheta} < 1$ for $\vartheta < 1$.
	Let 
	$$\eta:= \frac{\rho }{1-\rho} =
	\frac{2\vartheta}{1-\vartheta} 
	\ \ \text{and} \ \
	  \upsilon :=
	\frac{1}{(1+\vartheta)\delta(1-\rho)}
	=
	\frac{1}{(1-\vartheta)\delta}.
	$$
	 We multiply and divide each term in the above estimate by $\prod_{j \in \mathfrak{u}} \alpha_j^{\rho}$ with $\alpha_j > 0$ to be specified 
	later. Then we apply H\"older's inequality with conjugate exponents $\frac{1}{\rho}>1$ and $\frac{1}{1-\rho}>1$ and \cite[Lemma 6.3]{KuoSchwabSloan2012} to obtain
	\begin{equation} \label{Sstheta}
		\begin{split}
			S_{s,\vartheta}
			&\leq
			\left(
				\sum_{\mathfrak{u} \subseteq { \{1:{s}\}}}  ({\left|\mathfrak{u}\right|}!)^{\rho} 
				\left(\prod_{j \in \mathfrak{u}} \alpha_j^\rho\right)
				\left( \prod_{j \in \mathfrak{u}} 
					\left(\frac{(\rho{{\mathsf{b}}}_j)^{\frac{1}{\delta}}}{\alpha_j}\right)^{\rho} 
					\varrho_j(\vartheta)^{\frac{1}{(1+\vartheta)\delta}}\right)\right)^\delta
			\\
			&\leq
			\left(\sum_{{\left|\mathfrak{u}\right|} < \infty}  {\left|\mathfrak{u}\right|}! \prod_{j \in \mathfrak{u}} \alpha_j\right)^{\rho\delta}
			\left(\sum_{{\left|\mathfrak{u}\right|} < \infty}   \prod_{j \in \mathfrak{u}} 
				\left(\frac{(\rho{{\mathsf{b}}}_j)^{\frac{1}{\delta}}}{\alpha_j}\right)^{\eta}  \varrho_j(\vartheta)^{ \upsilon } \right)^{{(1-\rho)\delta}}
			\\
			&\leq
			\left(\frac{1}{1-\sum_{j \in \mathbb{N}} \alpha_j}\right)^{\rho\delta} 
			\exp \left((1-\rho)\delta\,{\varrho_{\max}(\vartheta)^{\upsilon}}
				\sum_{j \in \mathbb{N}}	\left(\frac{(\rho{{\mathsf{b}}}_j)^{\frac{1}{\delta}}}{\alpha_j}\right)^{\eta}\right).
		\end{split}
	\end{equation}
	{Notice that 
	the right-hand side is finite}, if $\sum_{j\geq 1}\alpha_j < 1$ and 
	$\sum_{j\geq 1} 	\big( (\rho{{\mathsf{b}}}_j)^{\frac{1}{\delta}}/\alpha_j\big)^{\eta}$ 
	is finite. In this case this bound is also independent of $s$. We now choose
	\begin{align*}
		\alpha_j := \frac{(\rho{{\mathsf{b}}}_j)^{p}}{\varpi} \text{ for some } \varpi > \rho^p\left\|{(\boldsymbol{\mathsf{b}})}\right\|_{\ell^p}^p.
	\end{align*}
	This clearly implies the first condition  $\sum_{j\geq 1}\alpha_j < 1$. Concerning the second condition, we observe that
	\[
	\sum_{j \in \mathbb{N}}	\left(\frac{(\rho{{\mathsf{b}}}_j)^{\frac{1}{\delta}}}{\alpha_j}\right)^{\eta} 	
	= \varpi^\eta \rho^{(\frac{1}{\delta}-p)\eta} \sum_{j \in \mathbb{N}} 	{{\mathsf{b}}}_j^{(\frac{1}{\delta}-p)\eta}. 
	\] 
	According to \eqref{QMC lemma assumption}, this sum is finite if $(\frac{1}{\delta}-p)\eta \geq p$. Recalling the definition of $\eta$, this is equivalent to
	\begin{align*}
		\vartheta \geq \frac{\delta p}{2-\delta p}.
	\end{align*}
	Since the error estimate \eqref{QMC error} and \eqref{Sstheta} are simultaneously valid for $\vartheta \in (\frac{1}{2},1)$, we choose $\vartheta$ as any fixed value $\omega\in\left(\tfrac{1}{2}, 1 \right)$ when $\delta p\in \left( 0,\tfrac{2}{3}  \right]$, and for $\delta p\in \left(\tfrac{2}{3}  , 1  \right)$, we set $\vartheta = \frac{\delta p}{2-\delta p}$. 
	\hfill
\end{proof}

In \cite[ Theorem 6.2]{GuthKaarnioja2025}, the authors established bounds on the QMC quadrature error \eqref{RMSE} for the dimensionally truncated weak solution of the parametric PDE \eqref{ellip} with random diffusion coefficient $a$
 exhibiting global Gevrey regularity satisfying \eqref{norm-G-regularity}.
Within this same regime, we demonstrate that the QMC method \eqref{QMC quad def} attains a substantially improved convergence rate, even under the weaker assumption of pointwise Gevrey regularity, as stated in the following {theorem}.
\begin{theorem}
	Suppose that the assumptions of Theorem \ref{QMC theorem} are satisfied and $\delta < \frac{3}{2}$. Then, $F \in \mathcal{W}_{\boldsymbol{\gamma}}({\mathbb{R}}^s)$, and there exists a constant $C_{\gamma,\vartheta} > 0$, independent of both the dimension $s$ and the number of points $n$, {such that} 
	\begin{align}\label{equ: QMC error coro}
		\sqrt{\mathbb E_{\Delta}\left({\left|I_s(F)-Q^{\Delta}_{s,n}(F)\right|}^2\right)}
		\leq
		C_{{\gamma,\vartheta}}^{\frac{1}{2}} \, n^{-\frac{1}{2\vartheta}},
	\end{align}
	where the parameter $\vartheta$ is given by
	\begin{align} \label{vartheta coro}
		\vartheta :=
		\begin{cases}
			\omega, & \text{for some } \omega \in (\tfrac{1}{2},\frac{1}{2\delta-1}), 
			\text{ if } p \in \left(0, \tfrac{2}{3}\right], \\[1.5ex]
			\frac{p}{2- p}, & \text{if } p \in \left(\frac{2}{3}, \frac{1}{\delta}\right).
		\end{cases}
	\end{align}
\end{theorem}

\begin{proof}
	{Following the proof of the previous theorem,
	we need to bound the quantity $S_{s,\vartheta}$ defined as  \eqref{S def}, independently of the dimension $s$. We have by \eqref{S def}, }
	\begin{align*}
		S_{s,\vartheta}
		&= {\sum_{\mathfrak{u} \subseteq { \{1:{s}\}}} ({\left|\mathfrak{u}\right|}!)^{\frac{2\vartheta\delta}{1+\vartheta}} \prod_{j\in \mathfrak{u}} \left(\rho{{\mathsf{b}}}_j\right)^{\frac{2\vartheta}{1+\vartheta}} \varrho_j(\vartheta)^{\frac{1}{1+\vartheta}}} \\
		&\leq \sum_{\ell=0}^s (\ell!)^{\frac{2\vartheta\delta}{1+\vartheta}} \left(\varrho_{\max}(\vartheta)\rho^{2\vartheta}\right)^{\frac{\ell}{1+\vartheta}} \sum_{\substack{\mathfrak{u} \subseteq { \{1:{s}\}} \\ {\left|\mathfrak{u}\right|}=\ell}} \prod_{j\in \mathfrak{u}} {{\mathsf{b}}}_j^{\frac{2\vartheta}{1+\vartheta}}.
	\end{align*}
	
	To bound the inner sum, we relate the sum over subsets to a sum over multi-indices ${\boldsymbol{\nu}} \in {\mathbb N}_0^s$. Since ${\boldsymbol{\nu}}! = 1$ whenever $\nu_j \in \{0,1\}$ for all $j$, we can write:
	\begin{align*}
		\sum_{\substack{\mathfrak{u} \subseteq { \{1:{s}\}} \\ {\left|\mathfrak{u}\right|}=\ell}} \prod_{j\in \mathfrak{u}} {{\mathsf{b}}}_j^{\frac{2\vartheta}{1+\vartheta}}
		&= \sum_{\substack{{\boldsymbol{\nu}} \in \{0,1\}^s \\ {\left|{\boldsymbol{\nu}}\right|}=\ell}} \frac{1}{{\boldsymbol{\nu}}!} \prod_{j=1}^s \left({{\mathsf{b}}}_j^{\frac{2\vartheta}{1+\vartheta}}\right)^{\nu_j} 
		\leq \frac{1}{\ell!} \sum_{\substack{{\boldsymbol{\nu}} \in {\mathbb N}_0^s \\ {\left|{\boldsymbol{\nu}}\right|}=\ell}} \frac{\ell!}{{\boldsymbol{\nu}}!} \prod_{j=1}^s \left({{\mathsf{b}}}_j^{\frac{2\vartheta}{1+\vartheta}}\right)^{\nu_j}.
	\end{align*}
	
Applying the multinomial theorem \eqref{eq:multinomial}, this expression simplifies exactly to $\frac{1}{\ell!} \left(\sum_{j=1}^s {\mathsf{b}}_j^{\frac{2\vartheta}{1+\vartheta}}\right)^\ell$. Substituting this back into the bound for $S_{s,\vartheta}$ and extending the outer sum to infinity, we obtain:
	\begin{align*}
		S_{s,\vartheta}
		&\leq \sum_{\ell=0}^s (\ell!)^{\frac{2\vartheta\delta}{1+\vartheta}-1} \left(\varrho_{\max}(\vartheta)\rho^{2\vartheta}\right)^{\frac{\ell}{1+\vartheta}} \left(\sum_{j=1}^s{\mathsf{b}}_j^{\frac{2\vartheta}{1+\vartheta}}\right)^\ell \\
		&\leq \sum_{\ell=0}^\infty (\ell!)^{\frac{2\vartheta\delta}{1+\vartheta}-1} \left(\varrho_{\max}(\vartheta)\rho^{2\vartheta}\right)^{\frac{\ell}{1+\vartheta}} \left(\sum_{j\in{\mathbb{N}}}{\mathsf{b}}_j^{\frac{2\vartheta}{1+\vartheta}}\right)^\ell.
	\end{align*}
	
	Since $F \in \mathcal{W}_{\boldsymbol{\gamma}}({\mathbb{R}}^s)$ requires $S_{s,\vartheta} < \infty$, we analyze the convergence of this series. Recall the necessary condition $p < \frac{1}{\delta}$. We proceed by analyzing two distinct cases based on the value of $p$:
	
	{For the first case where $p \in \left(\frac{2}{3}, \frac{1}{\delta}\right)$,}
	we set $\vartheta = \frac{p}{2-p}$, which implies $\frac{2\vartheta}{1+\vartheta} = p$. Since $p \in \left(\frac{2}{3}, \frac{1}{\delta}\right)$, it follows that $\vartheta \in \left(\frac{1}{2}, \frac{1}{2\delta-1}\right)$. The sum $\sum_{j\in{\mathbb{N}}}{\mathsf{b}}_j^p$ is finite because $\boldsymbol{\mathsf{b}} \in \ell^p$. In this case, the series admits the following upper bound
	\begin{align*}
		S_{s,\vartheta}
		\leq \sum_{\ell=0}^\infty (\ell!)^{p\delta-1} {\varrho_{\max}(\vartheta)}^{\frac{(2-p)\ell}{2}} \rho^{p\ell} \left(\sum_{j\in{\mathbb{N}}}{\mathsf{b}}_j^{p}\right)^\ell.
	\end{align*}
	Because $p\delta < 1$, the exponent of the factorial term is strictly negative ($p\delta - 1 < 0$). By  d'Alembert's ratio test for series convergence, the series in the right-hand side converges to a finite value independent of $s$.
	
	For the second case where $p \leq \frac{2}{3}$,
	we select $\vartheta \in \left(\frac{1}{2}, \frac{1}{2\delta-1} \right)$ such that the exponent $\frac{2\vartheta\delta}{1+\vartheta} - 1 < 0$.
	By the assumption\footnote{It is worth noting that the assumption $\delta < \frac{3}{2}$ is essential here, as the parameter $\vartheta$ governing the QMC error bound for rank-1 lattice rules in Lemma \ref{lemma:QMC error} is restricted to the interval $(\frac{1}{2},1]$.} $\delta < \frac{3}{2}$, we have $2\delta - 1 < 2$, implying $\frac{1}{2\delta-1} > \frac{1}{2}$. 
	Therefore, it is always possible to choose $\vartheta = \omega$ for some $\omega \in \left(\frac{1}{2}, \min\left\{1, \frac{1}{2\delta-1}\right\}\right) =(\tfrac{1}{2},\tfrac{2}{2\delta-1})$.
	For this choice of $\vartheta$, we have $\frac{2\vartheta}{1+\vartheta} > \frac{2(1/2)}{1+1/2} = \frac{2}{3} \geq p$, which guarantees $\sum_{j\in{\mathbb{N}}}{\mathsf{b}}_j^{\frac{2\vartheta}{1+\vartheta}} < \infty$ due to $\boldsymbol{\mathsf{b}} \in \ell^p$. 
	Recall $\frac{2\vartheta\delta}{1+\vartheta} < 1$, this makes the exponent of $\ell!$ negative. An application of  d'Alembert's  ratio test for series convergence confirms that $S_{s,\vartheta} < \infty$ and completes the proof. 
	\hfill
\end{proof}
\section{Total error}\label{sec: total error}
\subsection{Error analysis for finite element methods}
\label{subsec: FEM_error}
{%
	In this subsection, we address the spatial discretization of the parametric elliptic PDE \eqref{ellip} using {finite element methods (FEMs)}.

	We recall a result on approximation   by finite element methods, for detail see, e.g., \cite[Section 4.4]{Brenner2008}. Let $D \subset {\mathbb R}^d$ be a bounded polyhedral domain. 
	Let $r \in {\mathbb N}$ be given and ${\mathcal T}^h$, $0 < h \le 1$,  a regular subdivision of  $D$ with polyhedrons  $T$ such that
	$$
	\max_{T \in {\mathcal T}^h} \operatorname{diam} T \le h \operatorname{diam} D.
	$$
	Let $V_h$  be the finite element space which consists  of those continuous functions on $D$ that restrict to polynomials of order  at most $r$ on each polyhedron $T \in {\mathcal T}^h$. Let ${\mathcal I}_h: C(\bar{D}) \to L^1(D)$ be the global interpolation operator  associated with ${\mathcal T}^h$. 
	By \cite[Theorem 4.4.4]{Brenner2008}, there exists a constant $K_{D,r}$  such that  for any $v \in H^{r}(D)$
	\begin{equation} \label{norm{v- I v}{H^1}}
		\left\|{v- {\mathcal I}_h v}\right\|_{H^1(D)}
		\ \le \
		K_{D,r} \, h^{r-1}	{\left\|{v}\right\|_{H^{r}(D)}}.	
	\end{equation}

	For a given  $\boldsymbol{y} \in {{\mathbb R}^{\mathbb N}}$, the discrete solution $u_h(\boldsymbol{y}) \in V_h$ is defined as the unique solution to the following variational problem: find $u_h(\boldsymbol{y}) \in V_h$ such that
	$$
	\int_{D} a({\boldsymbol{y}}) \nabla u_h({\boldsymbol{y}}) \cdot \nabla v_h \, {{\rm d} {\boldsymbol{x}}}= \langle f, v_h \rangle, \quad \forall v_h \in V_h. 
	$${
		
		For numerical analysis of finite element approximation, we need some auxiliary results on the spatial higher regularity Sobolev norm of the parametric diffusion coefficient and weak solution.
		\begin{lemma}\label{lemma:norm{a(by)}{W^r_{infty,a(by)}}}
			Let $r\in {\mathbb N}$ and suppose that $Z(\cdot, {\boldsymbol{y}}) \in W^r_{\infty}(D)$ for all ${\boldsymbol{y}} \in Y^{\mathbb{N}}$. For the log-normal diffusion coefficient $a({\boldsymbol{y}}) := \exp(Z({\boldsymbol{y}}))$, we have the estimate
			\begin{equation*} 
				\begin{aligned}
					\left\|{a({\boldsymbol{y}})}\right\|_{W^r_{\infty,a({\boldsymbol{y}})}}
					&\le \
					C \left(1 + \left\|{Z({\boldsymbol{y}})}\right\|_{W^r_{\infty}}\right)^r \quad \forall {\boldsymbol{y}} \in Y^{\mathbb{N}},		
				\end{aligned}			
			\end{equation*}		
			where the constant $C > 0$ depends only on $r$ and $d$, and is independent of the parameter sequence ${\boldsymbol{y}}$.
		\end{lemma}
		
		\begin{proof} 
			The proof follows closely the arguments presented in \cite[Lemma 3.20]{DNSZ2023}. We thus omit the details here.
			\hfill
		\end{proof}	
		
		For the proof of the following result see \cite[Corollary 3.19]{DNSZ2023}.
		
		\begin{lemma} \label{lemma:regularity} 
			Let $r \ge 2$ and assume that $D \subset {\mathbb R}^d$ is a bounded domain with either a $C^\infty$-boundary or a convex $C^{r-1}$-boundary. Let $a(\cdot, {\boldsymbol{y}})$ be the parametric diffusion coefficient given by \eqref{def:diffus-coeff}. Suppose that the uniform ellipticity condition \eqref{PDE-ellipticity} holds, $a(\cdot, {\boldsymbol{y}}) \in W^{r-1}_{\infty}(D)$ for all ${\boldsymbol{y}} \in Y^{\mathbb{N}}$, and the source term satisfies $f \in H^{r - 2}(D)$. Then, the corresponding weak solution $u(\cdot, {\boldsymbol{y}})$ to \eqref{ellip} belongs to $H^r(D)$ and satisfies the bound
			\begin{equation*} \label{eq:u_Hr_bound} 
				\left\|{u({\boldsymbol{y}})}\right\|_{H^r(D)} 
				\le \
				C_{d,r} \left\|{f}\right\|_{H^{r - 2}_{a({\boldsymbol{y}})}}
				\left(1 + \left\|{a({\boldsymbol{y}})}\right\|_{W^{r -1}_{\infty,a({\boldsymbol{y}})}}\right)^{r-1} \quad \forall {\boldsymbol{y}} \in Y^{\mathbb{N}},
			\end{equation*}
			where $C_{d,r} > 0$ is a constant depending only on $d$ and $r$.
		\end{lemma} 
	}

	By combining \eqref{norm{v- I v}{H^1}} with C\'ea's Lemma and the regularity results from Lemma~\ref{lemma:regularity}, we obtain the following bound for all ${\boldsymbol{y}} \in {{\mathbb R}^{\mathbb{N}}}$:
	\begin{equation} \label{norm{u- u_h}{V}-a}
		\left\|{u({\boldsymbol{y}})- u_h({\boldsymbol{y}})}\right\|_{V}
		\leq
		C_{d,r} K_{D,r}\,
		h^{r-1} \left\|{f}\right\|_{H^{r - 2}_{a({\boldsymbol{y}})^{-1}}}
		\left(1 + \left\|{a({\boldsymbol{y}})}\right\|_{W^{r -1}_{\infty,a({\boldsymbol{y}})^{-1}}}\right)^{r-1}
	\end{equation}
	which by Lemma~\ref{lemma:norm{a(by)}{W^r_{infty,a(by)}}} can be rewritten as 
	\begin{equation} \label{norm{u- u_h}{V}-Z}
		\left\|{u({\boldsymbol{y}})- u_h({\boldsymbol{y}})}\right\|_{V}
		\leq
		C_{d,r} K_{D,r}\,
		h^{r-1} \left\|{f}\right\|_{H^{r - 2}_{a({\boldsymbol{y}})^{-1}}}
		\left(1 + \left\|{Z({\boldsymbol{y}})}\right\|_{W^{r -1}_{\infty,a({\boldsymbol{y}})^{-1}}}\right)^{r-1}.
	\end{equation}
	We now employ these established pointwise error bounds  to estimate the {numerical integration error over the dimension truncated parameter space ${\mathbb R}^s$.} The following corollary provides the corresponding bound for the finite element approximation.
	\begin{corollary}\label{corollary:FEM error}
		Suppose the assumptions of Lemma~\ref{lemma:regularity} hold. Let $G \in H^{-1}(D)$ be a bounded linear functional, and define the integral operator $I_s : V \to {\mathbb R}$ by
		\begin{equation*}
			I_s(w) := \int_{{\mathbb R}^s} G\big(w(\boldsymbol{y}_s)\big) {\rm d} {\boldsymbol{\mu}}_s({\boldsymbol{y}}_s).
		\end{equation*}
		Assume further that there exists a uniform constant $C_{\mathrm{FEM}} > 0$ such that:\\
		\resizebox{\textwidth}{!}{
			\begin{minipage}{\textwidth}
				\begin{align}\label{cond1}
					C_{\mathrm{FEM}} := C_{d,r} K_{D,r} \left\|{G}\right\|_{H^{-1}(D)} \sup_{s\in{\mathbb{N}}}\int_{{\mathbb R}^s} \left\|{f}\right\|_{H^{r - 2}_{a({\boldsymbol{y}}_s)^{-1}}}
					\left(1 + \left\|{a({\boldsymbol{y}}_s)}\right\|_{W^{r -1}_{\infty,a({\boldsymbol{y}}_s)^{-1}}}\right)^{r-1} {\rm d} {\boldsymbol{\mu}}_s({\boldsymbol{y}}_s) < \infty.
				\end{align}
			\end{minipage}
		}\\
		Then, we have the weak error estimate
		\begin{equation}\label{FEM weakly error bound}
			{\left|I_s(u-u_h)\right|} \leq C_{\mathrm{FEM}} \, h^{r-1}.
		\end{equation}
	\end{corollary}
	
	\begin{proof}
		From the definition of $I_s$ and standard properties of the duality pairing, we first observe that
		\begin{align*}
			{\left|I_s(u-u_h)\right|}
			&= {\left|\int_{{\mathbb R}^s} G \big(u(\boldsymbol{y}_s)-u_h(\boldsymbol{y}_s)\big) {\rm d} {\boldsymbol{\mu}}_s({\boldsymbol{y}}_s) \right|} \\
			&\leq \int_{{\mathbb R}^s} \left\|{G}\right\|_{H^{-1}(D)} \left\|{u(\boldsymbol{y}_s)-u_h(\boldsymbol{y}_s)}\right\|_{V} {\rm d} {\boldsymbol{\mu}}_s({\boldsymbol{y}}_s). 
		\end{align*}
		Applying the a priori error estimate \eqref{norm{u- u_h}{V}-a} to bound 
		{the norm $\left\|{u(\boldsymbol{y}_s)-u_h(\boldsymbol{y}_s)}\right\|_{V}$}  yields\\
		\resizebox{\textwidth}{!}{
			\begin{minipage}{\textwidth}
				\begin{align*}
					{\left|I_s(u-u_h)\right|}
					&\leq C_{d,r} K_{D,r} h^{r-1} \left\|{G}\right\|_{H^{-1}(D)} \int_{{\mathbb R}^s} \left\|{f({\boldsymbol{y}}_s)}\right\|_{H^{r - 2}_{a({\boldsymbol{y}}_s)^{-1}}} \left(1 + \left\|{a({\boldsymbol{y}}_s)}\right\|_{W^{r -1}_{\infty,a({\boldsymbol{y}}_s)^{-1}}}\right)^{r-1} {\rm d} {\boldsymbol{\mu}}_s({\boldsymbol{y}}_s).
				\end{align*}
			\end{minipage}
		}\\
		Finally, taking the supremum over $s \in {\mathbb{N}}$ and invoking the uniform bound \eqref{cond1}, {we deduce  \eqref{FEM weakly error bound}}, which completes the proof.
		\hfill
	\end{proof}
}

\subsection{Total error}\label{subsec: total error}
In this section, we focus on the numerical evaluation of {infinite-dimensional} integrals associated with {the weak solution $u({\boldsymbol{y}})$ to  \eqref{ellip}}. A typical example includes
\begin{equation}\label{QoI def}
	I(u)=\int_{\mathbb{R}^{\mathbb{N}}} G(u({\boldsymbol{y}})) {\boldsymbol{\mu}}({\rm d}{\boldsymbol{y}}) \in \mathbb{R},
\end{equation}
where $G \in H^{-1}(D)$ is a bounded linear functional {and ${\boldsymbol{\mu}}$ is the normal infinite dimensional probability measure  on ${\mathbb R}^{{\mathbb{N}}}$, together with its density function, defined as in Section \ref{sec:Error Analysis for Quasi-Monte Carlo methods}}.

In practice, evaluating the exact integrals in \eqref{QoI def} is computationally infeasible. Therefore, a 
{feasible} fully discrete approximation requires three intertwining steps: {first, the infinite-dimensional integral $I(u)$ is truncated to an $s$-dimensional integral $I_s(u)$; second, the integral $I_s(u)$ is approximated by the integral $I_s(u_h)$ with the finite element approximation $u_h({\boldsymbol{y}}_s)$ of the weak solution $u({\boldsymbol{y}}_s)$;  and finally, an $n$-point {quadrature} rule -- such as a QMC rule -- is applied for approximation of the integral $I_s(u_h)$.}
Applying these three approximations concurrently, the target {integral} $I(u)$ is {approximately replaced} by $Q^{\Delta}_{s,n}(u_h)$. {Consequently,  the total approximation error is split into three  component errors:}
\begin{align}\label{total err decomp}
	{\left|I(u)-Q^{\Delta}_{s,n}(u_h)\right|}
	\leq
	{\left|(I-I_s)(u)\right|}
	+
	{\left|I_s(u-u_h)\right|}
	+
	{\left|(I_s-Q^{\Delta}_{s,n})(u_h)\right|}
\end{align}

To properly characterize the parameter domain for these evaluations, we introduce a sequence of functions $\boldsymbol{g}(\boldsymbol{y})=\left\{g_j(y_j)\right\}_{j\in{\mathbb{N}}}$ and a positive summable sequence $\boldsymbol{\mathsf{b}}=\left\{{\mathsf{b}}_j\right\}_{j\in{\mathbb{N}}}\in\ell^1({\mathbb{N}})$. We then define the admissible parameter set as
\begin{equation*}
	U_{\boldsymbol{g},\boldsymbol{\mathsf{b}}} := \left\{ {\boldsymbol{y}} \in \mathbb{R}^{\mathbb{N}} : \sum_{j \in \mathbb{N}} {\mathsf{b}}_j\,{\left|g_j(y_j)\right|}  <\infty \right\}.
\end{equation*}

To establish rigorous bounds for the dimensional truncation and spatial discretization errors, {recalling  the diffusion coefficient of the elliptic PDE \eqref{ellip} of the form $a({\boldsymbol{y}})=\exp\left(Z({\boldsymbol{y}})\right)$,} we work under the following assumptions:

\begin{enumerate}[label={(A\arabic*)}]
	{
	\item \label{assum:A1} 
	Let $r \ge 2$ and assume that $D \subset {\mathbb R}^d$ is a bounded domain with either a $C^\infty$-boundary or a convex $C^{r-1}$-boundary.  Suppose that the uniform ellipticity condition \eqref{PDE-ellipticity} holds, $a({\boldsymbol{y}}) \in W^{r-1,\infty}(D)$ and $f({\boldsymbol{y}}) \in H^{r - 2}(D)$ for all ${\boldsymbol{y}} \in {{\mathbb R}^{\mathbb{N}}}$. 
	\item \label{assum:A2} 
	 {Assume that for all $s \in {\mathbb N}$,
	there holds the condition \eqref{cond1}.}
}
\item \label{assum:A3} { {Assume that there exists a sequence of positive numbers $\boldsymbol{\mathsf{b}} = ({\mathsf{b}}_j)_{j \in {\mathbb N}}$ such that} the the function $Z(\boldsymbol{y})$ satisfies the bound
	$$
	\left\|{{Z({\boldsymbol{y}})}}\right\|_{{\infty}} 
	\le \sum_{j \in \mathbb{N}} {\mathsf{b}}_j {\left|g_j(y_j)\right|},
	$$
	where $(g_j)_{j \ge 1}$ is a sequence of functions satisfying
	$$
	\int_{{\mathbb R}^{\mathbb{N}}} \exp\left(\sum_{j \in \mathbb{N}}{\mathsf{b}}_j{\left|g_j(y_j)\right|}\right) {\boldsymbol{\mu}} ({\rm d} \boldsymbol{y}) < \infty.
	$$
	}
\item \label{assum:A4} {Suppose that $\boldsymbol{\mathsf{b}} = ({\mathsf{b}}_j)_{j \ge 1} \in \ell^p(\mathbb{N})$ for some $p \in (0, 1)$, {and} $0 \le {\mathsf{b}}_j < 1$ for all $j \in {\mathbb{N}}$. Furthermore, assume the existence of a constant $\delta \ge 1$ and a sequence of positive functions ${\boldsymbol{\alpha}}(\boldsymbol{y}) = (\alpha_j(y_j))_{j \in {\mathbb{N}}}$ such that the parametric derivatives of $Z : D \times U_{\boldsymbol{g},\boldsymbol{\mathsf{b}}}  \to {\mathbb{R}}$ satisfy 
	\begin{equation}\label{Assump_Z_bound} 
		\left\|{\partial^{\boldsymbol{\nu}} Z({\boldsymbol{y}})}\right\|_{{\infty}} \le  {\boldsymbol{\alpha}}(\boldsymbol{y})^{\boldsymbol{\nu}} \boldsymbol{\mathsf{b}}^{\boldsymbol{\nu}} (|\boldsymbol{\nu}|!)^{\delta},
	\end{equation}
	for all $\boldsymbol{y} \in U_{\boldsymbol{g},\boldsymbol{\mathsf{b}}}$ and multi-indices $\boldsymbol{\nu} \in \mathcal F\setminus\{\boldsymbol{0}\}$.}

	\item \label{assum:A5} {Assume}
	$$C_f :=
	\left\|{f}\right\|_{V'}
	\prod_{j \in \mathbb{N}}
	{\int_{\mathbb{R}} \max_{t\in[0,1]}  \varsigma_j(t \xi_j)  \mu_j({\rm d}\xi_j)} < \infty
	$$
	{
	where
	$$
	{\varsigma_j}(y_j) := \exp({\mathsf{b}}_j {\left|g_j(y_j)\right|}) ,\quad \forall j\in{\mathbb{N}}.
	$$
}
	\item \label{assum:A6} {There exist an integer $k\geq 2$ such that}
	$$
	M_k:=
	\max_{j\in{\mathbb{N}}}
	\max_{1\leq m\leq k+1} \int_{\mathbb{R}} \max_{t\in[0,1]}  \varsigma_j(t \xi_j) \alpha_j(t \xi_j)^m {\left|\xi_j\right|}^m \mu_j({\rm d}\xi_j)
	<\infty.
	$$
and an integer $s$ such that $\left\|{\overline{\boldsymbol{\mathsf{b}}}_s}\right\|_{\ell^1} \leq \frac{1}{2} M_k^{-1}$.
	\item \label{assum:A7} {Assume}
	$$
	\rho:=
	 \sup_{j\in {\mathbb{N}}} \left(
		\int_{{\mathbb{R}}} \varsigma_j(\xi_j)^2\alpha_j (\xi_j)^2 \psi_j^2(\xi_j) {{d}\xi_j}
	\right)^{\frac{1}{2}}<\infty
	$$
	where the weight function $\psi_j$ defined in \eqref{psi def}.
\end{enumerate}

\begin{remark}
	A seemingly more general formulation of Assumption \ref{assum:A4} would allow for a ${\boldsymbol{y}}$-independent constant $\rho \ge 1$, such that the parametric derivatives of $Z$ satisfy
	\begin{equation*}
		\left\|{\partial^{\boldsymbol{\nu}} Z({\boldsymbol{y}})}\right\|_{{\infty}} \le \rho\, \widehat{{\boldsymbol{\alpha}}}(\boldsymbol{y})^{\boldsymbol{\nu}} \boldsymbol{\mathsf{b}}^{\boldsymbol{\nu}} (|\boldsymbol{\nu}|!)^{\delta}, \quad \forall \boldsymbol{\nu} \in \mathcal F\setminus\{\boldsymbol{0}\},
	\end{equation*}
	for some sequence of positive functions $\widehat{{\boldsymbol{\alpha}}}(\boldsymbol{y}) = (\widehat \alpha_j(y_j))_{j \in {\mathbb{N}}}$. However, by Lemma \ref{lemma:equiv defs} and defining the scaled sequence 
	${{\boldsymbol{\alpha}}(\boldsymbol{y})} := \left(\rho \widehat \alpha_j(y_j)\right)_{j\in{\mathbb{N}}}$, we can absorb the constant $\rho$ and directly recover the exact form of Assumption \ref{assum:A4}.
\end{remark}

\begin{remark}
	Suppose that the function $Z({\boldsymbol{y}})$ takes the separable form 
	\begin{equation}\label{eq:Z_form}
		Z({\boldsymbol{y}})({\boldsymbol{x}}) = \sum_{j \in \mathbb{N}} \psi_j({\boldsymbol{x}}) g_j(y_j),
	\end{equation}
	where $\left\|{\psi_j}\right\|_{L^\infty(D)} \leq {\mathsf{b}}_j$ and $g_j \in G^{\delta}({\mathbb{R}},{\mathbb{R}})$. By the definition of the Gevrey class, there exist functions $\kappa_j(y_j) \geq 1$ and $\gamma_j(y_j) > 0$ such that
	\begin{equation*}
		{\left|g_j^{(k)}(y_j)\right|} \leq \kappa_j(y_j) \gamma_j(y_j)^{k} {\mathsf{b}}_j^k (k!)^{\delta} \quad \text{for all } k \in {\mathbb{N}}_0.
	\end{equation*}
	Thanks to Lemma \ref{lemma:equiv defs}, we define $\alpha_j(y_j) := \kappa_j(y_j) \gamma_j(y_j)$, which allows us to simplify the derivative bound to
	\begin{equation*}
		{\left|g_j^{(k)}(y_j)\right|} \leq \alpha_j(y_j)^{k} {\mathsf{b}}_j^k (k!)^{\delta} \quad \text{for all } k \in {\mathbb{N}}_0.
	\end{equation*}
	Consequently, any function $Z({\boldsymbol{y}})$ admitting the representation \eqref{eq:Z_form} naturally satisfies Assumption \ref{assum:A4}.
\end{remark}

\begin{remark}	
	Alternatively, rather than imposing conditions on the exponent $Z({\boldsymbol{y}})$, one can formulate analogous assumptions directly on the diffusion coefficient $a({\boldsymbol{y}})$.
	
	\begin{enumerate}[label={(A\arabic*')},start=3]
		\item \label{assum:A3'} Assume that the diffusion coefficient satisfies the bound
		$$
		\left\|{a({\boldsymbol{y}})^{-1}}\right\|_{{\infty}} \le \exp\left( \sum_{j\geq 1} {\mathsf{b}}_j {\left|g_j(y_j)\right|}\right),
		$$
		where $(g_j)_{j \ge 1}$ is a sequence of functions fulfilling the Integrability condition
		$$
		\int_{{\mathbb R}^{\mathbb{N}}} \exp\left(\sum_{j \in \mathbb{N}}{\mathsf{b}}_j{\left|g_j(y_j)\right|}\right) {\boldsymbol{\mu}} ({\rm d} \boldsymbol{y}) < \infty.
		$$
		
		\item \label{assum:A4'} Suppose that $\boldsymbol{\mathsf{b}} = ({\mathsf{b}}_j)_{j \ge 1} \in \ell^p(\mathbb{N})$ for some $p \in (0, 1)$, with $0 \le {\mathsf{b}}_j < 1$ for all $j \in {\mathbb{N}}$. Furthermore, assume the existence of a constant $\delta \ge 1$ and a sequence of positive functions ${\boldsymbol{\beta}}(\boldsymbol{y}) = (\beta_j(y_j))_{j \in {\mathbb{N}}}$ such that the mapping $a : D \times U_{\boldsymbol{g},\boldsymbol{\mathsf{b}}} \to {\mathbb{R}}$ satisfies 
		\begin{equation}\label{Assump_a_bound} 
			\left\|{\partial^{\boldsymbol{\nu}} a({\boldsymbol{y}})}\right\|_{{\infty},a(\boldsymbol{y})^{-1}} \le  {\boldsymbol{\beta}}(\boldsymbol{y})^{\boldsymbol{\nu}} \boldsymbol{\mathsf{b}}^{\boldsymbol{\nu}} (|\boldsymbol{\nu}|!)^{\delta},
		\end{equation}
		for all $\boldsymbol{y} \in U_{\boldsymbol{g},\boldsymbol{\mathsf{b}}}$ and multi-indices $\boldsymbol{\nu} \in \mathcal F$.
	\end{enumerate} 
	
	Since $a({\boldsymbol{y}}) = \exp(Z({\boldsymbol{y}}))$, it is straightforward to see that Assumption~\ref{assum:A3'} follows directly from Assumption~\ref{assum:A3}. Moreover, by invoking Lemma~\ref{lemma:G-regularity-a(by)}, Assumption~\ref{assum:A4'} emerges as a natural consequence of Assumption~\ref{assum:A4} with  $\beta_j=4 \alpha_j$ for all $j\in {\mathbb{N}}$.
	
	We emphasize that Assumption~\ref{assum:A4'} significantly relaxes Assumption (A1) in \cite{GuthKaarnioja2025}. While their setting requires the bounds on the derivatives of the diffusion coefficient $a$ to be uniform with respect to ${\boldsymbol{y}}$, conditions \ref{assum:A3'} and \ref{assum:A4'} herein allow these bounds to depend explicitly on the parameter ${\boldsymbol{y}}$.
\end{remark}

With these assumptions in place, the following theorem establishes the comprehensive error bound for our QMC integration scheme.

\begin{theorem} \label{thm:total-error}
	Let the diffusion coefficient be given by $a({\boldsymbol{y}}) = \exp(Z({\boldsymbol{y}}))$. Suppose that Assumptions \ref{assum:A1}--\ref{assum:A7} hold 
		(Assumptions \ref{assum:A3} and \ref{assum:A4} may be alternatively replaced with  Assumptions \ref{assum:A3'} and \ref{assum:A4'}, respectively). Then, for a given dimension truncation parameter $s$ and a uniform finite element mesh of size $h$, there exists an $n$-point QMC quadrature rule, constructed via a CBC algorithm, such that the total error satisfies the bound
	\begin{equation} \label{eq:total_error_bound_thm}
		\sqrt{ \mathrm{E}_\Delta | I(u) - Q^\Delta_{s,n}({{}u_h}) |^2 } 
		\leq 
		C_{\mathrm{total}}\left( {s^{-\tau}} + {h^{r-1}} + n^{-\frac{1}{2\vartheta}} \right),
	\end{equation}
	where the total error constant is defined by
	$$
	C_{\mathrm{total}} := \sqrt{2}\max\left\{C_{\delta}, C_{\mathrm{FEM}}, C_{\gamma, \vartheta}^{1/2}\right\}.
	$$
	Here, the dimension truncation rate $\tau$ and the QMC quadrature rate $\vartheta$ are given by \eqref{gamma def} and \eqref{vartheta}, respectively.
\end{theorem}

\begin{proof}
	To begin, we recall the error decomposition from \eqref{total err decomp}. By applying the triangle inequality, the total error $\mathcal{E}:=\sqrt{ \mathrm{E}_\Delta \left( {\left|I(u) - Q^\Delta_{s,n}({{}u_h})\right|}^2\right) } $ can be {estimated} as follows:
	\begin{align}\label{eq:total_error_bound}
		\mathcal{E}
		\leq 
		\sqrt{2} \bigg({ \underbrace{{\left|(I - I_s)(u)\right|}}_{\text{dimension}\atop\text{truncation error}} 
		+ 
		\underbrace{{\left|I_s(u - u_h)\right|}}_{\text{FEM error}}
		+
		\underbrace{ \sqrt{ \mathrm{E}_\Delta\left( {\left| I_s({{}u_h}) - Q^\Delta_{s,n}({{}u_h})\right|}^2 \right)} }_{\text{QMC error}}}\bigg).
	\end{align}
	{We give a bound for each} of these three components individually.
	
Assumptions \ref{assum:A3'} and \ref{assum:A4'} follow directly from Assumptions \ref{assum:A3} and \ref{assum:A4} combined with the regularity results of Lemma \ref{lemma:G-regularity-a(by)}. {{}Invoking Corollary \ref{cor:G-regularity-lognormal}, which applies to both the exact solution $u$ (with $X=V$) and the discrete solution $u_h$ (with $X=V_h$), yields the following uniform bounds on their parametric derivatives:}
	\begin{align}\label{u derivative}
					\left\|{\partial^{\boldsymbol{\nu}} u({\boldsymbol{y}})}\right\|_{V}
				&
				\le \
		\left\|{f}\right\|_{V'} \exp\left(\left\|{Z({\boldsymbol{y}})}\right\|_{\infty}\right) {12}^{{\left|{\boldsymbol{\nu}}\right|}}\, {\boldsymbol{\beta}}({\boldsymbol{y}})^{\boldsymbol{\nu}} \boldsymbol{\mathsf{b}}^{\boldsymbol{\nu}} (|{\boldsymbol{\nu}}|!)^\delta \\
		&{{}=}
		C_f
		\widecheck {\boldsymbol{\beta}}(\boldsymbol{y})^{\boldsymbol{\nu}}
		\boldsymbol{\mathsf{b}}^{\boldsymbol{\nu}} (|{\boldsymbol{\nu}}|!)^\delta,\notag
			\end{align}
		{{}and
\begin{align}\label{u_h derivative bound}
				\left\|{\partial^{\boldsymbol{\nu}} u_h({\boldsymbol{y}})}\right\|_{V}
			\leq
				C_f
				\widecheck {\boldsymbol{\beta}}(\boldsymbol{y})^{\boldsymbol{\nu}}
				\boldsymbol{\mathsf{b}}^{\boldsymbol{\nu}} (|{\boldsymbol{\nu}}|!)^\delta,
\end{align}}
		 where $C_f$ is the constant from Assumption \ref{assum:A5} and we have defined the scaled sequence 
		 $\widecheck {\boldsymbol{\beta}}(\boldsymbol{y}):=\left(\widecheck \beta_j(y_j)\right)_{j\in{\mathbb{N}}}
		 = \left(12 \beta_j(y_j)\right)_{j\in {\mathbb{N}}}
		$.
	
The estimate in \eqref{u derivative}, together with Assumption~\ref{assum:A5} and \ref{assum:A6}, ensures that the requirements of {Theorem~\ref{thm:general trunc theorem_revised} are satisfied. Consequently, the dimension} truncation error is bounded by:
	\begin{align*}
		| (I - I_s)(u) | \leq C_{\delta}\, {s^{-\tau}}.
	\end{align*}
	
	Similarly, utilizing Assumptions \ref{assum:A5} and \ref{assum:A7} along with the derivative bound \eqref{u_h derivative bound}, we satisfy the conditions for 
	{Theorem~\ref{QMC theorem}.} This yields the following estimate for the QMC error:
	\begin{align*}
		\sqrt{ \mathrm{E}_\Delta | I_s({{}u_h}) - Q^\Delta_{s,n}({{}u_h}) |^2 } \leq C_{\gamma, \vartheta}^{1/2} n^{-\frac{1}{2\vartheta}}.
	\end{align*}
	
Finally, under Assumptions (A1) and (A2), we apply Corollary~\ref{corollary:FEM error} to bound the FEM error arising from the spatial discretization:
	\begin{align*}
		| I_s(u - u_h) | \leq	C_{\mathrm{FEM}}  h^{r-1}.
	\end{align*}
	
	Combining the bounds into {the right-hand of} \eqref{eq:total_error_bound}, we arrive at the desired result. 
	\hfill
\end{proof} 
\section{Numerical experiment}\label{sec:Numerical experiment}

Let $D= (0,1)^2$. We consider the parametric elliptic PDE \eqref{ellip} with a constant source term $f\equiv 1$ and a log-normal random diffusion coefficient $a({\boldsymbol{x}},\boldsymbol{y})=\exp(Z({\boldsymbol{x}},\boldsymbol{y}))$, where the random field $Z$ is chosen as one of the following two functions:
\begin{equation}\label{QMC-a1-new}
	Z^{(1)} ({\boldsymbol{x}},\boldsymbol{y}) = 
	-5+\sum_{j=1}^{s}
	j^{-3}
	\sin(j x_1)\sin(j x_2)\cosh(j^{-3} y_j),
\end{equation}
\begin{equation}\label{QMC-a2-new}
	Z^{(2)} ({\boldsymbol{x}},\boldsymbol{y}) = 	
	-5+\sum_{j=1}^{s}
	j^{-3}
	\sin(j x_1)\sin(j x_2)[\exp\left(-\sin(j^{-3} y_j)^{-2}\right)+ \cos(j^{-3} y_j)],
\end{equation}
and  $y_j$ are i.i.d. standard Gaussian random variables for all $j \in {\mathbb N}$.

We set the truncation dimension to $s=100$. It is easy to verify that the corresponding sequence $\boldsymbol{\mathsf{b}}^{(i)}$ belongs to $\ell^p$ for any $p>\frac{1}{3}$. Furthermore, $a^{(1)}$ is analytic with respect to $\boldsymbol{y}$ ($\delta^{(1)} = 1$), whereas $a^{(2)}$ belongs to the Gevrey class with index $\delta^{(2)} = \frac{3}{2}$. By virtue of Theorem \ref{thm:G-regularity-u(by)}, this parametric regularity carries over to the corresponding solutions $u^{(1)}$ and $u^{(2)}$ with the same Gevrey index $\delta$. 

To sufficiently isolate the quadrature error from the spatial discretization error, the solutions are computed on a highly refined finite element mesh with a mesh size of $h = \sqrt{2} \cdot 2^{-8}$. Specifically, for the quantity of interest $G(w)=\int_D w\, {\rm d} {\boldsymbol{x}}$, we compute the relative QMC error as follows:
\begin{equation}
	\mathrm{err}_n^{{\rm QMC},(k)}=\sqrt{
		\frac{1}{R} \sum_{j=1}^R
		\left({\left|\frac{ I_s^*(G(u^{(k)}))-Q^{(j)}_{s,n}(G(u^{(k)}))}{I_s^*(G(u^{(k)}))}\right|}^2\right)},
\end{equation}
where $Q^{(j)}_{s,n}$ denotes the randomly shifted QMC quadrature rule. The expected error is approximated by the empirical root-mean-square error over $R = 16$ independent random shifts $\Delta^{(j)} \sim \mathcal{U}(0,1)^s$. Depending on the decay rate of the sequence $\boldsymbol{\mathsf{b}}$ and Gevrey index $\delta^{(k)}$, the generating vector $\boldsymbol{z}$ in \eqref{QMC quad def} is computed via the QMC4PDE software detailed in the survey \cite{KuNu16}. A completely analogous error metric, denoted by $\mathrm{err}_n^{{\rm MC},(k)}$, is used for the plain Monte Carlo approximation. In all cases, the ``exact'' reference value $I_s^*(G(u^{(k)}))$ is approximated using the QMC estimation at the highest available number of sample points.

Given that $p \in (0,\frac{2}{3 \delta^{(2)}}) \subset (0,\frac{2}{3 \delta^{(1)}})$, our theoretical analysis predicts that the convergence rate of $\mathrm{err}_n^{{\rm QMC},(k)}$ is approximately $\mathcal{O}(n^{-1})$, while $\mathrm{err}_n^{{\rm MC},(k)}$ decays at the standard rate of $\mathcal{O}(n^{-\frac{1}{2}})$. As illustrated in Figure \ref{fig:QMC}, this expected convergence behavior is clearly reproduced in our numerical experiments.

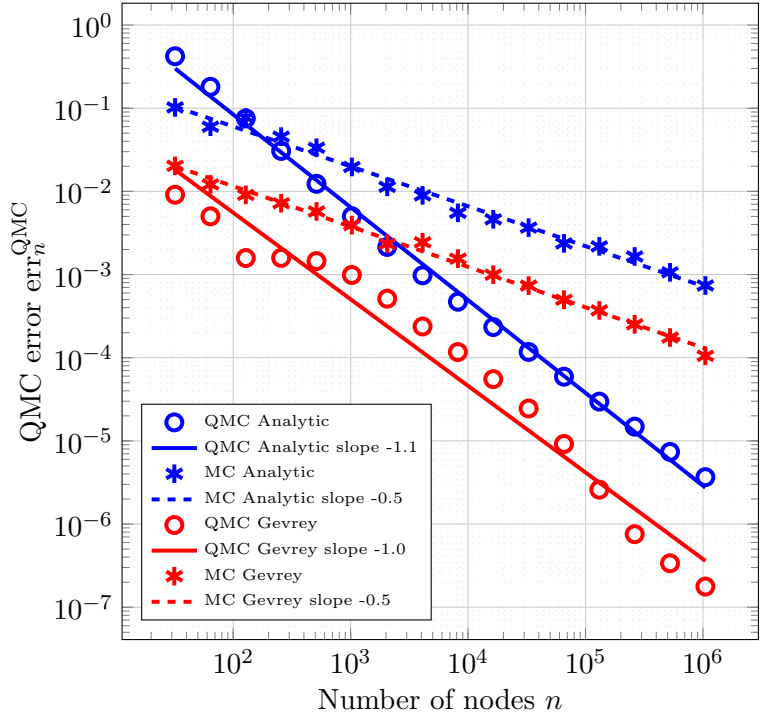
\begin{figure}[htbp]
	\centering
	\begin{tikzpicture}
		\begin{loglogaxis}[
			width=10cm, height=10cm, 
			xlabel={Number of nodes $n$},
			ylabel={QMC error $\mathrm{err}^{\mathrm{QMC}}_n$},
			grid=both, 
			major grid style={solid, gray!40},
			minor grid style={dotted, gray!20},
			legend pos=south west,
			font=\large,
			legend cell align={left},
			legend style={font=\tiny},
			]
			
			\addplot[
			color=blue, only marks, mark=o, mark size=3pt, line width=1.5pt
			] coordinates {
				(32, 0.4211857513177363343537251694215228781104087829589843750000000000)
				(64, 0.1806758463490375221649486547903507016599178314208984375000000000)
				(128, 0.0752897070291951159282106686987390276044607162475585937500000000)
				(256, 0.0306314820660975664634584347822965355589985847473144531250000000)
				(512, 0.0123233322198484599591106203320123313460499048233032226562500000)
				(1024, 0.0050260733057262384215757577976546599529683589935302734375000000)
				(2048, 0.0021464141181552517760600107266100167180411517620086669921875000)
				(4096, 0.0009779942251760681548589104750135447829961776733398437500000000)
				(8192, 0.0004708116054674602556084916926693040295504033565521240234375000)
				(16384, 0.0002335808779871873712506019638013299299927894026041030883789062)
				(32768, 0.0001173175257043817208951058672639078395150136202573776245117188)
				(65536, 0.0000590523828955164166052901886239112627663416787981986999511719)
				(131072, 0.0000296542994538804149538737536806465300287527497857809066772461)
				(262144, 0.0000148256194330320988035475451849443118135241093114018440246582)
				(524288, 0.0000073685421168375103286138534663951560332861845381557941436768)
				(1048576, 0.0000036656464007213894485944036955027414137475716415792703628540)
			};
			\addlegendentry{QMC Analytic}
			
			\addplot[
			color=blue, no marks, line width=1.5pt
			] coordinates {
				(32, 0.2998950708125042718954489373572869226336479187011718750000000000)
				(64, 0.1383556069329923421395989180382457561790943145751953125000000000)
				(128, 0.0638299052996589844077846009895438328385353088378906250000000000)
				(256, 0.0294477173775592460103212033573072403669357299804687500000000000)
				(512, 0.0135856077911686579096084415141376666724681854248046875000000000)
				(1024, 0.0062676755786889590607580124981268454575911164283752441406250000)
				(2048, 0.0028915715633444408726171470647159367217682301998138427734375000)
				(4096, 0.0013340170532073335630390831241243176918942481279373168945312500)
				(8192, 0.0006154443904510045748601498871721560135483741760253906250000000)
				(16384, 0.0002839332501986687427864008981259758002124726772308349609375000)
				(32768, 0.0001309916733651629518371539706578232653555460274219512939453125)
				(65536, 0.0000604325787099590431943052504237812172505073249340057373046875)
				(131072, 0.0000278803719023766965697842812321738392711267806589603424072266)
				(262144, 0.0000128625180988137711204618088278728293971653329208493232727051)
				(524288, 0.0000059340805216521633097298722703261120159368147142231464385986)
				(1048576, 0.0000027376685783399687496258816110650258224268327467143535614014)
			};
			\addlegendentry{QMC Analytic slope -1.1}
			
			\addplot[
			color=blue, only marks, mark=asterisk, mark size=3.5pt, line width=1.5pt
			] coordinates {
				(32, 0.1021796661737822825566368578620313201099634170532226562500000000)
				(64, 0.0599215316705703007760241973755910294130444526672363281250000000)
				(128, 0.0770566127037554116085971145366784185171127319335937500000000000)
				(256, 0.0454779405232919736734586990678508300334215164184570312500000000)
				(512, 0.0334438083905827729380710877649107715114951133728027343750000000)
				(1024, 0.0196601519656832458837758537129047908820211887359619140625000000)
				(2048, 0.0113297177572978229748157374956463172566145658493041992187500000)
				(4096, 0.0089468362741527210357794785977603169158101081848144531250000000)
				(8192, 0.0055191721766108953195484332354681100696325302124023437500000000)
				(16384, 0.0046121513374611840946348806369314843323081731796264648437500000)
				(32768, 0.0036665576565470727604667544596850348170846700668334960937500000)
				(65536, 0.0023682173808798061449831795499676445615477859973907470703125000)
				(131072, 0.0021807021175297039254636910499129953677766025066375732421875000)
				(262144, 0.0016495929504288214372670351437477620493154972791671752929687500)
				(524288, 0.0010634423274936671138801713709653995465487241744995117187500000)
				(1048576, 0.0007386735231089322766678817444585547491442412137985229492187500)
			};
			\addlegendentry{MC Analytic}
			
			\addplot[
			color=blue, no marks, dashed, line width=1.5pt
			] coordinates {
				(32, 0.1040860426102085956046039427747018635272979736328125000000000000)
				(64, 0.0746445408581340591469555079129349905997514724731445312500000000)
				(128, 0.0535307841492973709995872866329591488465666770935058593750000000)
				(256, 0.0383892086239071075870477045555162476375699043273925781250000000)
				(512, 0.0275305389635172662043949287635768996551632881164550781250000000)
				(1024, 0.0197433237826563773298360615626734215766191482543945312500000000)
				(2048, 0.0141587796193658067150922263977008697111159563064575195312500000)
				(4096, 0.0101538647958493051526307482390620862133800983428955078125000000)
				(8192, 0.0072817695496418567702168189725853153504431247711181640625000000)
				(16384, 0.0052220675417863184411548793661950185196474194526672363281250000)
				(32768, 0.0037449673771012756406939558928570477291941642761230468750000000)
				(65536, 0.0026856758445440052571195899133726925356313586235046386718750000)
				(131072, 0.0019260127033603525494032115972231622436083853244781494140625000)
				(262144, 0.0013812258620270255297934891558497838559560477733612060546875000)
				(524288, 0.0009905359806836940440161187737544423725921660661697387695312500)
				(1048576, 0.0007103556022250397724432380286430088744964450597763061523437500)
			};
			\addlegendentry{MC Analytic slope -0.5}
			
			\addplot[
			color=red, only marks, mark=o, mark size=3pt, line width=1.5pt
			] coordinates {
				(32, 0.0091098093558062646218909108597472368273884057998657226562500000)
				(64, 0.0050089183360178214307611987976542877731844782829284667968750000)
				(128, 0.0015801584408951620860234177357028784172143787145614624023437500)
				(256, 0.0015847517884800222873198372042224946198984980583190917968750000)
				(512, 0.0014515246195265353677111841435021233337465673685073852539062500)
				(1024, 0.0009878565388064715086652300968239615031052380800247192382812500)
				(2048, 0.0005124791807532936638650000737982281862059608101844787597656250)
				(4096, 0.0002372327439998545440397181849334629077930003404617309570312500)
				(8192, 0.0001169645526226410889216594801709447892790194600820541381835938)
				(16384, 0.0000552855481283278656990179400487761540716746821999549865722656)
				(32768, 0.0000245249034850487766671114919780904983781510964035987854003906)
				(65536, 0.0000091833511518815016494259040324976695046643726527690887451172)
				(131072, 0.0000025924744220572767116359067657294090736286307219415903091431)
				(262144, 0.0000007557943267933220480548034241052857140630294452421367168427)
				(524288, 0.0000003377483961294864246989824198286189727014061645604670047760)
				(1048576, 0.0000001773380654413738012768660947879317824060763086890801787376)
			};
			\addlegendentry{QMC Gevrey}
			
			\addplot[
			color=red, no marks, line width=1.5pt
			] coordinates {
				(32, 0.0181442359300516206643827388234058162197470664978027343750000000)
				(64, 0.0088174050834240753210524488281407684553414583206176757812500000)
				(128, 0.0042849218178663420544816098356477596098557114601135253906250000)
				(256, 0.0020823082087657717592632522496387537103146314620971679687500000)
				(512, 0.0010119221914887622983808546095474412140902131795883178710937500)
				(1024, 0.0004917555034921353238758223724857998604420572519302368164062500)
				(2048, 0.0002389743769321114516806375815960450381680857390165328979492188)
				(4096, 0.0001161324121937442032650653067626933534484123811125755310058594)
				(8192, 0.0000564359130676551050279435295387031601421767845749855041503906)
				(16384, 0.0000274256964409415673888936737512267427518963813781738281250000)
				(32768, 0.0000133278401001323140527944979827346116962871747091412544250488)
				(65536, 0.0000064768208208387888886407163235858774896769318729639053344727)
				(131072, 0.0000031474873370392896800538939938718385747051797807216644287109)
				(262144, 0.0000015295585304673754765378583184154415164357487810775637626648)
				(524288, 0.0000007433069771541098586450382022794425296297049499116837978363)
				(1048576, 0.0000003612187773665363143169265529308642470596169005148112773895)
			};
			\addlegendentry{QMC Gevrey slope -1.0}
			
			\addplot[
			color=red, only marks, mark=asterisk, mark size=3.5pt, line width=1.5pt
			] coordinates {
				(32, 0.0202747834813828582689598789556839619763195514678955078125000000)
				(64, 0.0121315327976217000177650007231022755149751901626586914062500000)
				(128, 0.0090788441411880881071194693276993348263204097747802734375000000)
				(256, 0.0072011813803641577802494566640234552323818206787109375000000000)
				(512, 0.0057351754051910451212559216571662545902654528617858886718750000)
				(1024, 0.0039317390761976055840731802959453489165753126144409179687500000)
				(2048, 0.0023381609523805534783802428933086048346012830734252929687500000)
				(4096, 0.0024430269296752309474141107159539387794211506843566894531250000)
				(8192, 0.0015521997293868549337714224378714789054356515407562255859375000)
				(16384, 0.0009962187141826308895559138889552741602528840303421020507812500)
				(32768, 0.0007363177255756451437124221470753582252655178308486938476562500)
				(65536, 0.0004979094049602688157318830519670882495120167732238769531250000)
				(131072, 0.0003704839662718546950716247412316306508728303015232086181640625)
				(262144, 0.0002519262555529000490762980835057760486961342394351959228515625)
				(524288, 0.0001747809035498701271237748100517706006939988583326339721679688)
				(1048576, 0.0001053675733997393068304762597797719081427203491330146789550781)
			};
			\addlegendentry{MC Gevrey}
			
			\addplot[
			color=red, no marks, dashed, line width=1.5pt
			] coordinates {
				(32, 0.0200348032796360819851599899266147986054420471191406250000000000)
				(64, 0.0143138157011116051026133888512958947103470563888549804687500000)
				(128, 0.0102264702610601696669556659458066860679537057876586914062500000)
				(256, 0.0073062764104351533150927267001861764583736658096313476562500000)
				(512, 0.0052199511290758112552756742275050783064216375350952148437500000)
				(1024, 0.0037293811867045094203543253996713247033767402172088623046875000)
				(2048, 0.0026644471742799605978868626721123291645199060440063476562500000)
				(4096, 0.0019036077003439257331746103929503988183569163084030151367187500)
				(8192, 0.0013600278180737285314144680015147059748414903879165649414062500)
				(16384, 0.0009716685142638397049488774293024562211940065026283264160156250)
				(32768, 0.0006942061692156615030843691549478080560220405459403991699218750)
				(65536, 0.0004959738823504020828122906827672977669863030314445495605468750)
				(131072, 0.0003543473147921729336264073939588570283376611769199371337890625)
				(262144, 0.0002531625635313486735733512489332497352734208106994628906250000)
				(524288, 0.0001808713680005000375410756374705556481785606592893600463867188)
				(1048576, 0.0001292231019706882469508940625146919956023339182138442993164062)
			};
			\addlegendentry{MC Gevrey slope -0.5}
			
		\end{loglogaxis}
	\end{tikzpicture}
	\caption{QMC and MC errors vs Number of nodes $n$}\label{fig:QMC}
\end{figure}

As detailed in Figure \ref{fig:QMC}, the numerical results are derived from an error estimation utilizing 16 QMC runs, comprising one unshifted run and 15 random shifts. However, since the unshifted quadrature point set naturally includes the origin $\boldsymbol{0}_s \in [0,1)^s$, mapping this specific point via the inverse cumulative distribution function results in $\Phi^{-1}(\boldsymbol{0}_s)=(-\infty,\dots,-\infty)$, which inevitably induces a singularity during the evaluation of the quantity of interest. To circumvent this computational issue, we explicitly omit this singular point from the quadrature sum.

To analyze the dimension truncation error, we fix the  mesh size at $h=\sqrt{2}\cdot 2^{-8}$ and employ the PDE solution with a high truncation dimension $s^*=256$ as our reference solution. We then consider the error quantities
\begin{align*}
	&\bigg|\int_{{{\mathbb{R}}}^{s^*}}G(u_{s^*,h}({\boldsymbol{y}}))\,{\boldsymbol{\mu}}_{\beta}({\rm d}{\boldsymbol{y}})-\int_{\mathbb R^{s}}G(u_{s,h}({\boldsymbol{y}}))\,{\boldsymbol{\mu}}_{\beta}({\rm d}{\boldsymbol{y}})\bigg|
\end{align*}
for varying dimensions $s\in\{2,4,8,16,32,64\}$, where the high-dimensional integrals within this norm are approximated using lattice rules with $n=65536$ nodes subjected to a single random shift. Focusing on the first test case where $a({\boldsymbol{x}},\boldsymbol{y})=\exp(Z^{(1)}({\boldsymbol{x}},\boldsymbol{y}))$, the corresponding results are depicted in Figure~\ref{fig:trunc}. For $p>\tfrac{1}{3}$, the theoretically expected dimension truncation rate of $\mathcal O(s^{-5.045})$ closely matches the observed empirical rate of $\mathcal O(s^{-5.001})$. These numerical findings are in strong agreement with Theorem \ref{thm:general trunc theorem_revised}, with minor discrepancies being largely attributable to the residual numerical integration and finite element errors inherent in the computations.

\begin{figure}[htbp]
	\centering
	\begin{tikzpicture}
		\begin{loglogaxis}[
			width=10cm, height=10cm, 
			xlabel={Dimension $s$},
			ylabel={Truncation error},
			grid=both, 
			major grid style={solid, gray!40}, 
			minor grid style={dotted, gray!20}, 
			legend pos=north east, 
			font=\large,
			legend cell align={left},
			]
			
			\addplot[
			color=blue,
			only marks, 
			mark=o,
			mark size=3pt,
			line width=1.5pt
			] coordinates {
				(2, 0.0372068575471224960438121343031525611877441406250)
				(4, 0.0032567741220836943227823212509974837303161621094)
				(8, 0.0004333307981818634857518190983682870864868164062)
				(16, 0.0000087622359026262586212396854534745216369628906)
				(32, 0.0000004844381606972092413343489170074462890625000)
				(64, 0.0000000109103064360738244431558996438980102539062)
				(128, 0.0000000002041851132617011899128556251525878906250)
			};
			\addlegendentry{Truncation Error Analytic}

			\addplot[
			color=blue,
			no marks,
			line width=1.5pt
			] coordinates {
				(2, 0.3932248858479072683991262238123454153537750244141)
				(4, 0.0122802880291336090345932774425818934105336666107)
				(8, 0.0003835094865710288875026989785510522779077291489)
				(16, 0.0000119768792019408840804611110497646109251945745)
				(32, 0.0000003740341254669768899627460990303662313749555)
				(64, 0.0000000116809666904860579235062413919199464373833)
				(128, 0.0000000003647928719174873815766019001592245646348)
			};
			\addlegendentry{slope -5.001} 
			\addplot[
			color=red,
			only marks,
			mark=o,
			mark size=3pt,
			line width=1.5pt
			] coordinates {
				(2, 0.0418921434205810072626263718120753765106201171875)
				(4, 0.0039425091667784073479197104461491107940673828125)
				(8, 0.0005339618564264014821674209088087081909179687500)
				(16, 0.0000100675765493107860493182670325040817260742188)
				(32, 0.0000005479501181682167043618392199277877807617188)
				(64, 0.0000000122080638931265639257617294788360595703125)
				(128, 0.0000000002269597842996518011204898357391357421875)
			};
			\addlegendentry{Truncation Error Gevrey}

			\addplot[
			color=red,
			no marks,
			line width=1.5pt
			] coordinates {
				(2, 0.5062053508059254358641965154674835503101348876953)
				(4, 0.0153337828170451023757081543408276047557592391968)
				(8, 0.0004644852036944439031308784482376950109028257430)
				(16, 0.0000140700117528235686860252698360262968435563380)
				(32, 0.0000004262035241381401845154166044271626390127494)
				(64, 0.0000000129103974594276268529122102886493028695014)
				(128, 0.0000000003910769224572902859008626107482691977579)
			};
			\addlegendentry{slope -5.001} 
			
		\end{loglogaxis}
	\end{tikzpicture}
	\caption{Dimension truncation error}\label{fig:trunc}
\end{figure}
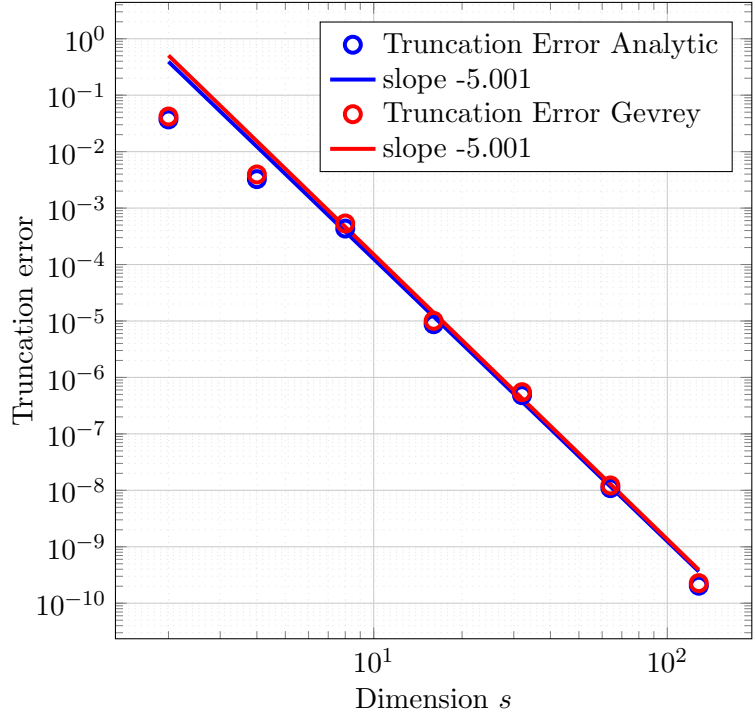

Finally, to assess the finite element discretization error, we fix the truncation dimension at $s=256$ and utilize the finite element solution on a highly refined mesh with $h^*=2^{-8}$ as the reference. We approximate the quantities
\begin{align*}
	\int_{\mathbb R^s} \sqrt{\left\|{u_{s,h^*}({\boldsymbol{y}})}\right\|_{H^1(D)}^2-\left\|{u_{s,h}({\boldsymbol{y}})}\right\|_{H^1(D)}^2} \,{\boldsymbol{\mu}}_{\beta}({\rm d}{\boldsymbol{y}})
\end{align*}
across a hierarchy of regular finite element meshes with sizes $h=\sqrt{2}\cdot 2^{-k}$ for $k\in\{3,4,5,6,7\}$. Similar to the truncation analysis, the high-dimensional integrals were computed via a QMC cubature rule using $n=65536$ nodes subject to a single random shift. The results for the $Z^{(1)}$ case, displayed in Figure~\ref{fig:fem}, demonstrate that the observed finite element errors align perfectly with the theoretical error rate of $\mathcal O(h^{-1})$.

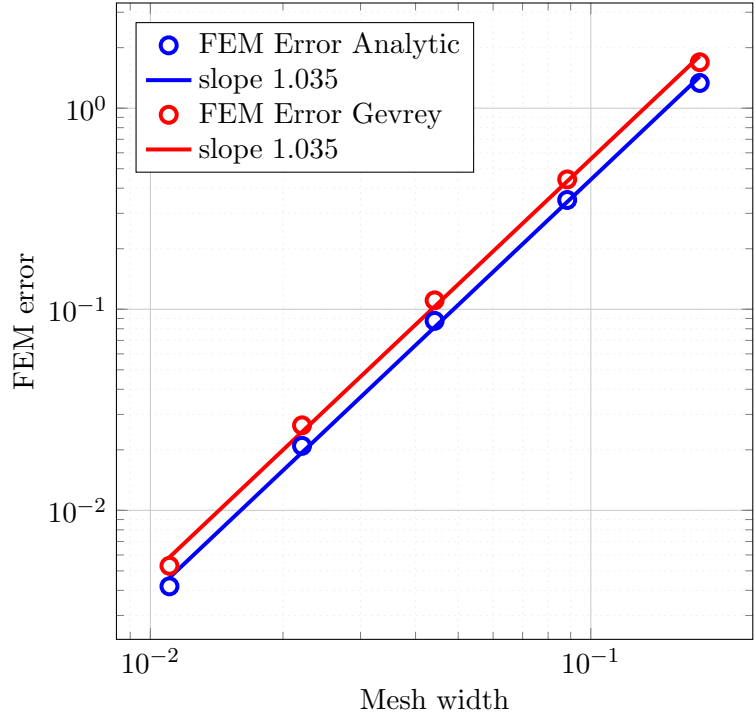
\begin{figure}[htbp]
	\centering
	\begin{tikzpicture}
		\begin{loglogaxis}[
			width=10cm, height=10cm, 
			xlabel={Mesh width},
			ylabel={FEM error},
			grid=both, 
			major grid style={solid, gray!40}, 
			minor grid style={dotted, gray!20}, 
			legend pos=north west, 
			font=\large,
			legend cell align={left},
			]
			
			\addplot[
			color=blue,
			only marks, 
			mark=o,
			mark size=3pt,
			line width=1.5pt
			] coordinates {
				(0.17677669529663689318432773234235355630517005920410156250000, 1.3350715638669754525835742242634296417236328125000)
				(0.08838834764831844659216386617117677815258502960205078125000, 0.3495048115535066557413301779888570308685302734375)
				(0.04419417382415922329608193308558838907629251480102539062500, 0.0875517825342946309774561086669564247131347656250)
				(0.02209708691207961164804096654279419453814625740051269531250, 0.0209260304174527789200510596856474876403808593750)
				(0.01104854345603980582402048327139709726907312870025634765625, 0.0041896231089229019062258885242044925689697265625)
			};
			\addlegendentry{FEM Error Analytic}

			\addplot[
			color=blue,
			no marks,
			line width=1.5pt
			] coordinates {
				(0.17677669529663689318432773234235355630517005920410156250000, 1.4345047754034581100057721414486877620220184326172)
				(0.08838834764831844659216386617117677815258502960205078125000, 0.3417900768671371936058278606651583686470985412598)
				(0.04419417382415922329608193308558838907629251480102539062500, 0.0814360876644607406538156624264956917613744735718)
				(0.02209708691207961164804096654279419453814625740051269531250, 0.0194032443389855896453966010994918178766965866089)
				(0.01104854345603980582402048327139709726907312870025634765625, 0.0046230842084359824104011771339628467103466391563)
			};
			\addlegendentry{slope 1.035} 
			\addplot[
			color=red,
			only marks,
			mark=o,
			mark size=3pt,
			line width=1.5pt
			] coordinates {
				(0.17677669529663689318432773234235355630517005920410156250000, 1.692641427675535936714368290267884731292724609375)
				(0.08838834764831844659216386617117677815258502960205078125000, 0.442356855449023811388542526401579380035400390625)
				(0.04419417382415922329608193308558838907629251480102539062500, 0.110755093306917018480817205272614955902099609375)
				(0.02209708691207961164804096654279419453814625740051269531250, 0.026468062129200831122943782247602939605712890625)
				(0.01104854345603980582402048327139709726907312870025634765625, 0.005298982139976260441471822559833526611328125000)
			};
			\addlegendentry{FEM Error Gevrey}

			\addplot[
			color=red,
			no marks,
			line width=1.5pt
			] coordinates {
				(0.17677669529663689318432773234235355630517005920410156250000, 1.8175331254411595249820265962625853717327117919922)
				(0.08838834764831844659216386617117677815258502960205078125000, 0.4328154658628736717318474802596028894186019897461)
				(0.04419417382415922329608193308558838907629251480102539062500, 0.1030678477701070933081339831005607265979051589966)
				(0.02209708691207961164804096654279419453814625740051269531250, 0.0245439040002502495418923444958636537194252014160)
				(0.01104854345603980582402048327139709726907312870025634765625, 0.0058447249710420123500198563704088883241638541222)
			};
			\addlegendentry{slope 1.035} 
			
		\end{loglogaxis}
	\end{tikzpicture}
	\caption{FEM error vs Mesh width}\label{fig:fem}
\end{figure}

\providecommand{\bysame}{\leavevmode\hbox to3em{\hrulefill}\thinspace}
\providecommand{\MR}{\relax\ifhmode\unskip\space\fi MR }
\providecommand{\MRhref}[2]{
  \href{http://www.ams.org/mathscinet-getitem?mr=#1}{#2}
}
\providecommand{\href}[2]{#2}

\end{document}